\documentclass[a4paper, reqno,12pt]{amsart}
\usepackage{amsmath, amsthm, amssymb, amstext}

\usepackage[left=1in,right=1in,top=1in,bottom=1in]{geometry}

\usepackage{xurl,hyperref,xcolor}
\usepackage[hyperpageref]{backref}
\usepackage{esint}
\hypersetup{pdfborder={0 0 0},colorlinks}
\allowdisplaybreaks
\usepackage{enumitem}
\usepackage{extarrows}
\allowdisplaybreaks
\newtheorem{theorem}{Theorem}
\newtheorem{remark}[theorem]{Remark}
\newtheorem{lemma}[theorem]{Lemma}
\newtheorem{proposition}[theorem]{Proposition}
\newtheorem{corollary}[theorem]{Corollary}
\newtheorem{definition}[theorem]{Definition}
\newtheorem{example}[theorem]{Example}

\numberwithin{theorem}{section}
\numberwithin{equation}{section}

\title[Best Constants for Anisotropic Moser--Trudinger Inequalities]{Best Constants for Anisotropic Moser--Trudinger Inequalities with the Exact Growth Condition and Their Applications}

\author[G.M. Figueiredo]{Giovany M. Figueiredo}
\address[G.M. Figueiredo]{Universidade de Brasília, Departamento de Matemática, 70910-900 Brasília DF, Brazil}
\email{giovany@unb.br}

\author[D.K. Mahanta]{Deepak Kumar Mahanta}
\address[D.K. Mahanta]{Department of Mathematics, Indian Institute of Technology Kanpur, Uttar Pradesh 208016, India}
\email{mahanta.1@iitj.ac.in}


\begin{document}
\begin{abstract}
We investigate the best constants associated with several types of Moser--Trudinger inequalities within the framework of the Finsler metric in $\mathbb{R}^n$. By employing a convex-symmetrization-free approach, we first establish both the critical and subcritical anisotropic singular Moser--Trudinger inequalities. We further study anisotropic Moser--Trudinger inequalities in both singular and nonsingular cases under the exact growth condition, highlighting the optimality of the best constant as well as the sharp exponent appearing in the denominator. To further enrich the theory, we introduce quasi-conformal-type transformations adapted to the Finsler metric setting to derive several new variants of such inequalities. We also characterize the attainability and nonattainability of maximizers for a generalized version of anisotropic Moser--Trudinger inequalities with the exact growth and explicitly determine the corresponding supremum values. As an application of these inequalities, we analyze the existence and nonexistence of positive ground state solutions to a class of $n$-Finsler-Laplace equations involving constant potential in $\mathbb{R}^n$, where the nonlinearity exhibits a critical exponential growth at infinity.
\end{abstract}
\maketitle
{
\hypersetup{
    linkcolor=black
}
\tableofcontents
}
\section{Introduction}\label{sec1}
The theory of critical Sobolev embeddings is distinguished by the transition from power-type to exponential integrability. This remarkable phenomenon is encapsulated by the Moser--Trudinger inequality, which has become fundamental in the analysis of critical problems. Its classical origins are briefly reviewed below.
\subsection{The Classical Theory of the Moser--Trudinger Inequality} It is well-known that the Moser--Trudinger inequality arises naturally as the limiting counterpart of the classical Sobolev embedding theorem and constitutes one of the most remarkable phenomena in the modern nonlinear analysis. Indeed, let $\Omega\subset \mathbb{R}^n$ with $n\geq2$ be a bounded domain and let $p\in[1,n)$. Then the Sobolev embedding theorem ensures that the embedding $W^{1,p}_0(\Omega)\hookrightarrow L^r (\Omega)$ holds for all $r\in [1,p^\ast]$, where $p^\ast:=\frac{np}{n-p}$ is called the critical Sobolev exponent. It follows that
\begin{align*}
    \sup_{u\in W^{1,p}_0(\Omega),\,\|\nabla u\|_{L^p(\Omega)}\leq 1}\int_{\Omega}|u|^r\,\mathrm{d}x<+\infty\quad\text{for all}\quad r\in [1,p^\ast].
\end{align*}
In contrast, the supremum defined above ceases to be finite for all $r>p^\ast$. This allows us to treat variational problems in $W^{1,p}_0(\Omega)$ involving subcritical and critical power-type nonlinearities, where the maximal admissible growth is of order $p^\ast$. However, the Sobolev borderline case $p=n$ exhibits a completely different behavior. In this case, the embedding  $W^{1,n}_0(\Omega)\hookrightarrow L^r (\Omega)$ holds for all $r\in [1,+\infty)$, which shows that every polynomial growth remains admissible. Evidently, we can notice that $p^\ast\to+\infty$ as $p\to n$ and thus, one may expect that  $W^{1,n}_0(\Omega)\subset L^\infty (\Omega)$ holds. But, it is not true in general. For example, consider the function
\begin{align*}
    u(x):=\begin{cases}
        \ln\left|\ln|x|\right|\quad\text{if}\quad x\in\mathbb{R}^n \quad\text{and}\quad 0<|x|<\frac{1}{\operatorname{\exp}(1)},\\
        0~\qquad\qquad\text{otherwise}.
    \end{cases}
\end{align*}
By direct calculations, one may verify that $u\in W^{1,n}_0(\Omega)$ for any domain $\Omega\subset \mathbb{R}^n$ containing the unit ball centered at the origin, whereas $u\notin L^\infty (\Omega)$. It reveals that the classical Sobolev embedding theorem is insufficient to describe a notion of critical growth in the borderline case. This leads to the question of characterizing the maximal growth enjoyed by functions in 
$W^{1,n}_0(\Omega)$. To overcome this difficulty, it is natural to determine the maximal admissible growth of a function $f\colon\mathbb{R}\to(0,+\infty)$ satisfying
\begin{align*}
    \sup_{u\in W^{1,n}_0(\Omega),\,\|\nabla u\|_{L^n(\Omega)}\leq 1}\int_{\Omega}f(u)\,\mathrm{d}x<+\infty.
\end{align*}
It was shown independently by Yudovich \cite{Yudovich-1961}, Poho\v{z}haev \cite{Pokhozhaev-1965}, and Trudinger \cite{Trudinger-1967} that the maximal admissible growth of such a function $f$ is of exponential type. The statement of Trudinger \cite{Trudinger-1967} was later refined by Moser \cite{Moser-1970/71} and established the precise notion of critical growth for the borderline Sobolev space $W^{1,n}_0(\Omega)$. Through a careful use of Schwarz symmetrization, he obtained
\begin{align}\label{eq10.1}
    \sup_{u\in W^{1,n}_0(\Omega),~\|\nabla u\|_{L^n(\Omega)}~\leq 1}~\int_{\Omega} \operatorname{\exp}\left(\alpha|u|^{\frac{n}{n-1}}\right)\,\mathrm{d}x\leq C(n)|\Omega|\quad\text{for all}\quad\alpha\leq \alpha_n, 
 \end{align}
 where $\alpha_n=n\omega_{n-1}^{\frac{1}{n-1}}$ and $\omega_{n-1}$ is the surface measure of the unit sphere $\mathbb{S}^{n-1}\subset\mathbb{R}^n$. The above supremum fails to be finite once $\alpha>\alpha_n$. These inequalities are nowadays recognized as \emph{Moser--Trudinger type inequalities} in the literature. The failure of \eqref{eq10.1} on domains of infinite measure, despite the condition $\alpha\leq\alpha_n$, gave rise to an extensive study of Moser--Trudinger inequalities in $\mathbb{R}^n$. The first contribution in this direction was due to Cao \cite{Cao-1992}, who proved a weaker form of the Moser--Trudinger inequality in the plane. The higher dimensional analogue was subsequently established by do \'O \cite{do O-1997}. Their combined result asserts that for all $n\geq 2$, $0<\alpha<\alpha_n$ and $u\in W^{1,n}(\mathbb{R}^n)$ with $\|\nabla u\|_n\leq1$, $\|u\|_n\leq M<+\infty$, there exists a constant $C=C(n,M,\alpha)>0$ depending only on $n,M$ and $\alpha$ such that
 $$\int_{\mathbb{R}^n}\Phi\left(\alpha |u|^{\frac{n}{n-1}}\right)\,\mathrm{d}x\leq C(n,M,\alpha),$$
where 
\begin{align}\label{eq10.2}
\Phi(t):=\operatorname{\exp}(t)-\sum_{j=0}^{n-2}\frac{t^j}{j!}\quad\text{for all}\quad t\in\mathbb{R}.    
\end{align}

Building upon Moser's pioneering work, Ruf \cite{Ruf-2005} showed that replacing the Dirichlet norm with the Sobolev full norm in \eqref{eq10.1} leads to a sharp Moser--Trudinger inequality valid on arbitrary domains of $\mathbb{R}^{2}$, including unbounded ones. In particular, the corresponding supremum is bounded uniformly with respect to the domain. His proof was based on classical symmetrization arguments. By employing sophisticated blow-up analysis techniques, Li--Ruf \cite{Li-Ruf-2008} later extended this result to higher dimensions. They proved that for every domain $\Omega\subset\mathbb{R}^{n}$ with $n\geq 2$, there exists a constant $C>0$ such that
 \begin{align}\label{eq10.3}
    \sup_{u\in W^{1,n}_0(\Omega),\,\int_\Omega(|\nabla u|^n+|u|^n)\,\mathrm{d}x\leq1} \int_{\Omega}\Phi\left(\alpha_n |u|^\frac{n}{n-1}\right)\,\mathrm{d}x\leq C.
 \end{align}
In addition, the inequality is sharp in the sense that if the constant $\alpha_n$ is replaced by any $\alpha>\alpha_n$, then the supremum above diverges. Their analysis also yielded the existence of maximizers attaining the supremum  defined in \eqref{eq10.3} when the underlying domain is the whole space $\Omega=\mathbb{R}^n$.

Moreover, by interpolation between the borderline Hardy inequality and the Moser--Trudinger inequality \eqref{eq10.1} in the bounded domains, Adimurthi--Sandeep \cite{Adimurthi-Sandeep-2007} obtained the so-called singular Moser--Trudinger inequality. They established that for all $\beta\in[0,n)$ with $n\geq 2$ and $0<\alpha\leq \alpha_{\beta,n}:=\big(1-\frac{\beta}{n}\big)\alpha_n$, there holds
\begin{align}\label{eq10.4}
    \sup_{u\in W^{1,n}_0(\Omega),~\|\nabla u\|_{L^n(\Omega)\leq1}}\int_\Omega\frac{\operatorname{\exp}\left(\alpha|u|^{\frac{n}{n-1}}\right)}{|x|^\beta}\,\mathrm{d}x<+\infty.
\end{align}
The constant $\alpha_{\beta,n}$ is optimal in the sense that for any $\alpha > \alpha_{\beta,n}$, the supremum defined in \eqref{eq10.4} diverges. The corresponding result in $\mathbb{R}^n$ was subsequently obtained by Adimurthi--Yang \cite{Adimurthi-Yang-2010}. In a later work, Lam--Lu \cite{Lam-Lu-2013} revisited the same singular Moser--Trudinger inequality obtained by Adimurthi--Yang \cite{Adimurthi-Yang-2010} and provided an alternative proof through symmetrization-free approach. These inequalities have several applications in conformal geometry and in the geometric analysis of PDEs. In this context, we refer to Beckner \cite{Beckner-1993}, Branson--Fontana--Morpurgo \cite{Branson-Fontana-Morpurgo-2013} and Chang--Yang \cite{Chang-Yang-2003} and the references therein.

By means of the Moser's inequality \eqref{eq10.1}, it is quite natural to ask the following question:
  \[ \textit{Is the embedding $W^{1,n}_0(\Omega)\hookrightarrow L_{\Phi_\alpha}(\Omega)$ compact}\,~?\]      
where $L_{\Phi_\alpha}(\Omega)$ denotes the Orlicz space associated with the Young function $\Phi_\alpha(t)=\operatorname{\exp}(\alpha|t|^{\frac{n}{n-1}})-1$ for some constant $\alpha>0$. Lions \cite{Lions-1985} established  that although this embedding is continuous, it fails to be compact in the  small weak neighborhoods of zero and identified that the best constant $\alpha_n$ can be improved. More precisely, if the sequence $\{u_k\}_k\subset W^{1,n}_0(\Omega)$ satisfies $\|\nabla u_k\|_{L^n(\Omega)}=1$ and $u_k\rightharpoonup u\neq 0$ weakly in $W^{1,n}_0(\Omega)$ as $k\to\infty$, then for all $p\in[1,\theta)$, there holds
\begin{align}\label{eq10.5}
    \sup_{k\in\mathbb{N}}\int_{\Omega} \operatorname{\exp}\left(\alpha_n p|u_k|^{\frac{n}{n-1}}\right)\,\mathrm{d}x<+\infty,\,\quad\text{where}\quad \theta:=\left(1-\|\nabla u \|^n_{L^n(\Omega)}\right)^{-\frac{1}{n-1}}.
\end{align}
This phenomenon is often called \emph{concentration--compactness principle} in the borderline Sobolev space $W^{1,n}_0(\Omega)$ and it provides a refinement of the classical Moser--Trudinger inequality in the case $u_k \rightharpoonup u \neq 0$ weakly in $W^{1,n}_0(\Omega)$ as $k\to\infty$. The analogous inequality in the whole space $\mathbb{R}^n$ was subsequently studied by Zhang--Chen \cite{Zhang-Chen-2018} using a symmetrization-free approach. Inspired by Lions’ result \eqref{eq10.5}, Adimurthi--Druet \cite{Adimurthi-Druet-2005} derived extensions of Moser’s inequality \eqref{eq10.1} in bounded smooth domains $\Omega\subset\mathbb{R}^2$, which enrich the understanding of \eqref{eq10.5} even in the regime $u_k\rightharpoonup  0$ weakly in $W^{1,n}_0(\Omega)$ as $k\to\infty$. They employed delicate blow-up analysis techniques to establish that
\begin{align*}
 \sup_{u\in W^{1,2}_0(\Omega),\,\|\nabla u\|_{L^2(\Omega)}=1}\int_{\Omega} \operatorname{\exp}\left(4\pi u^2\left(1+\alpha\|u\|_{L^2(\Omega)}^2\right)\right)\,\mathrm{d}x<+\infty
\end{align*}
for all $0\leq \alpha<\lambda_1(\Omega)$. In addition, the above inequality is sharp in the sense that if $ \alpha\geq\lambda_1(\Omega)$, then the supremum is infinite, where $\lambda_1(\Omega)$ is the first eigenvalue of the Laplace operator with the Dirichlet boundary condition in $\Omega$. The validity of this result in higher dimensions was later confirmed by Yang \cite{Yang-2006}.

It is classical that for any smooth bounded domain $\Omega\subset\mathbb{R}^n$ and $p\in[1,n)$, there exists an extremal function (maximizer) for the associated variational problem given by
\begin{align*}
    \sup_{u\in W^{1,p}_0(\Omega),\,\|\nabla u\|_{L^p(\Omega)} =1}\int_{\Omega}|u|^r\,\mathrm{d}x
\end{align*}
for all $r\in [1,p^\ast)$. On the other hand, the compactness is lost at the critical exponent $r=p^\ast$ and the above supremum is not achieved by any function in $W^{1,p}_0(\Omega)$ whenever $\Omega\neq\mathbb{R}^n$. It is therefore natural to investigate the analogous problem in the borderline Sobolev space, which leads to the following question:
\[ \textit{Does the Moser--Trudinger functional admit extremals}\]  
\[ \textit{in the limiting Sobolev regime}\,~?\]    
The first attempt to avoid this difficulty was due to the celebrated work of Carleson--Chang \cite{Carleson-Chang}. They showed by employing symmetrization arguments that 
\begin{align*}
    \sup_{u\in W^{1,n}_0(\Omega),~\|\nabla u\|_{L^n(\Omega)}= 1}~\fint_{\Omega} \operatorname{\exp}\left(\alpha|u|^{\frac{n}{n-1}}\right)\,\mathrm{d}x
 \end{align*}
 for all $\alpha\in(0,\alpha_n]$ admits an extremal functions if $\Omega$ is the $n$-dimensional unit disk. Later, the study of maximizers for Moser--Trudinger inequalities in the whole space setting has attracted considerable attention; see, among others, Ruf \cite{Ruf-2005}, Li--Ruf \cite{Li-Ruf-2008}, and the references therein. The question of nonexistence of maximizers for Moser--Trudinger inequalities has also been addressed by Chen--Lu--Zhu \cite{Chen-Lu-Zhu-2020@}, Lam--Lu--Zhang \cite{Lam-Lu-Zhang-2019@} and the papers cited therein.

The literature on variational problems driven by the $n$-Laplace operator with critical exponential growth in the sense Moser--Trudinger inequality is very vast. In this context, we refer the reader to Adimurthi--Yang \cite{Adimurthi-Yang-2010}, do \'{O}--Medeiros--Severo \cite{do O-Medeiros-Severo-2009}, Figueiredo--Miyagaki--Ruf \cite{Figueiredo-Miyagaki-Ruf-1995}, Yang \cite{Yang-2012} and  Zhang--Chen \cite{Zhang-Chen-2018} for some important contributions and the further references.

The above discussion suggests that the failure of Moser's inequality \eqref{eq10.1} in the whole space $\mathbb{R}^n$ can be compensated either by weakening the exponent $\alpha=\alpha_n$ or by strengthening the norm constraint $\|\nabla u\|_n$. This gives rise to the following question: 
  \[ \textit{What happens if both $\alpha=\alpha_n$ and $\|\nabla u\|_n\leq1$ are retained}\,~?\]         
This question was first addressed by Ibrahim--Masmoudi--Nakanishi \cite{Ibrahim-Masmoudi-Nakanishi-2015} in the plane by weakening the exponential growth, while the higher dimensional counterpart was later established by Masmoudi--Sani \cite{Masmoudi-Sani-2015}. They proved that for all $n\geq 2$, there exists a constant $C=C(n)>0$ depending only on $n$ such that
$$\int_{\mathbb{R}^n}\frac{\Phi(\alpha_n |u|^{\frac{n}{n-1}})}{(1+|u|)^{\frac{n}{n-1}}}\,\mathrm{d}x\leq C\|u\|_n^n\quad\text{for all}\quad u\in W^{1,n}(\mathbb{R}^n)$$
satisfying the condition $\|\nabla u\|_n\leq1$. In addition, the above inequality fails to hold if the power $\frac{n}{n-1}$ in the denominator is replaced by any $p<\frac{n}{n-1}$. Nowadays, such types of inequalities are widely known as \emph{Moser--Trudinger inequalities with exact growth condition} and is regarded as a stronger form of both critical and subcritical Moser--Trudinger inequalities. In this direction, we also invite the reader to see Lam--Lu--Zhang \cite{Lam-Lu-Zhang-2019}, where several improved versions of singular Moser--Trudinger type inequalities with exact growth conditions were established in a newly constructed Sobolev-type space. The exact growth condition has also found significant applications in variational methods. Indeed, by means of the associated Moser--Trudinger ratio, some authors have studied the existence and multiplicity of solutions to related variational problems. For further details, we refer to Masmoudi--Sani \cite{Masmoudi-Sani-2015}, Chen--Qin--R\u{a}dulescu--Tang \cite{Chen-Qin-Radulescu-Tang-2025}, Chen--Qin--R\u{a}dulescu--Tang--Wen \cite{Chen-Qin-Radulescu-Tang-Wen-2024} and the references therein.

The higher order extension of the Moser--Trudinger inequality was first established by Adams \cite{Adams-1988}. In recognition of his pioneering contribution, inequalities of this type are now known as \emph{Adams inequalities}. The theory of Adams inequalities with exact growth conditions has also received considerable attention in recent years. For a comprehensive treatment on this topic, we refer to see  Masmoudi--Sani \cite{Masmoudi-Sani-2014}, Morpurgo--Qin \cite{Morpurgo-Qin-2024}, Qin \cite{Qin-2022}, and the references therein. Motivated by Adams inequalities established by Lam--Lu \cite{Lam-Lu-2013}, Chen--Lu--Zhu \cite{Chen-Lu-Zhu-2020, Chen-Lu-Zhu-2023} introduced the corresponding Adams ratio and applied it to the study of existence and nonexistence of ground state solutions for variational problems driven by the bi-Laplace operator.
\subsection{The Development of the Anisotropic Moser--Trudinger Inequality} In this subsection, we shift our focus from the isotropic framework to its anisotropic counterpart. Let $F\colon\mathbb{R}^n\to \mathbb{R}$ be a nonnegative convex function of class $C^2(\mathbb{R}^n\setminus\{0\})$, which is even as well as positively homogeneous function of degree one and its polar $F^0$ denotes the Finsler metric in $\mathbb{R}^n$. For further details, we refer the reader to Section \ref{sec2}. Let $\Omega\subset\mathbb{R}^n$ with $n\geq2$ be any smooth bounded domain. Then the anisotropic Sobolev space $W^{1,n}_0(\Omega)$ is defined as the completion of $C^\infty_0(\Omega)$ with respect to the anisotropic Dirichlet norm given by
$$ \|F(\nabla u)\|_{L^n(\Omega)}:=\bigg(\int_\Omega F(\nabla u)^n\,\mathrm{d}x\bigg)^{\frac{1}{n}}.$$
Similarly, the anisotropic Sobolev space $W^{1,n}(\mathbb{R}^n)$ is defined as the completion of $C^\infty_0(\mathbb{R}^n)$ equipped with the norm
$$ \|u\|:=\bigg(\int_{\mathbb{R}^n}\Big( F(\nabla u)^n+|u|^n\Big)\,\mathrm{d}x\bigg)^{\frac{1}{n}} \quad\text{for all}\quad u\in W^{1,n}(\mathbb{R}^n).$$
In fact, for all $\zeta>0$, we define another norm 
$$ \|u\|_\zeta:=\bigg(\int_{\mathbb{R}^n}\Big( F(\nabla u)^n+\zeta|u|^n\Big)\,\mathrm{d}x\bigg)^{\frac{1}{n}} \quad\text{for all}\quad u\in W^{1,n}(\mathbb{R}^n),$$
which is equivalent to the standard norm of $W^{1,n}(\mathbb{R}^n)$. It is well-known that the space $W^{1,n}(\mathbb{R}^n)$ is a uniform convex and separable Banach space. By using Lemma \ref{lem2.3}, the space $W^{1,n}(\mathbb{R}^n)$ is continuously embedded into $L^q(\mathbb{R}^n)$ for all $q\in[n,+\infty)$. Hence, the best Sobolev constant is defined as
\begin{align*}
    S_{q}:=\inf_{u\in W^{1,n}(\mathbb{R}^n)\setminus\{0\} }\frac{\|u\|_\zeta}{\|u\|_q}\quad\text{for all}\quad q\in [n,+\infty),
\end{align*}
where $\|\cdot\|_q$ denotes the usual norm of the Lebesgue space $L^{q}(\mathbb{R}^n)$. Moreover, for all $p\in[1,+\infty)$ and $\beta\in [0,n)$, we define the anisotropic weighted Lebesgue space $L^p(\mathbb{R}^n,F^0(x)^{-\beta} \mathrm{d}x)$, consisting of real-valued measurable function $u$ satisfying $|u|^p F^0(x)^{-\beta}\in L^1(\mathbb{R}^n)$, equipped with its natural norm
$$\|u\|_{p,\beta}:= \left(\int_{\mathbb{R}^n} |u|^p{F^0(x)^{-\beta}}\,\mathrm{d}x\right)^{\frac{1}{p}}.$$

Initially, the anisotropic analogue of Moser's inequality \eqref{eq10.1} was appeared in the existing literature by the seminal work of Wang--Xia \cite{Wang-Xia-2012}. They showed that if $\Omega\subset\mathbb{R}^n$ with $n\geq2$ be any smooth bounded domain, then there exists a constant $C>0$ depending only on $n$ such that 
\begin{align}\label{eq10.6}
    \sup_{u\in W^{1,n}_0(\Omega),~\int_\Omega F(\nabla u)^n\,\mathrm{d}x~\leq 1}~\int_{\Omega} \operatorname{\exp}\left(\lambda|u|^{\frac{n}{n-1}}\right)\,\mathrm{d}x\leq C|\Omega|\quad\text{for all}\quad\lambda\leq \lambda_n, 
 \end{align}
 where $\lambda_n=n^{\frac{n}{n-1}}\kappa_n^{\frac{1}{n-1}}$ and $\kappa_n$ is the measure of the unit Wulff ball in $\mathbb{R}^n$. The supremum defined in \eqref{eq10.6} ceases to be finite for all $\alpha>\alpha_n$. In addition, by exploiting this inequality and using delicate blow-up techniques, they derived existence results for Liouville-type and mean field-type equations driven by the Finsler-Laplace operator in dimension two. The extremal functions associated with \eqref{eq10.6} were further investigated by Zhou--Zhou \cite{Zhou-Zhou-2019} by means of blow-up analysis techniques. In particular, they proved that the supremum defined in \eqref{eq10.6} can be achieved by some function $u_0\in W^{1,n}_0(\Omega)\cap C^0(\overline{\Omega})$ satisfying $\int_\Omega F(\nabla u_0)^n\,\mathrm{d}x= 1$. We note that Zhou--Zhou \cite{Zhou-Zhou-2018} also studied the existence of maximizers for the anisotropic Moser--Trudinger supremum \eqref{eq10.6} in the special case when $\Omega$ is a Wulff ball of radius $R$ and centered at the origin in the plane. We would like to point out that the extremal functions for the Moser--Trudinger inequality with $L^p$-norm perturbation involving the $n$-Finsler-Laplace operator in any smooth bounded domain $\Omega\subset\mathbb{R}^n$ has been established by Liu \cite{Liu-2024}. By adapting the ideas of Lions \cite{Lions-1985}, Zhou--Zhou \cite{Zhou-Zhou-2019} also provided an isotropic counterpart of the concentration-compactness principle in the bounded domain $\Omega\subset\mathbb{R}^n$. Specifically, they showed if $\{u_k\}_k\subset W^{1,n}_0(\Omega)$ is a sequence such that $\|F(\nabla u_k)\|_{L^n(\Omega)}=1$ and $u_k\rightharpoonup u\neq 0$ weakly in $W^{1,n}_0(\Omega)$ as $k\to\infty$, then  for all $p$ satisfying $$ 0<p<\mathcal{P}:= \left(1-\|F(\nabla u)\|^n_{L^n(\Omega)}\right)^{-\frac{1}{n-1}}$$
 where $\mathcal{P}=+\infty$ if $\|F(\nabla u)\|_{L^n(\Omega)}=1$, it holds that
\begin{align}\label{eq10.7}
\sup_{k\in\mathbb{N}}\int_{\Omega} \operatorname{\exp}\left(\lambda_n p|u_k|^{\frac{n}{n-1}}\right)\,\mathrm{d}x<+\infty.
\end{align}
Notably, the constant $\mathcal{P}$ is sharp in the sense that the supremum \eqref{eq10.7} is infinite for all $p>\mathcal{P}$. Later, the inequality of Adimurthi--Druet \cite{Adimurthi-Druet-2005} was further revisited by Zhou \cite{Zhou-2020}, who established a refinement of the inequality \eqref{eq10.7}. More explicitly, he obtained
\begin{align*}
 \sup_{u\in W^{1,n}_0(\Omega),\,\|F(\nabla u)\|_{L^n(\Omega)}=1}\int_{\Omega} \operatorname{\exp}\left(\alpha_n \left(1+\alpha\|u\|_{L^n(\Omega)}^n\right)^{\frac{1}{n-1}}|u|^{\frac{n}{n-1}}\right)\,\mathrm{d}x<+\infty
\end{align*}
for all $0\leq \alpha<\lambda_1(\Omega)$ and the above supremum is infinite for all $ \alpha\geq\lambda_1(\Omega)$,  where $\lambda_1(\Omega)$ is the first eigenvalue associated with the $n$-Finsler-Laplace operator in the Dirichlet boundary condition in $\Omega$. The anisotropic version of inequlity \eqref{eq10.3} was shown by Zhou--Zhou \cite{Zhou-Zhou-2020}. In the same work, they studied the existence of extremal functions for the anisotropic Moser--Trudinger inequality in $\mathbb{R}^n$ and, by applying convex-symmetrization methods, extended the inequality \eqref{eq10.7} to the entire space $\mathbb{R}^n$.

In a subsequent contribution, Zhu \cite{Zhu-2020} generalized the singular Moser--Trudinger inequalities of Adimurthi--Sandeep \cite{Adimurthi-Sandeep-2007} and Adimurthi--Yang \cite{Adimurthi-Yang-2010} to the anisotropic framework. By employing convex-symmetrization, he established corresponding inequalities in bounded domains and in the whole space $\mathbb{R}^n$. He also addressed the attainability of the corresponding extremal functions. In particular, he showed that for all $\beta\in[0,2)$ and $0<\lambda\leq \big(1-\frac{\beta}{2}\big)\lambda_2$, then 
\begin{align}\label{eq10.8}
 \sup_{u\in W^{1,2}_0(\mathcal{W}_R),~\|F(\nabla u)\|_{L^2(\mathcal{W}_R)\leq1}}\int_{\mathcal{W}_R}\frac{\operatorname{\exp}(\alpha u^2)}{F^0(x)^\beta}\,\mathrm{d}x   
\end{align}
is achieved, where $\mathcal{W}_R$ is a Wulff ball of radius $R$ and centred at the origin in the plane. Moreover, he also proved that for all $\beta\in[0,n)$ and $0<\lambda\leq \big(1-\frac{\beta}{n}\big)\lambda_n$, then the supremum defined below
\begin{align*}
     \sup_{u\in W^{1,n}(\mathbb{R}^n),~\|u\|_\zeta\leq 1}\int_{\mathbb{R}^n}\frac{\Phi\big(\lambda|u|^{\frac{n}{n-1}}\big)}{F^0(x)^\beta}\,\mathrm{d}x
\end{align*}
 admits an extremal function. Later, following the ideas of do \'O \cite{do O-1997}, Das--Tiwari \cite{Das-Tiwari-2026} again proved a subcritical anisotropic Moser--Trudinger inequality in the entire space $\mathbb{R}^n$. 

Recently, Lu et al. \cite{Lu-Shen-Xue-Zhu-2024} obtained a class of anisotropic isoperimetric inequalities in $\mathbb{R}^n$. They also succeeded in extending the existence of extremals for the supremum in \eqref{eq10.8} to arbitrary bounded domains $\Omega\subset\mathbb{R}^n$ containing the origin and established that the associated supremum is achieved by some function $u_0\in W^{1,n}_0(\Omega)\cap C^0(\overline{\Omega})$ satisfying $\int_\Omega F(\nabla u_0)^n\,\mathrm{d}x= 1$. Their  analysis also enables a complete characterization of the solutions to the singular Lioville's equation driven by the $n$-Finsler-Laplace operator as follows
\begin{align*}
    \begin{cases}
        -Q_n(u)=F^0(x)^{-\beta}\operatorname{\exp}(u)\quad\text{in}\quad\mathbb{R}^n\setminus\{0\},\\
        \int_{\mathbb{R}^n}F^0(x)^{-\beta}\operatorname{\exp}(u)\,\mathrm{d}x\leq\bigg(\frac{n}{n-1}\Big(1-\frac{\beta}{n}\Big)\lambda_n\bigg)^{n-1},
    \end{cases}
\end{align*}
where $\beta\in(0,n)$, $u\in W^{1,n}_{\text{loc}}(\mathbb{R}^n)$ with $n\geq2$ satisfying
$\sup_{x\in \mathbb{R}^n} u(x)\leq u(0)$ and $Q_n (\cdot)=\textnormal{div}\big(F(\nabla\cdot)^{n-1}F_\xi(\nabla\cdot)\big)$ is called the $n$-Finsler-Laplace operator.

The first significant contribution to the anisotropic Moser--Trudinger inequality with exact growth was made by Liu \cite{Liu-2022}. By extending the work of Masmoudi--Sani \cite{Masmoudi-Sani-2015}, the author established that for all $n\geq2$ and $b>0$, there exists a constant $C>0$ depending only on $n$ such that
$$\int_{\mathbb{R}^n}\frac{\Phi(\lambda _n|u|^{\frac{n}{n-1}})}{1+b|u|^{\frac{n}{n-1}}}\,\mathrm{d}x\leq C\|u\|_n^n\quad\text{for all}\quad u\in W^{1,n}(\mathbb{R}^n)$$
satisfying the condition $\int_{\mathbb{R}^n} F(\nabla u)^n\,\mathrm{d}x\leq1$. Moreover, the exponent $\frac{n}{n-1}$ in the denominator is optimal, since the above inequality fails for every $p<\frac{n}{n-1}$. For related developments, we refer to Guo--Liu \cite{Guo-Liu-2024}, where various anisotropic singular Moser--Trudinger inequalities with exact growth condition were established in Sobolev-type spaces.

The study of variational problems associated with the $p$-Finsler-Laplace operator has attracted considerable attention in recent years. In this context, we refer to Ciraolo--Figalli--Roncoroni \cite{Ciraolo-Figalli-Roncoroni-2020}, where the authors characterized the solutions to a critical $p$-Finsler-Laplace equation induced by a smooth norm inside the convex cone. The recent paper by Biagi--Esposito--Roncoroni \cite{Biagi-Esposito-Roncoroni-2025} established the existence of positive solutions to a Brezis–Nirenberg problem driven by the $p$-Finsler-Laplace operator with both $p$-linear and $p$-superlinear perturbations. The literature on variational problems involving the $n$-Finsler-Laplace operator includes, among others, the works of Ciraolo--Li \cite{Ciraolo-Li-2022, Ciraolo-Li-2024}, Das--Tiwari \cite{Das-Tiwari-2026}, and Xie--Gong \cite{Xie-Gong-2016}; see also the references cited therein. In addition, we also mention the contributions of Cianchi--Salani \cite{Cianchi-Salani-2009, Cianchi-Salani-2023}, Ciraolo--Li \cite{Ciraolo-Li-2022}, and Wang--Xia \cite{Wang-Xia-2011}, which provide valuable insights into overdetermined anisotropic elliptic problems.
\subsection{Novel Aspects and Principal Contributions of the Present Work} This article contributes to the theory of anisotropic Moser--Trudinger inequalities by introducing new methods, establishing several sharp inequalities with the exact growth condition, determining the corresponding best constants, and exploring their applications to extremal problems and nonlinear partial differential equations. Our contributions enrich the existing literature in several significant directions as follows:
\begin{itemize}
    \item[$\triangleright$] The critical and subcritical anisotropic singular Moser--Trudinger inequalities are revisited through a completely new proof based on a convex-symmetrization-free approach. This provides an alternative methodology to the existing literature, where convex-symmetrization and rearrangement techniques play a key role.
    \item[$\triangleright$] The exact-growth Moser--Trudinger inequality of Masmoudi--Sani \cite{Masmoudi-Sani-2015} is extended to the anisotropic setting, with the best constant identified and the exponent of the denominator shown to be optimal. As a consequence, an anisotropic Adachi--Tanaka type inequality is obtained, and the anisotropic exact-growth Moser--Trudinger inequality is proved equivalent to the anisotropic critical Moser--Trudinger inequality.
    \item[$\triangleright$] By introducing anisotropic quasi-conformal-type transformations, we build a bridge between anisotropic nonsingular and singular Moser--Trudinger inequalities with the exact growth conditions. This enables us to derive sharp anisotropic singular Moser--Trudinger inequalities with the exact growth condition directly from their nonsingular counterparts. Such a connection appears to be new in the anisotropic setting and provides a systematic mechanism for generating corresponding singular inequalities from nonsingular ones.
    \item[$\triangleright$] Employing anisotropic quasi-conformal-type transformations with the Caffarelli--Kohn--Nirenberg inequalities, we derive a broad family of weighted singular Moser--Trudinger inequalities with the exact growth condition. This significantly enlarges the known class of anisotropic singular Moser--Trudinger inequalities and provides a systematic approach for determining the corresponding best constants.
    \item[$\triangleright$] The scope of anisotropic singular Moser--Trudinger inequalities with the exact growth condition is enlarged to encompass a broad family of previously known inequalities. In this generalized setting, the corresponding sharp variational problem is investigated, leading to existence and nonexistence results for extremal functions under suitable hypotheses.
    \item[$\triangleright$] The applicability of the developed analysis is demonstrated through a class of $n$-Finsler-Laplace equations with constant potentials. The sharp inequalities obtained in this article permit a precise characterization of threshold regimes governing the existence and nonexistence of ground-state solutions, thereby illustrating the effectiveness of exact-growth inequalities in nonlinear anisotropic PDEs.
\end{itemize}

Throughout this article, unless otherwise stated, we adopt the following notations:
\begin{itemize}
\item[$\diamond$] $(X^\ast,\|\cdot\|_{*})$ denotes the continuous dual of Banach space$(X,\|\cdot\|)$, 
\item[$\diamond$] $\langle \cdot, \cdot \rangle$ stands for the duality order pair between $X^\ast$ and $X$,
\item[$\diamond$] $|\Omega|$ represents the $n$-dimensional Lebesgue measure of $\Omega\subset\mathbb{R}^n$,
\item[$\diamond$] $\chi_A$ denotes the characteristic function associated with set $A\subset\mathbb{R}^n$.
\item[$\diamond$]  $u^+=$ max~$\{u,0\}$ and $u^-=$ max~$\{- u,0\}$,
\item[$\diamond$] $o_k(1)$ denotes a real sequence that converges to zero as $k\to\infty$, 
\item[$\diamond$] $\rightharpoonup$ and $\to$ indicate weak and strong convergence,
\item[$\diamond$] $\mathcal{W}_r(y)$ denotes the Wulff ball centered at $y\in \mathbb{R}^n$ with radius $r>0$ and $\mathcal{W}_r=\mathcal{W}_r(0)$,
\item[$\diamond$] $f\lesssim ~( \gtrsim)~ g$ means that $f\leq ( \geq)~C g$ for some suitable constant $C>0$. We write $f\sim g$ whenever both $f\lesssim g$ and $f\gtrsim g$ hold,
\item[$\diamond$] $O(f)$ denotes a quantity bounded by a constant multiple of $f$,
\end{itemize}
\subsection{{Main Results of This Article}} In this subsection, we present main results of this article. We begin with the critical and subcritical anisotropic Moser--Trudinger inequalities associated with singular weights in $\mathbb{R}^n$. The proof strategy is based on a convex-symmetrization-free approach.
\begin{theorem}[\textbf{The Critical Anisotropic Singular Moser--Trudinger Inequality}] \label{thm1.1}
Assume that $n\geq2$ and $0\leq \beta<n$. Then for all $0<\lambda\leq \lambda_{\beta,n}:=\big(1-\frac{\beta}{n}\big)\lambda_n$ and $u\in W^{1,n}(\mathbb{R}^n)$, there holds
\begin{align}\label{eq1.12}
    \sup_{u\in W^{1,n}(\mathbb{R}^n),~\|u\|\leq 1}\int_{\mathbb{R}^n}\frac{\Phi\big(\lambda|u|^{\frac{n}{n-1}}\big)}{F^0(x)^\beta}\,\mathrm{d}x<+\infty,
\end{align}
where $\Phi(\cdot)$ is defined in \eqref{eq10.2} and $\lambda_n=n^{\frac{n}{n-1}}\kappa_n^{\frac{1}{n-1}}$ with $\kappa_n$ stands for the  measure of the unit Wulff ball in $\mathbb{R}^n$. In addition, the constant $\lambda_{\beta,n}$ is sharp, that is, if $\lambda>\lambda_{\beta,n}$, then the supremum \eqref{eq1.12} is infinite.
\end{theorem}
\begin{theorem}[\textbf{The Subcritical Anisotropic Singular Moser--Trudinger Inequality}] \label{thm1.2}
Assume that $n\geq2$ and $0\leq \beta<n$. Then for all $0<\lambda<\lambda_{\beta,n}$ and $u\in W^{1,n}(\mathbb{R}^n)$, there holds
\begin{align}\label{eq1.13}
\sup_{u\in W^{1,n}(\mathbb{R}^n),~\|F(\nabla u)\|_n\leq 1}\frac{1}{\|u\|_n^{n\big(1-\frac{\beta}{n}\big)}}\int_{\mathbb{R}^n}\frac{\operatorname{exp}\big(\lambda|u|^{\frac{n}{n-1}}\big)|u|^n}{F^0(x)^\beta}\,\mathrm{d}x<+\infty.
\end{align}
The constant $\lambda_{\beta,n}$ is sharp in the sense that the supremum \eqref{eq1.13} is infinite for $\lambda\geq \lambda_{\beta,n}$.
\end{theorem}
The result below provides an anisotropic version of the Moser--Trudinger inequality with the exact growth condition in $\mathbb{R}^n$. Its statement is as follows.
\begin{theorem}[\textbf{The Anisotropic Moser--Trudinger Inequality with Exact Growth}]\label{thm1.3}
For all $n\geq2$, $0<\lambda\leq\lambda_n$ and $u\in W^{1,n}(\mathbb{R}^n)$, there holds
\begin{align}\label{eq1.14}
\sup_{u\in W^{1,n}(\mathbb{R}^n),~\|F(\nabla u)\|_n\leq 1}\frac{1}{\|u\|_n^n}\int_{\mathbb{R}^n}\frac{\Phi\big(\lambda|u|^{\frac{n}{n-1}}\big)}{(1+|u|)^{\frac{n}{n-1}}}\,\mathrm{d}x<+\infty.
\end{align}
Moreover, the exponent $\frac{n}{n-1}$ and the constant $\lambda_n$ are sharp in the following context:
\begin{itemize}
\item [\textnormal{(a)}] If the exponent $\frac{n}{n-1}$ in the denominator is replaced by any $q<\frac{n}{n-1}$, then there exists a sequence $\{v_k\}_k\subset W^{1,n}(\mathbb{R}^n)$ satisfying $\|F(\nabla v_k)\|_n\leq 1$ such that
$$\frac{1}{\|v_k\|_n^n}\int_{\mathbb{R}^n}\frac{\Phi\big(\lambda_n|v_k|^{\frac{n}{n-1}}\big)}{(1+|v_k|)^q}\,\mathrm{d}x\to+\infty\quad\text{as}\quad k\to\infty.$$
 \item [\textnormal{(b)}] If $\lambda>\lambda_n$, then there exists a sequence $\{v_k\}_k\subset W^{1,n}(\mathbb{R}^n)$ satisfying $\|F(\nabla v_k)\|_n\leq 1$ such that
$$\frac{1}{\|v_k\|_n^n}\int_{\mathbb{R}^n}\frac{\Phi\big(\lambda|v_k|^{\frac{n}{n-1}}\big)}{(1+|v_k|)^q}\,\mathrm{d}x\to+\infty\quad\text{as}\quad k\to\infty\quad\text{for all}\quad  q\geq 0.$$ 
\end{itemize}
\end{theorem}
The following result reveals the relationship between the anisotropic Moser--Trudinger inequality and its exact growth counterpart.
\begin{theorem}\label{thm1.4}
  The anisotropic Moser--Trudinger inequality is a direct consequence of the anisotropic Moser--Trudinger inequality with the exact growth condition.
\end{theorem}
We next state the anisotropic singular Moser--Trudinger inequality with the exact growth condition in $\mathbb{R}^n$. It can be read as follows.
\begin{theorem}[\textbf{The Anisotropic Singular Moser--Trudinger Inequality with Exact Growth-I}]\label{thm1.5}
Assume that $n\geq2$, $0\leq \beta<n$, $0<\lambda\leq\lambda_{\beta,n}$ and $u\in W^{1,n}(\mathbb{R}^n)$. Then there holds
\begin{align}\label{eq1.15}
\sup_{u\in W^{1,n}(\mathbb{R}^n),~\|F(\nabla u)\|_n\leq 1} \bigg(\int_{\mathbb{R}^n}\frac{|u|^n}{F^0(x)^\beta}\,\mathrm{d}x\bigg)^{-1}\int_{\mathbb{R}^n}\frac{\Phi\big(\lambda|u|^{\frac{n}{n-1}}\big)}{(1+|u|)^{\frac{n}{n-1}}F^0(x)^\beta}\,\mathrm{d}x<+\infty.
\end{align}
In addition, the sharpness of the exponent $\frac{n}{n-1}$ and the constant $\lambda_{\beta,n}$ can be characterized as follows:
\begin{itemize}
\item [\textnormal{(a)}] If the exponent $\frac{n}{n-1}$ in the denominator is replaced by any $q<\frac{n}{n-1}$, then there exists a sequence $\{v_k\}_k\subset W^{1,n}(\mathbb{R}^n)$ satisfying $\|F(\nabla v_k)\|_n\leq 1$ such that
$$\bigg(\int_{\mathbb{R}^n}\frac{|v_k|^n}{F^0(x)^\beta}\,\mathrm{d}x\bigg)^{-1}\int_{\mathbb{R}^n}\frac{\Phi\big(\lambda_{\beta,n}|v_k|^{\frac{n}{n-1}}\big)}{(1+|v_k|)^{q}F^0(x)^\beta}\,\mathrm{d}x\to+\infty\quad\text{as}\quad k\to\infty.$$
 \item [\textnormal{(b)}] If $\lambda>\lambda_{\beta,n}$, then there exists a sequence $\{v_k\}_k\subset W^{1,n}(\mathbb{R}^n)$ satisfying $\|F(\nabla v_k)\|_n\leq 1$ such that
$$\bigg(\int_{\mathbb{R}^n}\frac{|v_k|^n}{F^0(x)^\beta}\,\mathrm{d}x\bigg)^{-1}\int_{\mathbb{R}^n}\frac{\Phi\big(\lambda|v_k|^{\frac{n}{n-1}}\big)}{(1+|v_k|)^{q}F^0(x)^\beta}\,\mathrm{d}x\to+\infty\quad\text{as}\quad k\to\infty$$ 
for all $q\geq0$.
\end{itemize}
\end{theorem}
The above theorem contains the anisotropic singular Moser--Trudinger inequality as a special case. For instance, we have the following result.
\begin{theorem}\label{thm1.6}
  The anisotropic singular Moser--Trudinger inequality follows directly from the anisotropic singular Moser--Trudinger inequality with the exact growth condition.
\end{theorem}
The next two theorems provide refined anisotropic singular Moser--Trudinger inequalities with the exact growth condition. These inequalities can be viewed as singular anisotropic counterparts of several exact-growth results in the isotropic setting. 
\begin{theorem}[\textbf{The Anisotropic Singular Moser--Trudinger Inequality with Exact Growth-II}]\label{thm1.7}
Suppose that $n\geq2$, $0\leq \beta<n$, $0<\lambda\leq\lambda_{n}$ and $u\in W^{1,n}(\mathbb{R}^n)$. Then there holds
\begin{align}\label{eq1.16}
\sup_{\substack{u\in W^{1,n}(\mathbb{R}^n),\\  \|F(\nabla u)\|_n\leq 1}}\bigg(\int_{\mathbb{R}^n}\frac{|u|^n}{F^0(x)^\beta}\,\mathrm{d}x\bigg)^{-1}\int_{\mathbb{R}^n}\frac{\Phi\big(\lambda\big(1-\frac{\beta}{n}\big)|u|^{\frac{n}{n-1}}\big)}{\Big(1+\big(1-\frac{\beta}{n}\big)|u|^{\frac{n}{n-1}}\Big)F^0(x)^\beta}\,\mathrm{d}x<+\infty.
\end{align}
Moreover, the precise sense in which the exponent $\frac{n}{n-1}$ and the constant $\lambda_{n}$ are sharp is given below:
\begin{itemize}
\item [\textnormal{(a)}] If the exponent $\frac{n}{n-1}$ in the denominator is replaced by any $q<\frac{n}{n-1}$, then there exists a sequence $\{v_k\}_k\subset W^{1,n}(\mathbb{R}^n)$ satisfying $\|F(\nabla v_k)\|_n\leq 1$ such that 
$$\bigg(\int_{\mathbb{R}^n}\frac{|v_k|^n}{F^0(x)^\beta}\,\mathrm{d}x\bigg)^{-1}\int_{\mathbb{R}^n}\frac{\Phi\big(\lambda_n\big(1-\frac{\beta}{n}\big)|v_k|^{\frac{n}{n-1}}\big)}{\big(1+\big(1-\frac{\beta}{n}\big)|v_k|^q\big)F^0(x)^\beta}\,\mathrm{d}x\to+\infty$$
as $k\to\infty$.
 \item [\textnormal{(b)}] If $\lambda>\lambda_{n}$, then there exists a sequence $\{v_k\}_k\subset W^{1,n}(\mathbb{R}^n)$ satisfying $\|F(\nabla v_k)\|_n\leq 1$ such that for all $q\geq 0$, there holds
$$\bigg(\int_{\mathbb{R}^n}\frac{|v_k|^n}{F^0(x)^\beta}\,\mathrm{d}x\bigg)^{-1}\int_{\mathbb{R}^n}\frac{\Phi\big(\lambda\big(1-\frac{\beta}{n}\big)|v_k|^{\frac{n}{n-1}}\big)}{\big(1+\big(1-\frac{\beta}{n}\big)|v_k|^q\big)F^0(x)^\beta}\,\mathrm{d}x\to+\infty$$ 
as $k\to\infty$.
\end{itemize}    
\end{theorem}
\begin{theorem}[\textbf{The Anisotropic Singular Moser--Trudinger Inequality with Exact Growth-III}]\label{thm1.8}
    For all $n\geq2$, $0\leq s\leq \beta<n$, $0<\lambda<\lambda_{n}$ and $u\in W^{1,n}(\mathbb{R}^n)$, there holds
\begin{equation}\label{1.999}
\sup_{\substack{u\in W^{1,n}(\mathbb{R}^n),\\  \|F(\nabla u)\|_n\leq 1}}
\bigg(\int_{\mathbb{R}^n}\frac{|u|^n}{F^0(x)^s}\,\mathrm{d}x\bigg)^{-(\frac{n-\beta}{n-s})}\int_{\mathbb{R}^n}\frac{\Phi(\lambda(1-\frac{\beta}{n})|u|^{\frac{n}{n-1}})}{\big(1+(1-\frac{\beta}{n})^{(\frac{n-\beta}{n-s})}|u|^{\frac{n}{n-1}(\frac{n-\beta}{n-s})}\big)F^0(x)^\beta}\,\mathrm{d}x<+\infty.
\end{equation}
Notably, due to the subcritical nature of \eqref{1.999}, the above inequality is sharp in the joint sense: if $\lambda\geq\lambda_n$ and the exponent $\frac{n}{n-1}$ in the denominator is replaced by any $q<\frac{n}{n-1}$, then the supremum \eqref{1.999} ceases to be finite.
\end{theorem}
The anisotropic singular Moser--Trudinger inequality with exact growth condition in $\mathbb{R}^n$ admits the following generalization. The result may be interpreted as follows.
\begin{theorem}[\textbf{The Anisotropic Singular Moser--Trudinger Inequality with Exact Growth-IV}]\label{thm1.9}
For all $n\geq2$, $q\geq n$, $d>0$, $0\leq \beta<n$, $0<\lambda\leq\lambda_{\beta,n}$ and $u\in W^{1,n}(\mathbb{R}^n)$, there holds
\begin{align}\label{eq1.17}
\sup_{u\in W^{1,n}(\mathbb{R}^n),~\|F(\nabla u)\|_n\leq 1} \frac{1}{\|u\|_n^{n\big(1-\frac{\beta}{n}\big)}}\int_{\mathbb{R}^n}\frac{\Phi\big(\lambda|u|^{\frac{n}{n-1}}\big)}{\Big(1+d|u|^{\frac{q}{n-1}\big(1-\frac{\beta}{n}\big)}\Big)F^0(x)^\beta}\,\mathrm{d}x<+\infty.
\end{align}
Furthermore, the exponent $q$ and the constant $\lambda_{\beta,n}$ are sharp in the following senses:
\begin{itemize}
\item [\textnormal{(a)}] If the exponent $q$ in the denominator satisfies $q<n$, then there exists a sequence $\{v_k\}_k\subset W^{1,n}(\mathbb{R}^n)$ satisfying $\|F(\nabla v_k)\|_n\leq 1$ such that
$$\frac{1}{\|v_k\|_n^{n\big(1-\frac{\beta}{n}\big)}}\int_{\mathbb{R}^n}\frac{\Phi\big(\lambda_{\beta,n}|v_k|^{\frac{n}{n-1}}\big)}{\Big(1+d|v_k|^{\frac{q}{n-1}\big(1-\frac{\beta}{n}\big)}\Big)F^0(x)^\beta}\,\mathrm{d}x\to+\infty\quad\text{as}\quad k\to\infty.$$
 \item [\textnormal{(b)}] If $\lambda>\lambda_{\beta,n}$, then there exists a sequence $\{v_k\}_k\subset W^{1,n}(\mathbb{R}^n)$ satisfying $\|F(\nabla v_k)\|_n\leq 1$ such that for all $q\geq0$, we have
$$\frac{1}{\|v_k\|_n^{n\big(1-\frac{\beta}{n}\big)}}\int_{\mathbb{R}^n}\frac{\Phi\big(\lambda|v_k|^{\frac{n}{n-1}}\big)}{\Big(1+d|v_k|^{\frac{q}{n-1}\big(1-\frac{\beta}{n}\big)}\Big)F^0(x)^\beta}\,\mathrm{d}x\to+\infty\quad\text{as}\quad k\to\infty.$$ 
\end{itemize}
\end{theorem}
Another noteworthy consequence of Theorem \ref{thm1.9} is the complete classification of the attainability and nonattainabilty of the associated supremum for $q=n$ in the subcritical regime. The precise statement is stated below.
\begin{theorem}\label{thm1.10}
Assume that $n\geq2$, $d>0$, $0\leq \beta<n$, and $0<\lambda<\lambda_{\beta,n}$. Denote 
$$\textnormal{ASTME}_{n,d,\beta}(\lambda):=\sup_{\substack{u\in W^{1,n}(\mathbb{R}^n),\\  \|F(\nabla u)\|_n\leq 1}} \frac{1}{\|u\|_n^{n(1-\frac{\beta}{n})}}\int_{\mathbb{R}^n}\frac{\Phi(\lambda|u|^{\frac{n}{n-1}})}{\big(1+d|u|^{\frac{n}{n-1}(1-\frac{\beta}{n})}\big)F^0(x)^\beta}\,\mathrm{d}x.$$
Then the following assertions hold:
\begin{itemize}
    \item[\textnormal{(a)}] If $0<\beta<n$, then  $\textnormal{ASTME}_{n,d,\beta}(\lambda)$ can be attained for all $d>0$ and $0<\lambda<\lambda_{\beta,n}$.
    \item[\textnormal{(b)}] If $\beta=0$, then  $\textnormal{ASTME}_{n,d,0}(\lambda)$ can be attained for all $0<d\leq\frac{\lambda}{n}$ and $0<\lambda<\lambda_{n}$.
    \item[\textnormal{(c)}] If $\beta=0$, then  $\textnormal{ASTME}_{n,d,0}(\lambda)$ is not attained only if
    $$\textnormal{ASTME}_{n,d,0}(\lambda)=\frac{\lambda^{n-1}}{(n-1)!} .$$
\end{itemize}
\end{theorem}

The results established in Theorems \ref{thm1.1} and \ref{thm1.3}, naturally lead us to study the existence and nonexistence of positive ground state solutions for a class of $n$-Finsler-Laplace equations involving constant potentials and critical exponential growth in $\mathbb{R}^n$. More precisely, we deal with the following problem
\begin{equation}\label{main problem}
   -Q_n u+\zeta|u|^{n-2}u=g(u)\quad\text{in}\quad \mathbb{R}^n,\tag{$\mathcal{P}$}
\end{equation}
where $n\geq2$ and $\zeta>0$ is a positive parameter. The operator $Q_n (\cdot)=\textnormal{div}\big(F(\nabla\cdot)^{n-1}F_\xi(\nabla\cdot)\big)$ is known as the $n$-Finsler-Laplace operator, where $F\colon \mathbb{R}^n\to\mathbb{R}$ is a nonnegative, convex, and homogeneous function of degree one and $F_\xi$ likely denotes the gradient of $F$ with respect to its vector argument $\xi$. In addition, the nonlinearity $g\colon \mathbb{R}\to\mathbb{R}$ is a continuous function and has \emph{critical exponential growth at infinity}, that is, it behaves like $\operatorname{exp}(\lambda s^{\frac{n}{n-1}})$ as $s\to\infty$ for some $\lambda>0$, which means that there exists a positive constant $\lambda_0>0$ such that 
\begin{align*}
	\lim_{s\to +\infty} g(s)\operatorname{exp}\big(-\lambda s^{\frac{n}{n-1}}\big)=
	\begin{cases}
		0       & \text{for all } \lambda>\lambda_0, \\
		+\infty & \text{for all }\lambda<\lambda_0. 
	\end{cases}
\end{align*}
This concept of criticality arises from the anisotropic Moser--Trudinger inequality discussed above. To ensure the positivity of ground state solutions, we shall require that $g$ satisfies appropriate nonnegativity conditions. Accordingly, the following assumptions on the nonlinearity $g\colon \mathbb{R}\to\mathbb{R}$ will be imposed throughout this article:
\begin{enumerate}[label=\textnormal{(\bf g\arabic*)},ref=\textnormal{\bf g\arabic*}]
	\item\label{g1} $g\in C^1(\mathbb{R},\mathbb{R})$ such that $g(s)=0$ for all $s\leq 0$, $g(s)>0$ for all $s>0$ and $$g(s)=o(s^{n-1})\quad\text{as}\quad s\to 0.$$ 
\item\label{g2} Critical exponential growth: there exists a positive constant $\lambda_0$ such that  
    \begin{align*}
	\lim_{s\to +\infty} \frac{g(s)}{\operatorname{exp}\big(\lambda s^{\frac{n}{n-1}}\big)}=
	\begin{cases}
		0       & \text{for all } \lambda>\lambda_0, \\
		+\infty & \text{for all }\lambda<\lambda_0. 
	\end{cases}
\end{align*}
\item\label{g3} Ambrosetti-Rabinowitz (A-R) condition: there exists a constant $\mu>n$ such that
      $$0<\mu G(s):=\mu\int_{0}^{s} g(t)\,\mathrm{d}t\leq g(s)s\quad{\text{for all}}\quad s>0. $$
\item\label{g4} There exist positive constants $s_0$ and $M_0$ such that 
       $$0< G(s)\leq M_0g(s)\quad{\text{for all}}\quad s\geq s_0 .$$
\item\label{g5} The maps $s\mapsto\displaystyle\frac{g(s)}{s^{n-1}}$ and $s\mapsto\displaystyle\frac{G(s)}{s^n}$ are continuously differentiable on $(0,+\infty)$ and satisfy
$$
\left(\frac{g(s)}{s^{n-1}}\right)^\prime>0
\quad\text{and}\quad
\left(\frac{G(s)}{s^n}\right)^\prime>0\quad\text{for all}\quad s>0.
$$
\item\label{g6} It holds that $g^\prime(s)s=o(s^{n-1})$ as $s\to 0$, and for any $\lambda>\lambda_0$, there exists $\varepsilon>0$ and a constant $C_\varepsilon>0$ such that
\begin{align*}
     |g^\prime(s)s|\leq \varepsilon |s|^{n-1}+C_\varepsilon|s|^{\mu-1}\Phi(\lambda|s|^{\frac{n}{n-1}})\quad{\text{for all}}\quad s\in\mathbb{R}.
\end{align*}
\end{enumerate}
By using \eqref{g1} and \eqref{g2}, for any $\lambda>\lambda_0$, there exists $\varepsilon>0$ and a constant $D_\varepsilon>0$ possibly depending on $\varepsilon$, $\lambda$ and $\mu$ such that
\begin{align}\label{eq1.1}
    |g(s)|\leq \varepsilon |s|^{n-1}+D_\varepsilon|s|^{\mu-1}\Phi(\lambda|s|^{\frac{n}{n-1}})\quad{\text{for all}}\quad s\in\mathbb{R}.
\end{align}
In addition, combining \eqref{g3} with \eqref{eq1.1}, we obtain 
\begin{align}\label{eq1.2}
    |G(s)|\leq \varepsilon |s|^{n}+D_\varepsilon|s|^{\mu}\Phi(\lambda|s|^{\frac{n}{n-1}})\quad{\text{for all}}\quad s\in\mathbb{R}.
\end{align}

The function stated below serves as a concrete example satisfying all of the preceding hypotheses on the nonlinear term.
\begin{example} The function $G\colon \mathbb{R}\to \mathbb{R}$ given by $G(s):=(s^+)^\mu \operatorname{exp}\big(\lambda_0(s^+)^{\frac{n}{n-1}}\big)$ for all $s\in\mathbb{R}$, where $n\geq 2$, $\lambda_0>0$, $\mu>n$, $s^+=\max\{s,0\}$ and $g(s)=G^\prime(s)$ for all $s\in \mathbb{R}$. Then $g$ satisfies all the hypotheses \eqref{g1}--\eqref{g6}. 
\end{example}
Next, we define the notion of weak solution to \eqref{main problem} as follows.
\begin{definition}
    We say that $u\in W^{1,n}(\mathbb{R}^n)$ is a weak solution to  \eqref{main problem}, if it holds
    $$\int_{\mathbb{R}^n}F(\nabla u)^{n-1}F_\xi(\nabla u)\nabla v\,\mathrm{d}x+\zeta\int_{\mathbb{R}^n}|u|^{n-2}uv\,\mathrm{d}x=\int_{\mathbb{R}^n}g(u)v\,\mathrm{d}x$$ 
    for all $v\in C_0^\infty(\mathbb{R}^n)$.
\end{definition}
The result stated below addresses the existence and nonexistence of positive ground state solutions to \eqref{main problem}.
\begin{theorem} \label{thm1.13}
Let $n\geq2$ and suppose that hypotheses \eqref{g1}--\eqref{g6} are satisfied. Then there exists a threshold $\zeta^\ast\in(0,+\infty]$ such that \eqref{main problem} admits a positive ground state solution for every $\zeta\in(0,\zeta^\ast)$. If $\zeta^\ast<+\infty$, then \eqref{main problem} admits no ground state solution for any $\zeta\in(\zeta^\ast,+\infty)$, where $\zeta^\ast$ is the Moser--Trudinger ratio defined by
$$\zeta^\ast:=\sup\left\{\frac{n}{\|u\|_n^n}\int_{\mathbb{R}^n}G(u)\,\mathrm{d}x\colon\, u\in W^{1,n}(\mathbb{R}^n)\setminus\{0\}\quad\text{and}\quad \|F(\nabla u)\|_n^n\leq\bigg(\frac{\lambda_n}{\lambda_0}\bigg)^{n-1}\right\}.$$
Moreover, the finiteness or infiniteness of the threshold $\zeta^\ast$ can be further characterized under additional asymptotic assumptions on $G$ as follows:
\begin{itemize}
\item[\textnormal{(i)}] It holds that $\zeta^\ast=+\infty$ whenever \begin{align}\label{eq1.18}
 \lim_{s\to+\infty}\frac{s^{\frac{n}{n-1}}G(s)}{\operatorname{\exp}(\lambda_0s^{\frac{n}{n-1}})}
=+\infty.   
\end{align}
\item[\textnormal{(ii)}] It holds that $\zeta^\ast<+\infty$ whenever  \begin{align}\label{eq1.19}
 \limsup_{s\to+\infty}\frac{s^{\frac{n}{n-1}}G(s)}{\operatorname{\exp}(\lambda_0s^{\frac{n}{n-1}})}
<+\infty.
\end{align}
\end{itemize}
\end{theorem}

The remainder of this article is organized as follows. Section \ref{sec2} collects the preliminary results required throughout the paper. Section \ref{sec3} is devoted to the proofs of Theorems \ref{thm1.1} and \ref{thm1.2} by means of a convex-symmetrization-free approach. The proofs of Theorems \ref{thm1.3} and \ref{thm1.4} are carried out in Section \ref{sec4}, whereas Section \ref{sec5} is devoted to Theorems \ref{thm1.5} and \ref{thm1.6}. Section \ref{sec6} establishes Theorems \ref{thm1.7} and \ref{thm1.8} using quasi-conformal-type transformations. Section \ref{sec7} contains the proofs of Theorems \ref{thm1.9} and \ref{thm1.10}. Finally, Section \ref{sec8} is devoted to the proof of Theorem \ref{thm1.13}, which concerns the existence and nonexistence of ground state solutions to our main problem.
\section{Preliminary results}\label{sec2}
The aim of this section is to introduce notation and summarize preliminary properties of the Finsler metric, along with a discussion of convex-symmetrization and several lemmas to be used in the sequel. Let $F\colon\mathbb{R}^n\to \mathbb{R}$ be a nonnegative convex function of class $C^2(\mathbb{R}^n\setminus\{0\})$, which is even as well as positively homogeneous function of degree one in the sense that
\begin{align}\label{eq2.333}
    F(s\xi)=|s|F(\xi)\quad\text{for all}\quad \xi\in \mathbb{R}^n \quad\text{and}\quad s\in \mathbb{R}.
\end{align}
In addition, we assume that $F(\xi)>0$ for all $\xi\in \mathbb{R}^n\setminus\{0\}$. Due to homogeneity of $F$, there exist two positive constants $c_1$ and $c_2$ satisfying $c_1\leq c_2$ such that
\begin{align*}
    c_1|\xi|\leq F(\xi)\leq c_2|\xi| \quad\text{for all}\quad \xi\in \mathbb{R}^n.
\end{align*}
For an example, one may consider the function $F$, which is given by
\begin{align*}
    F(\xi)=\|\xi\|_p=\bigg(\sum_{i=1}^{n}|\xi_i|^p\bigg)^{\frac{1}{p}} \quad\text{for}\quad\xi\in\mathbb{R}^n\quad\text{and}\quad p\in[1,+\infty),
\end{align*}
where $\|\cdot\|_p$ denotes the $p$-norm in $\mathbb{R}^n$. Note that we always impose the condition that the matrix $\textit{Hess}\,(F^2)$ is positive definite in $\mathbb{R}^n\setminus\{0\}$. It then follows from Xie--Gong \cite{Xie-Gong-2016} that the matrix $\textit{Hess}\,(F^n)$ enjoys the same property. Define $K=\{x\in\mathbb{R}^n\colon\,F(x)\leq 1\}$. The associated support function $F^0$ is given by
\begin{align*}
    F^0(x):=\sup_{\xi\in K}\langle{x,\xi}\rangle\quad\text{for all}\quad x\in \mathbb{R}^n,
\end{align*}
where $\langle{\cdot,\cdot}\rangle$ denotes the standard inner product in $\mathbb{R}^n$. It is not difficult to see that $F^0\colon\mathbb{R}^n\to[0,+\infty)$ is also a convex and homogeneous function of class $C^2(\mathbb{R}^n\setminus\{0\})$. Moreover, by Alvino et al. \cite{Alvino-Ferone-Trombetti-Lions-1997}, $F^0$ is the dual of $F$ and one obtains
\begin{align*}
    F^0(x)=\sup_{\xi\in \mathbb{R}^n\setminus\{0\}}\frac{\langle{x,\xi}\rangle}{F(\xi)}\quad\text{and}\quad  F(x)=\sup_{\xi\in \mathbb{R}^n\setminus\{0\}}\frac{\langle{x,\xi}\rangle}{F^0(\xi)}\quad\text{for all}\quad x\in \mathbb{R}^n,
\end{align*}
Define the mapping $\Psi\colon \mathbb{S}^{n-1}\to \mathbb{R}^n$ by
\begin{align*}
  \Psi(\xi):= F_\xi(\xi) \quad\text{for all}\quad \xi\in \mathbb{S}^{n-1}.
\end{align*}
As shown by Wang--Xia \cite{Wang-Xia-2011}, there holds
\begin{align*}
    \Psi(\mathbb{S}^{n-1})=\{x\in\mathbb{R}^n\colon\,F^0(x)= 1\}.
\end{align*}
It is worth mentioning that $\Psi(\mathbb{S}^{n-1})$ is a smooth convex hypersurface in $\mathbb{R}^n$, widely recognized as the Wulff shape or equilibrium crystal shape. The notion of the Wulff shape arises naturally from Wulff’s seminal work \cite{Wulff-1901} on equilibrium crystal shapes associated with the minimization of the anisotropic surface energy functional
\begin{align*}
    \int_{\partial\Omega} F(\overset{\to}{n})\,\mathrm{d}\sigma
\end{align*}
within the class of regular domains $\Omega$, where $\overset{\to}{n}$ is the outer normal vector field to $\partial\Omega$. Later, Fonseca--M\"{u}ller \cite{Fonseca-Muller-1991} and Taylor \cite{Taylor-1978} observed that the Wulff shape $\Psi(\mathbb{S}^{n-1})$ is, up to translations, the unique minimizer of the anisotropic surface energy functional. In particular, if $F(\xi)=|\xi|$, then the corresponding Wulff shape is the standard sphere $\mathbb{S}^{n-1}$ in $\mathbb{R}^{n}$. Thus, the Wulff shape serves as a natural extension of the notion of a sphere. Denote $K^0=\{x\in\mathbb{R}^n\colon\,F^0(x)\leq 1\}$. Let $\kappa_n$ be the measure of $K^0$. It is known that $K^0$ is the unit Wulff ball in $\mathbb{R}^n$. Moreover, it can be found in Lu et al. \cite{Lu-Shen-Xue-Zhu-2024} that $K^0$ and $K$ are polar to each other. The Wulff ball of radius $r$ and center at $x_0$ is defined by
\begin{align*}
    \mathcal{W}_r(x_0):=\{x\in \mathbb{R}^n\colon\, F^0(x-x_0)\leq r\}.
\end{align*}
The measure of $\mathcal{W}_r(x_0)$ is given by $|\mathcal{W}_r(x_0)|=\kappa_n r^n$. Now, we collect some elementary properties of $F$ and $F^0$ from the works of Ferone--Kawohl \cite{Ferone-Kawohl-2009}, Ciraolo--Li \cite{Ciraolo-Li-2022} and Wang--Xia \cite{Wang-Xia-2011}.
\begin{lemma}\label{lem2.90}
The function $F$ and its polar $F^0$ have the following properties:
\begin{itemize}
   \item [\textnormal{(a)}] $|F(x)-F(y)|\leq F(x+y)\leq F(x)+F(y)$ for all $x,y\in\mathbb{R}^n$,
    \item [\textnormal{(b)}] $\frac{1}{C}\leq |\nabla F(x)|\leq C$ and $\frac{1}{C}\leq |\nabla F^0(x)|\leq C$ for all $x\in\mathbb{R}^n\setminus\{0\}$ and some constant $C>0$,
    \item [\textnormal{(c)}] $\langle{x,\nabla F(x)}\rangle=F(x)$ and $\langle{x,\nabla F^0(x)}\rangle=F^0(x)$ for all $x\in\mathbb{R}^n\setminus\{0\}$,
    \item [\textnormal{(d)}] $F(\nabla F^0(x))=1$ and $F^0(\nabla F(x))=1$ for all $x\in\mathbb{R}^n\setminus\{0\}$,
    \item [\textnormal{(e)}] $F^0(x)F_\xi(\nabla F^0(x))=x=F(x)F^0_\xi(\nabla F(x))$ for all $x\in\mathbb{R}^n\setminus\{0\}$,
    \item [\textnormal{(f)}] $F_{\xi}(tx)=\text{sgn}(t)F_{\xi}(\xi)$ for all $\xi\in\mathbb{R}^n\setminus\{0\}$ and $t\neq 0$.
\end{itemize}
\end{lemma}
The celebrated anisotropic co-area formula and isoperimetric inequality can be expressed by
\begin{align*}
    \int_\Omega|\nabla u|_F=\int_{0}^{\infty}P_F(|u|>t)\,\mathrm{d}t
\end{align*}
and
\begin{align*}
    P_F(\Omega)\geq n\kappa_n^{\frac{1}{n}}|\Omega|^{1-\frac{1}{n}},
\end{align*}
where $P_F(\Omega)$ stands for the de Giorgi perimeter of $\Omega$, interpreted as the anisotropic total variation of the characteristic function of $\Omega$. The above isoperimetric inequality reduces to an equality if and only if $\Omega$ is homothetic to $K^0$. In this context, we refer the scholars to Fonseca--M\"{u}ller \cite{Fonseca-Muller-1991}. This inequality, commonly referred to as the Wulff inequality, was first introduced by Wulff \cite{Wulff-1901} without proof. By the application of the ABP method, Cabr\'{e} et al.  \cite{Cabre-Ros-Oton-Serra-2016} derived anisotropic weighted isoperimetric inequalities for weights enjoying positivity, local Lipschitz regularity, and nonnegative homogeneity. For some recent contributions to anisotropic isoperimetric type inequalities, we refer to Bianchi et al. \cite{Bianchi-Cianchi-Gronchi-2025}, Lu et al. \cite{Lu-Shen-Xue-Zhu-2024}, Paoli--Trani \cite{Paoli-Trani-2019}, and the references therein.

In the sequel, we introduce convex-symmetrization and recall its fundamental properties. This notion extends the classical Schwarz symmetrization. For a comprehensive treatment of this topic, we refer to the monograph of Bennett--Sharpley \cite{Bennett-Sharpley-1988} and the works of Bianchi et al. \cite{Bianchi-Cianchi-Gronchi-2025}, Alvino et al. \cite{Alvino-Ferone-Trombetti-Lions-1997}, as well as Esposito--Trombetti \cite{Esposito-Trombetti-2004}. To proceed, we define the convex-symmetrization of a Lebesgue measurable set $S\subset\mathbb{R}^n$. If $|S|=+\infty$, we set $S^\ast=\mathbb{R}^n$, while in the case $0<|S|<+\infty$, we write $S^\ast=\mathcal{W}_R$ with $|S^\ast|=|S|=\kappa_n R^n$, where $R=R(S)$ is the volume radius of $S$. The set $S^\ast$ is called the convex decreasing rearrangement of $S$. Next, let $f\colon S\to\mathbb{R}$ be a measurable function vanishing at infinity, meaning that
\begin{align*}
 |\{x\in S\colon~|f(x)|>t\}|<+\infty\quad\text{for all}\quad t>0.  
\end{align*}
The distribution function associated with $f$ is the mapping $\mu_f\colon[0,+\infty)\to [0,+\infty]$ such that $\mu_f(t)=|\{x\in S\colon~|f(x)|>t\}|$ for all $t\geq 0$. The function $\mu_f$ is nonnegative, decreasing and right-continuous on $[0,+\infty)$. The unidimensional convex rearrangement of $f$ is then defined as the function $f^\sharp\colon[0,|S|]\to [0,+\infty]$ such that $$f^\sharp(0):=\text{ess sup} f\quad\text{and}\quad f^\sharp(s):=\inf\{t\geq 0\colon~\mu_f(t)<s\}\quad\text{for all}\quad s\in(0,|S|].$$ 
It was further shown by Bennett--Sharpley \cite{Bennett-Sharpley-1988} that
$$f^\sharp(s)=m_{\mu_f}(s)=\sup\{t\geq 0:~\mu_f(t)>s\},~\forall~s\geq 0,$$
where $m_{\mu_f}$ is the distribution function of $\mu_f$. In addition, $f^\sharp$ is a nonnegative, decreasing and right-continuous function on $[0,+\infty)$. The convex-symmetrization of $f$, denoted by $f^\ast\colon S^\ast\to [0,+\infty]$, is defined as
$$ f^\ast(x):= f^\sharp(\kappa_n F^0(x)^n)\quad\text{for all}\quad x\in S^\ast.$$
Note that $f^\ast$ is a Wulff radially symmetric, nonnegative and nonincreasing function. The standard properties of convex-symmetrization ensure that if $\Psi\colon\mathbb{R}\to \mathbb{R}$ is a nondecreasing continuous function, then 
\begin{equation}\label{eq2.1}
\int_{\mathbb{R}^n}\Psi(f)\,\mathrm{d}x=\int_{\mathbb{R}^n}\Psi(f^\ast)\,\mathrm{d}x.   
\end{equation}
Moreover, one can also notice that if $\Psi\colon\mathbb{R}\to \mathbb{R}$ is nondecreasing, then there holds
\begin{equation}\label{eq2.2}
 (\Psi(f))^\ast=\Psi(f^\ast).   
\end{equation}
It is easy to check that $f^\ast$ and $|f|$ are equimeasurable, namely, $\{x\colon~ f^\ast>t\}=\{x\colon~ |f|>t\}^\ast$ and $(\chi_A)^\ast=\chi_{A^\ast}$ for any measurable set $A\subset\mathbb{R}^n$. In addition, there holds $\|f\|_p=\|f^\ast\|_p$ for all $p\in[1,+\infty)$.

The anisotropic P\'{o}lya--Szeg\"{o} inequality, established by Alvino et al. \cite[Theorem 3.1]{Alvino-Ferone-Trombetti-Lions-1997}, reads as follows.
\begin{lemma}\label{lem2.91}
  Let $u\in W^{1,p}(\mathbb{R}^n)$ for any $p\in[1,+\infty)$. Then  $u^\ast\in W^{1,p}(\mathbb{R}^n)$ such that
  \begin{align*}
   \|u\|_p=\|u^\ast\|_p\quad\text{and}\quad \|F(\nabla u^\ast)\|_p\leq  \|F(\nabla u)\|_p. 
  \end{align*}
\end{lemma}
The Hardy--Littlewood inequality stated below appears in the monograph of Baernstein \cite[Corollary 2.16]{Baernstein-2019}.
\begin{lemma}\label{lem2.92}
 Let $f$ and $g$ be two nonnegative measurable functions on $\mathbb{R}^n$ with $\mu_f(t)<+\infty$ and $\mu_g(t)<+\infty$ for all $t>0$. Then there holds
 \begin{equation}\label{eq2.3}
  \int_{\mathbb{R}^n} f(x)g(x)\,\mathrm{d}x\leq \int_{\mathbb{R}^n} f^\ast(x) g^\ast(x)\,\mathrm{d}x    
 \end{equation}
 with the understanding that when the left side is infinite, so is the right side. Furthermore, if $\int_{\mathbb{R}^n} f(x)g(x)\,\mathrm{d}x<+\infty$, then the equality in \eqref{eq2.3} holds if and only if $|\{(x,y)\in\mathbb{R}^{2n}\colon~f(x)<f(y)~\text{and}~g(x)>g(y)\}|=0.$
\end{lemma}
In the spirit of the pioneering works of Berestycki--Lions \cite[Lemma A.VI]{Berestycki-Lions-1983} and Lions \cite[Lemma I.1]{Lions-1982}, we have the following anisotropic radial decay estimates.
\begin{lemma}\label{lem2.2} Let $n\geq 2$ and $1\leq p<+\infty$. Then the following assertions hold:
\begin{itemize}
    \item [\textnormal{(a)}] If $u\in L^p(\mathbb{R}^n)$ is a nonnegative, Wulff radially symmetric, and nonincreasing function, then
    $$|u(x)|\leq \bigg(\frac{1}{\kappa_n}\bigg)^{\frac{1}{p}}F^0(x)^{-\frac{n}{p}}\|u\|_p\quad\text{for a.e.}~ x\in\mathbb{R}^n\setminus\{0\}. $$
    \item [\textnormal{(b)}] If $u\in W^{1,p}_{\mathrm{rad}}(\mathbb{R}^n)$ with $1<p<+\infty$, then
    $$|u(x)|\leq \bigg(\frac{p}{n\kappa_n}\bigg)^{\frac{1}{p}}F^0(x)^{-\frac{n-1}{p}}\|u\|_p^{\frac{p-1}{p}}\|F(\nabla u)\|_p^{\frac{1}{p}}\quad\text{for a.e.}~ x\in\mathbb{R}^n\setminus\{0\}. $$
\end{itemize}
\end{lemma}
\begin{proof}
By using the abuse notation of radial symmetry, we write $u(x)=u(\eta)$ for $\eta=F^0(x)$ for a.e. $x\in\mathbb{R}^n$. From the given hypothesis, we obtain for all $R>0$ that
\begin{align*}
    \int_{\mathbb{R}^n}|u(x)|^p\,\mathrm{d}x&\geq \int_{\mathcal{W}_R}|u(x)|^p\,\mathrm{d}x=\int_{0}^{R}|u(\eta)|^p\int_{\partial\mathcal{W}_\eta}\frac{1}{|\nabla F^0(x)|}\,\mathrm{d}\sigma\,\mathrm{d}\eta\\
    &=n\kappa_n\int_{0}^{R}|u(\eta)|^p\eta^{n-1}\,\mathrm{d}\eta\geq\kappa_n R^n|u(R)|^p.
\end{align*}
This yields $$|u(R)|\leq \bigg(\frac{1}{\kappa_n}\bigg)^{\frac{1}{p}}R^{-\frac{n}{p}}\|u\|_p\quad\text{for all}\quad R>0. $$
Now, setting $R=F^0(x)$, we infer that $\textnormal{(a)}$ is verified. By the standard density argument, it suffices to prove $\textnormal{(b)}$ for all $u\in C^\infty_{0,\mathrm{rad}}(\mathbb{R}^n)$. Due to radial symmetry, we have $u\in C_0^\infty(0,+\infty)$. It follows from H\"{o}lder's inequality that for all $\eta>0$, one has
\begin{align*}
    r^{n-1}|u(r)|^p&\leq p\int_{r}^{\infty}|u(\eta)|^{p-1}|u^\prime(\eta)|\eta^{n-1}\,\mathrm{d}\eta\\
    &\leq p\Bigg(\int_{r}^{\infty}|u(\eta)|^p\eta^{n-1}\,\mathrm{d}\eta\Bigg)^{\frac{p-1}{p}}\Bigg(\int_{r}^{\infty}|u^\prime(\eta)|^p\eta^{n-1}\,\mathrm{d}\eta\Bigg)^{\frac{1}{p}}\\
    &\leq\frac{p}{n\kappa_n}\Bigg(\int_{\mathbb{R}^n}|u|^p\,\mathrm{d}x\Bigg)^{\frac{p-1}{p}}\Bigg(\int_{\mathbb{R}^n}F(\nabla u)^p\,\mathrm{d}x\Bigg)^{\frac{1}{p}}.
\end{align*}
The above inequality implies that 
$$|u(r)|\leq \bigg(\frac{p}{n\kappa_n}\bigg)^{\frac{1}{p}}r^{-\frac{n-1}{p}}\|u\|_p^{\frac{p-1}{p}}\|F(\nabla u)\|_p^{\frac{1}{p}}\quad\text{for all}\quad r>0. $$
By choosing $r=F^0(x)$, the proof of $\textnormal{(b)}$ is finished, and Lemma \ref{lem2.2} follows. This completes the proof.
\end{proof}
Next, we state some embedding results that will be utilized in the subsequent sections.
\begin{lemma}\label{lem2.3}
 Let $n\geq 2$, then the embedding $W^{1,n}(\mathbb{R}^n)\hookrightarrow L^q (\mathbb{R}^n)$ is continuous for all  $q\in [n,+\infty)$. Moreover, the embedding $W^{1,n}(\mathbb{R}^n)\hookrightarrow L^q _{\textit{loc}}(\mathbb{R}^n)$ is compact for all $q\in [1,+\infty)$.
\end{lemma}
The proof of the following result is omitted, as the argument proceeds along essentially the same lines as that of Zhang--Chen \cite[Theorem 1.2]{Zhang-Chen-2018}, requiring only minor modifications.
\begin{lemma}\label{lem2.6}
Let $n\geq 2$ and $\beta\in[0,n)$ be hold. Then the embedding $W^{1,n}(\mathbb{R}^n)\hookrightarrow L^q (\mathbb{R}^n, F^0(x)^{-\beta}\mathrm{d}x)$ is continuous for all  $q\in [n,+\infty)$. Moreover, this embedding is compact for all $\beta\in(0,n)$.
\end{lemma}
The lemma stated below can be obtained from Yang \cite[Lemma 2.1 and Lemma 2.2]{Yang-2012}. 
\begin{lemma}\label{lem2.4}
    The function $s\mapsto\operatorname{\exp}(s)-\sum_{j=0}^{n-2}\frac{s^j}{j!}$ is continuous, increasing, and convex in $[0,+\infty)$. In addition, if $\lambda>0$, $p>1$ and $n\geq 2$, then
 $$ \bigg(\operatorname{exp}(\lambda|s|^{\frac{n}{n-1}})-\sum_{j=0}^{n-2} \frac{\lambda^j}{j!}|s|^{\frac{nj}{n-1}}\bigg)^p\leq  \operatorname{exp}(\lambda p|s|^{\frac{n}{n-1}})-\sum_{j=0}^{n-2} \frac{(\lambda p)^j}{j!}|s|^{\frac{nj}{n-1}}\quad\text{for all}\quad s\in\mathbb{R}. $$
\end{lemma} 
To close this section, we highlight a key result from Masmoudi--Sani \cite[Lemma 3.4]{Masmoudi-Sani-2015}, which is central to our analysis and is now stated for convenience.
\begin{lemma}\label{lem2.5}
Let $\{\xi_k\}_{k}$ with $k\in\mathbb{N}\cup\{0\}$ be a sequence of real numbers and $q\geq 1$. Denote 
$$\|\xi\|_q:=\bigg(\sum_{k=0}^{\infty}|\xi_k|^q\bigg)^{\frac{1}{q}},\quad\|\xi\|_{\operatorname{\exp}}:=\bigg(\sum_{k=0}^{\infty}|\xi_k|^n\operatorname{\exp}(kn)\bigg)^{\frac{1}{n}}$$
 and $$\mu(f):=\inf\{\|\xi\|_{\operatorname{\exp}}\colon \|\xi\|_1=f,~\|\xi\|_n\leq1\}.$$
 Then, for all $f>1$, there holds
 $$\mu(f)\sim\frac{\operatorname{\exp}(f^{\frac{n}{n-1}})}{f^{\frac{1}{n-1}}}. $$
\end{lemma}
\section{The Critical and Subcritical Anisotropic Singular Moser--Trudinger Inequalities}\label{sec3}
It is worth mentioning that to establish the critical and subcritical anisotropic Moser-Trudinger inequalities scholars usually rely on the properties of convex-symmetrization and the P\'{o}lya--Szeg\"{o} inequality. For example, we refer the reader to see Das--Tiwari \cite[Lemma 5.2]{Das-Tiwari-2026}, Wang--Xia \cite[Theorem 3.1]{Wang-Xia-2012}, Zhou--Zhou \cite[Theorem 1.2]{Zhou-Zhou-2020} and Zhu \cite[Theorem 2.1]{Zhu-2020}. However, in this section, to prove Theorems \ref{thm1.1} and \ref{thm1.2}, we develop a different method based on completely convex-symmetrization-free approach, which is motivated by the seminal work of Lam--Lu \cite{Lam-Lu-2013} in the isotropic setting. This method is quite easy to understand and easily applicable to establish Moser--Trudinger inequalities as well as the associated concentration-compactness principles in the context of Finsler metric framework.
\subsection{{Proof of Theorem \ref{thm1.1}}}
It is well-known that $C^\infty_0(\mathbb{R}^n)$ is a dense subset of $W^{1,n}(\mathbb{R}^n)$. To complete the proof, it suffices to verify the desired inequality \eqref{eq1.12} for any $u\in C^\infty_0(\mathbb{R}^n)\setminus\{0\}$ with $u\geq 0$ and satisfying $\|u\|\leq 1$. Next, we define
$$\mathcal{L}(u):=2^{-\frac{1}{n(n-1)}}\| u\|_n\quad\text{and}\quad\quad\Gamma:=\{x\in \mathbb{R}^n\colon~u(x)>\mathcal{L}(u)\}.$$
It is standard to check that
\begin{equation}\label{eq3.1}
   \mathcal{L}(u)<1\quad\text{and}\quad |\Gamma|\leq 2^{\frac{1}{n-1}}.
\end{equation}
Thus, we infer that $\Gamma$ is a bounded domain in $\mathbb{R}^n$. By splitting the whole integral in the following way, one has
$$ \int_{\mathbb{R}^n}\frac{\Phi\big(\lambda|u|^{\frac{n}{n-1}}\big)}{F^0(x)^\beta}\,\mathrm{d}x=\bigg(\int_{\mathbb{R}^n\setminus\Gamma}+\int_{\Gamma}\bigg) \frac{\Phi\big(\lambda|u|^{\frac{n}{n-1}}\big)}{F^0(x)^\beta}\,\mathrm{d}x=:J_1+J_2.$$
Due to \eqref{eq3.1}, we have $\mathbb{R}^n\setminus \Gamma\subset \{x\in \mathbb{R}^n:~u(x)<1\}$. Hence, we deduce from Lemma \ref{lem2.3} and $\|u\|\leq1$ that
\begin{equation}\label{eq3.2}
\begin{aligned}
J_1 & \leq \int_{\{x\in \mathbb{R}^n\colon~ u(x)<1\}}\frac{1}{F^0(x)^\beta}\sum_{j=n-1}^\infty \frac{\lambda^j}{j!}|u|^{\frac{nj}{n-1}}\,\mathrm{d}x\\
&\leq \int_{\{x\in \mathbb{R}^n\colon~ u(x)<1\}}\frac{1}{F^0(x)^\beta}\sum_{j=n-1}^\infty \frac{\lambda^j}{j!}|u|^{n}\,\mathrm{d}x \\ 
 & \leq \operatorname{\exp}(\lambda)\Bigg( \int_{\{x\in \mathbb{R}^n\colon~ u(x)<1,~F^0(x)<1\}}+  \int_{\{x\in \mathbb{R}^n\colon~u(x)<1,~F^0(x)\geq1\}}\Bigg)\frac{|u|^{n}}{F^0(x)^\beta}\,\mathrm{d}x\\ 
& \leq \operatorname{\exp}(\lambda)\Bigg( \int_{\{ x\in \mathbb{R}^n\colon~F^0(x)<1\}}\frac{1}{F^0(x)^\beta}\,\mathrm{d}x+ \int_{\{ x\in \mathbb{R}^n\colon~F^0(x)\geq1\}}|u|^{n}\,\mathrm{d}x\Bigg)\\
&\leq C(\lambda,\beta,n),
\end{aligned}
\end{equation}
where $C(\lambda,\beta,n)>0$ is a suitable constant depending only on  $\lambda$, $\beta$ and $n$. Finally, we compute $J_2$. Define a function $\psi\colon \Gamma\to\mathbb{R}$ by $\psi(x):=u(x)-\mathcal{L}(u)$ for all $x\in \Gamma$. Hence, one can easily obtain $\psi\in W^{1,n}_0(\Gamma)$. It follows that
\begin{align*}
    |u|^{\frac{n}{n-1}} &\leq  (|\psi|+\mathcal{L}(u))^{\frac{n}{n-1}} \\
    &\leq |\psi|^{\frac{n}{n-1}}+\frac{n}{n-1} 2^{\frac{n}{n-1}-1}\bigg(|\psi|^{\frac{n}{n-1}-1}|\mathcal{L}(u)|+|\mathcal{L}(u)|^{\frac{n}{n-1}}\bigg)\\
    & \leq |\psi|^{\frac{n}{n-1}}+\frac{n}{n-1} 2^{\frac{n}{n-1}-1}\Bigg(\frac{|\psi|^{\frac{n}{n-1}}|\mathcal{L}(u)|^n}{n}+\frac{n-1}{n}+|\mathcal{L}(u)|^{\frac{n}{n-1}}\Bigg) \\
    & \leq \bigg(1+\frac{2^{\frac{1}{n-1}}}{n-1}|\mathcal{L}(u)|^n\bigg)|\psi|^{\frac{n}{n-1}}+R_n\quad\text{in}\quad \Gamma,
\end{align*}
with $$R_n:=\frac{n}{n-1} 2^{\frac{1}{n-1}}\bigg(\frac{n-1}{n}+1\bigg),$$
where the above estimates follow from Young's inequality and the inequality stated below
$$(a+b)^r\leq a^r+r 2^{r-1}(a^{r-1}b+b^r)\quad\text{for all}\quad r\geq 1\quad\text{and}\quad a,b\geq 0.$$
Define another function $\zeta\colon\Gamma\to\mathbb{R}$ by $$\zeta(x):=\bigg(1+\frac{2^{\frac{1}{n-1}}}{n-1}|\mathcal{L}(u)|^n\bigg)^{\frac{n-1}{n}}\psi(x)\quad\text{for all}\quad x\in\Gamma.$$
Notice that $\zeta\in W^{1,n}_0(\Gamma)$ and there holds  $|u|^{\frac{n}{n-1}}\leq |\zeta|^{\frac{n}{n-1}}+R_n$ in $\Gamma$. Moreover, by using the homogeneity of $F$, we deduce that
\begin{align*}
\|F(\nabla\zeta)\|^n_{L^n(\Gamma)}&=\bigg(1+\frac{2^{\frac{1}{n-1}}}{n-1}|\mathcal{L}(u)|^n\bigg)^{n-1}\|F(\nabla\psi)\|^n_{L^n(\Gamma)}\\
&=\bigg(1+\frac{2^{\frac{1}{n-1}}}{n-1}|\mathcal{L}(u)|^n\bigg)^{n-1}\|F(\nabla u)\|^n_{L^n(\Gamma)}\\ &\leq\bigg(1+\frac{2^{\frac{1}{n-1}}}{n-1}|\mathcal{L}(u)|^n\bigg)^{n-1}\|F(\nabla u)\|^n_n\\
&\leq \bigg(1+\frac{2^{\frac{1}{n-1}}}{n-1}|\mathcal{L}(u)|^n\bigg)^{n-1}\big(1-\|u\|^n_n\big).
\end{align*}
By employing $(1-x)^q\leq 1-q x$ for all $x\in[0,1]$ and $q\in[0,1]$, we obtain from the above inequality that  
\begin{align*}
 \|F(\nabla\zeta)\|_{L^n(\Gamma)}^{\frac{n}{n-1}}&\leq \bigg(1+\frac{2^{\frac{1}{n-1}}}{n-1}|\mathcal{L}(u)|^n\bigg)\big(1-\| u\|^n_n\big)^{\frac{1}{n-1}}\\
 &\leq\bigg(1+\frac{2^{\frac{1}{n-1}}}{n-1}|\mathcal{L}(u)|^n\bigg)\bigg(1-{\frac{1}{n-1}}\| u\|^n_n\bigg)\\ 
 &\leq \bigg(1+{\frac{1}{n-1}}\| u\|^n_n\bigg)\bigg(1-{\frac{1}{n-1}}\| u\|^n_n\bigg)\\
 &= 1-\frac{1}{(n-1)^2}\| u\|^{2n}_n\leq 1.   
\end{align*}
This concludes that $\zeta\in W^{1,n}_0(\Gamma)$ and $\|F(\nabla \zeta)\|_{L^n(\Gamma)}\leq 1$. Now, we deduce from the anisotropic singular Moser--Trudinger inequality in the bounded domain (see Zhu \cite[Theorem 2.1]{Zhu-2020}) and \eqref{eq3.1} that
\begin{equation}\label{eq3.3}
\begin{aligned}
 J_2 & \leq \int_\Gamma\frac{\operatorname{\exp}(\lambda |u|^{\frac{n}{n-1}})}{F^0(x)^\beta}\,\mathrm{d}x\leq \int_\Gamma\frac{\operatorname{\exp}(\lambda_{\beta,n} (|\zeta|^{\frac{n}{n-1}}+R_n))}{F^0(x)^\beta}\,\mathrm{d}x \\ 
 &\leq \operatorname{\exp}(\lambda_{\beta,n} R_n)C(\beta,n)|\Gamma|^{1-\frac{\beta}{n}}\leq C(\beta,n),
\end{aligned}
\end{equation}
where $C(\beta,n)>0$ is a suitable constant depending only on  $\beta$ and $n$. From \eqref{eq3.2} and \eqref{eq3.3}, the inequality \eqref{eq1.12} is well-established. To finish the prove, it remains to demonstrate the sharpness of $\lambda_{\beta,n}$. To this end, let $s>0$ be fixed and consider a Moser-type sequence of functions $\{v_k\}_k$, which is defined by
\begin{align}\label{eq3.4}
    v_k(x):=\frac{1}{(n\kappa_n)^{\frac{1}{n}}}\begin{cases}
    (\ln k)^{\frac{n-1}{n}}\quad\quad\quad\quad\quad\quad\text{if}\quad F^0(x)\in[0,sk^{-1}],\\
    (\ln k)^{-\frac{1}{n}}\ln\big(\frac{s}{F^0(x)}\big)\quad\quad\text{if}\quad F^0(x)\in[sk^{-1},s],\\
    0 \quad\quad\quad\quad\quad\quad\qquad\quad\quad\text{if}\quad F^0(x)\in[s,+\infty).
\end{cases}   
\end{align}
Through a direct calculation, we obtain 
\begin{align*}
    F(\nabla v_k)=(n\kappa_n )^{-\frac{1}{n}}(\ln k)^{-\frac{1}{n}}\frac{1}{F^0(x)}~\chi_{\mathcal{W}_s\setminus\mathcal{W}_{sk^{-1}}}(x)
\end{align*}
and 
\begin{align*}
    \int_{\mathbb{R}^n}F(\nabla v_k)^n\,\mathrm{d}x&=(n\kappa_n \ln k )^{-1}\int_{sk^{-1}}^{s}\frac{1}{\eta^n}\int_{\partial \mathcal{W}_\eta}\frac{1}{|\nabla F^0(x)|}\,\mathrm{d}\sigma\,\mathrm{d}\eta\\
    &=(n\kappa_n \ln k )^{-1}n\kappa_n\int_{sk^{-1}}^{s}\eta^{-1}\,\mathrm{d}\eta=(\ln k )^{-1}\ln\Big(\frac{s}{sk^{-1}}\Big)=1.
\end{align*}
Now, by successive integrations by parts, we obtain for all $\beta\in[0,n)$ that
\begin{align*}
&\int_{0}^{\ln k}r^{n-1}\operatorname{\exp}(-(n-\beta)r)\,\mathrm{d}r\\
&=\frac{n-1}{n-\beta}\int_{0}^{\ln k}r^{n-2}\operatorname{\exp}(-(n-\beta)r)\,\mathrm{d}r-\bigg(\frac{1}{n-\beta}\bigg)\frac{(\ln k)^{n-1}}{k^{n-\beta}}\\
&=-~\frac{n-1}{(n-\beta)^2}\int_{0}^{\ln k}r^{n-2}\,\mathrm{d}\Big(\operatorname{\exp}(-(n-\beta)r)\Big)-\bigg(\frac{1}{n-\beta}\bigg)\frac{(\ln k)^{n-1}}{k^{n-\beta}}\\
&=\frac{(n-1)(n-2)}{(n-\beta)^2}\int_{0}^{\ln k}r^{n-3}\operatorname{\exp}(-(n-\beta)r)\,\mathrm{d}r-\sum_{j=1}^{2}\bigg[\frac{(n-1)!}{(n-j)!(n-\beta)^j}\bigg]\frac{(\ln k)^{n-j}}{k^{n-\beta}}\\
&=-\frac{(n-1)(n-2)}{(n-\beta)^3}\int_{0}^{\ln k}r^{n-3}\,\mathrm{d}\Big(\operatorname{\exp}(-(n-\beta)r)\Big)-\sum_{j=1}^{2}\bigg[\frac{(n-1)!}{(n-j)!(n-\beta)^j}\bigg]\frac{(\ln k)^{n-j}}{k^{n-\beta}}\\
&=-\frac{(n-1)!}{(n-\beta)^n}\int_{0}^{\ln k}\,\mathrm{d}\Big(\operatorname{\exp}(-(n-\beta)r)\Big)-\sum_{j=1}^{n-1}\bigg[\frac{(n-1)!}{(n-j)!(n-\beta)^j}\bigg]\frac{(\ln k)^{n-j}}{k^{n-\beta}}\\
&=\bigg(1-\frac{1}{k^{n-\beta}}\bigg)\frac{(n-1)!}{(n-\beta)^n}-\sum_{j=1}^{n-1}\bigg[\frac{(n-1)!}{(n-j)!(n-\beta)^j}\bigg]\frac{(\ln k)^{n-j}}{k^{n-\beta}}\\
&= \frac{(n-1)!}{(n-\beta)^n}+o_k(1)\quad\text{as}\quad k\to\infty.
\end{align*}
The preceding estimate, together with integration by parts and a change of variables, yields for all $\beta\in[0,n)$ that
\begin{align*}
\int_{\mathbb{R}^n}\frac{|v_k|^n}{F^0(x)^\beta}\,\mathrm{d}x&=(n\kappa_n)^{-1}(\ln k)^{n-1} \int_{0}^{sk^{-1}}\frac{1}{\eta^\beta}\int_{\partial \mathcal{W}_\eta}\frac{1}{|\nabla F^0(x)|}\,\mathrm{d}\sigma\,\mathrm{d}\eta \\
&\qquad+(n\kappa_n \ln k )^{-1}\int_{sk^{-1}}^{s}\bigg|\ln\bigg(\frac{s}{\eta}\bigg)\bigg|^n\frac{1}{\eta^\beta}\int_{\partial \mathcal{W}_\eta}\frac{1}{|\nabla F^0(x)|}\,\mathrm{d}\sigma\,\mathrm{d}\eta \\
&=(\ln k)^{n-1}\int_{0}^{sk^{-1}}\eta^{n-\beta-1}\,\mathrm{d}\eta+\bigg(\frac{1}{\ln k}\bigg) \int_{sk^{-1}}^{s}\bigg|\ln\bigg(\frac{s}{\eta}\bigg)\bigg|^n  \eta^{n-\beta-1}\,\mathrm{d}\eta\\
&=(\ln k)^{n-1}~\frac{(sk^{-1})^{n-\beta}}{n-\beta}+\bigg(\frac{s^{n-\beta}}{\ln k}\bigg)\int_{0}^{\ln k}r^n\operatorname{\exp}(-(n-\beta)r)\,\mathrm{d}r\\
&=\bigg(\frac{n}{n-\beta}\bigg)\bigg(\frac{s^{n-\beta}}{\ln k}\bigg)\int_{0}^{\ln k}r^{n-1}\operatorname{\exp}(-(n-\beta)r)\,\mathrm{d}r\\
&=O\Big(\frac{1}{\ln k}\Big)\quad\text{as}\quad k\to\infty.
\end{align*}
 The above estimate will serve as an essential tool throughout this article in demonstrating the sharpness of the best constants for anisotropic singular Moser--Trudinger inequalities with the exact growth condition. Moreover, we obtain from the above computations that
$$\|F(\nabla v_k)\|^n_n=1 \qquad \text{and}\qquad \|v_k\|_n^n=O\Big(\frac{1}{\ln k}\Big)\quad\text{as}\quad k\to\infty.$$
Hence, it follows that $\|v_k\|=1+o_k(1)$ as $k\to\infty$. Now, we choose $\lambda>\lambda_{\beta,n}$ and thus, there exists $\vartheta_0>0$ such that $\lambda=(1+\vartheta_0)\lambda_{\beta,n}$. In view of the above discussion, we infer that 
\begin{align*}
    \sup_{u\in W^{1,n}(\mathbb{R}^n),~\|u\|\leq 1}&\int_{\mathbb{R}^n}\frac{\Phi\big(\lambda|u|^{\frac{n}{n-1}}\big)}{F^0(x)^\beta}\,\mathrm{d}x\\
    &\geq \int_{\{x\in\mathbb{R}^n\colon~F^0(x)\leq sk^{-1} \}}\frac{\Phi\big(\lambda\big|v_k/\|v_k\|\big|^{\frac{n}{n-1}}\big)}{F^0(x)^\beta}\,\mathrm{d}x \\
    & =\int_{0}^{sk^{-1}}\frac{\Phi\Big(\lambda\Big|\frac{(n\kappa_n)^{-\frac{1}{n}}(\ln k)^{\frac{n-1}{n}}}{\|v_k\|}\Big|^{\frac{n}{n-1}}\Big)}{\eta^{\beta}}\int_{\partial\mathcal{W}_\eta}\frac{1}{|\nabla F^0(x)|}\,\mathrm{d}\sigma\,\mathrm{d}\eta\\
    &\gtrsim  \operatorname{\exp}\Bigg(\lambda\Bigg|\frac{(n\kappa_n)^{-\frac{1}{n}}(\ln k)^{\frac{n-1}{n}}}{\|v_k\|}\Bigg|^{\frac{n}{n-1}}\Bigg)\int_{0}^{sk^{-1}}\eta^{n-\beta-1}\,\mathrm{d}\eta\\
    &\gtrsim \operatorname{\exp}\Bigg(\lambda\Bigg|\frac{(n\kappa_n)^{-\frac{1}{n}}(\ln k)^{\frac{n-1}{n}}}{\|v_k\|}\Bigg|^{\frac{n}{n-1}}\Bigg)k^{-(n-\beta)}\\
    &=\operatorname{\exp}\Bigg((n-\beta)\ln k\Bigg(\frac{1+\vartheta_0}{\|v_k\|^{\frac{n}{n-1}}}-1\Bigg)\Bigg)\to+\infty\end{align*} 
    as $k\to\infty$. The sharpness of $\lambda_{\beta,n}$ is thereby established, which completes the proof of Theorem \ref{thm1.1}.\qed

The following two results, obtained as consequences of the previous theorem, are also of significant importance for our analysis.
\begin{corollary}\label{cor3.1}
Assume that $n\geq2$ and $0\leq \beta<n$. Then for all $0<\lambda\leq \lambda_{\beta,n}$ and $u\in W^{1,n}(\mathbb{R}^n)$, there holds
\begin{align}\label{eq3.5}
    \sup_{u\in W^{1,n}(\mathbb{R}^n),~\|u\|_\zeta\leq 1}\int_{\mathbb{R}^n}\frac{\Phi\big(\lambda|u|^{\frac{n}{n-1}}\big)}{F^0(x)^\beta}\,\mathrm{d}x<+\infty
\end{align}
for any real number $\zeta>0$. The constant $\lambda_{\beta,n}$ is sharp, in the sense that if $\lambda>\lambda_{\beta,n}$, then the supremum \eqref{eq3.5} is no longer finite.    
\end{corollary}
\begin{proof}
    Let $\zeta>0$, $0<\lambda\leq \lambda_{\beta,n}$, and let $u\in W^{1,n}(\mathbb{R}^n)$ be such that $\|u\|_\zeta\leq 1$. Define $w:\mathbb{R}^n\to\mathbb{R}$ by setting
    $$w(\zeta^{\frac{1}{n}}x)=u(x)\quad\text{for all}\quad \zeta>0\quad\text{and}\quad x\in\mathbb{R}^n. $$
    By performing the change of variables $y=\zeta^{\frac{1}{n}}x$, yields 
    $$\int_{\mathbb{R}^n}\big(F(\nabla w)^n+|w|^n\big)\,\mathrm{d}x=\int_{\mathbb{R}^n}\big(F(\nabla u)^n+\zeta|u|^n\big)\,\mathrm{d}x\leq 1$$
    and
    $$ \int_{\mathbb{R}^n}\frac{\Phi\big(\lambda|u|^{\frac{n}{n-1}}\big)}{F^0(x)^\beta}\,\mathrm{d}x=\frac{1}{\zeta^{1-\frac{\beta}{n}}}\int_{\mathbb{R}^n}\frac{\Phi\big(\lambda|w|^{\frac{n}{n-1}}\big)}{F^0(x)^\beta}\,\mathrm{d}x.$$ 
    This shows that $w\in W^{1,n}(\mathbb{R}^n)$ with $\|w\|\leq 1$. It then follows from Theorem \ref{thm1.1} that
\begin{align*}
 \sup_{u\in W^{1,n}(\mathbb{R}^n),~\|u\|_\zeta\leq 1}\int_{\mathbb{R}^n}\frac{\Phi\big(\lambda|u|^{\frac{n}{n-1}}\big)}{F^0(x)^\beta}\,\mathrm{d}x&\leq \frac{1}{\zeta^{1-\frac{\beta}{n}}}\sup_{w\in W^{1,n}(\mathbb{R}^n),~\|w\|\leq 1}\int_{\mathbb{R}^n}\frac{\Phi\big(\lambda|w|^{\frac{n}{n-1}}\big)}{F^0(x)^\beta}\,\mathrm{d}x\\
 &<+\infty.   
\end{align*}
 Now, by repeating the argument used in the proof of Theorem \ref{thm1.1}, one verifies the sharpness of $\lambda_{\beta,n}$, which concludes the proof.
\end{proof}
\begin{corollary}\label{cor3.2}
Assume that $n\geq2$ and $0\leq \beta<n$. Then for all $\lambda>0$ and $u\in W^{1,n}(\mathbb{R}^n)$, there holds
\begin{align*}
    \frac{\Phi\big(\lambda|u|^{\frac{n}{n-1}}\big)}{F^0(x)^\beta}\in L^1(\mathbb{R}^n).
\end{align*}    
\end{corollary}
\begin{proof}
Let $\lambda>0$ and $u\in W^{1,n}(\mathbb{R}^n)$. By the density of
$C_0^\infty(\mathbb{R}^n)$ in $W^{1,n}(\mathbb{R}^n)$, for any given
$\varepsilon>0$, there exists $w\in C_0^\infty(\mathbb{R}^n)$ such that $\|u-w\|<\varepsilon$. Next, we decompose the whole integral as follows
$$ \int_{\mathbb{R}^n}\frac{\Phi\big(\lambda|u|^{\frac{n}{n-1}}\big)}{F^0(x)^\beta}\,\mathrm{d}x=\bigg(\int_{{D}_1}+\int_{{D}_2}\bigg) \frac{\Phi\big(\lambda|u|^{\frac{n}{n-1}}\big)}{F^0(x)^\beta}\,\mathrm{d}x=:J_1+J_2,$$
where
$${D}_1:=\{x\in \mathbb{R}^n\colon\,|u(x)-w(x)|\leq 1\}\quad\text{and}\quad {D}_2:=\{x\in \mathbb{R}^n\colon\,|u(x)-w(x)|> 1\}.$$
It is standard to check that $|u(x)|\leq 1+\sup_{x\in \mathbb{R}^n}|w(x)|
:=C$ for all $x\in {D}_1$, where $C>0$ is a suitable constant. By applying Lemma \ref{lem2.3}, we get
\begin{equation*}
\begin{aligned}
       J_1&=\int_{{D}_1}\sum_{j=n-1}^{\infty}\frac{\lambda^j}{j!}\frac{|u|^{\frac{nj}{n-1}}}{F^0(x)^\beta}\,\mathrm{d}x=\int_{{D}_1}\sum_{j=n-1}^{\infty}\frac{\big(\lambda C^{\frac{n}{n-1}}\big)^j}{j!}\bigg|\frac{u}{C}\bigg|^{\frac{nj}{n-1}}\,\frac{\mathrm{d}x}{F^0(x)^\beta}\\
       &\leq \sum_{j=n-1}^{\infty}\frac{\big(\lambda C^{\frac{n}{n-1}}\big)^j}{j!}\int_{{D}_1}\bigg|\frac{u}{C}\bigg|^n\,\frac{\mathrm{d}x}{F^0(x)^\beta}
       \lesssim  \int_{{D}_1}\frac{|u|^n}{F^0(x)^\beta}\,\mathrm{d}x\\ &\lesssim \int_{\{x\in \mathbb{R}^n:~F^0(x)<1\}} \frac{1}{F^0(x)^\beta} \,\mathrm{d}x+\int_{\{x\in \mathbb{R}^n:~F^0(x)\geq 1\}} |u|^n \,\mathrm{d}x<+\infty.
   \end{aligned}
   \end{equation*}
It is well-known that for a given $\delta>0$ and $q>1$, we have
   \begin{align}\label{eq3.6}
     (a+b)^q\leq (1+\delta)a^q+C_\delta b^q\quad\text{for all}\quad a,b>0,  
\end{align}
where $C_\delta=\big(1-(1+\delta)^{-\frac{1}{q-1}}\big)^{1-q}$. In virtue of \eqref{eq3.6} together with Young's inequality, we obtain that
\begin{align*}
   J_2 & \leq \int_{{D}_2}\frac{\operatorname{\exp}(\lambda|u|^{\frac{n}{n-1}})}{F^0(x)^\beta}\,\mathrm{d}x\leq \int_{{D}_2 \cap\, \text{supp}(w)}\frac{\operatorname{\exp}(\lambda(1+\delta)|u-w|^{\frac{n}{n-1}}+C_\delta\lambda|w|^{\frac{n}{n-1}})}{F^0(x)^\beta}\,\mathrm{d}x\\
   & \lesssim\int_{{D}_2 \cap\, \text{supp}(w)}\frac{\operatorname{\exp}(2\lambda(1+\delta)|u-w|^{\frac{n}{n-1}})}{F^0(x)^\beta}\,\mathrm{d}x+\int_{\text{supp}(w)}\frac{\operatorname{\exp}(2C_\delta\lambda|w|^{\frac{n}{n-1}})}{F^0(x)^\beta}\,\mathrm{d}x\\
   &<+\infty.
\end{align*}
Take $\delta,\varepsilon>0$ sufficiently small so that $2\lambda(1+\delta)\varepsilon^{\frac{n}{n-1}}< \lambda_{\beta,n}$. It follows from the definition of ${D}_2$ and Theorem \ref{thm1.1} that
\begin{align*}
\int_{{D}_2 \cap\, \text{supp}(w)}&\frac{\operatorname{\exp}(2\lambda(1+\delta)|u-w|^{\frac{n}{n-1}})}{F^0(x)^\beta}\,\mathrm{d}x \\&\lesssim \int_{\mathbb{R}^n}\frac{\Phi\big(2\lambda(1+\delta)|u-w|^{\frac{n}{n-1}}\big)}{F^0(x)^\beta}\,\mathrm{d}x\\
&\lesssim\int_{\mathbb{R}^n}\frac{\Phi\Big(\lambda_{\beta,n}\Big|u-w/\|u-w\|\Big|^{\frac{n}{n-1}}\Big)}{F^0(x)^\beta}\,\mathrm{d}x <+\infty.  
\end{align*}
The compact support of $w$ ensures that
$$ \int_{\text{supp}(w)}\frac{\operatorname{\exp}(2C_\delta\lambda|w|^{\frac{n}{n-1}})}{F^0(x)^\beta}\,\mathrm{d}x<+\infty. $$
Taken together, the preceding estimates imply that
$$\int_{\mathbb{R}^n} \frac{\Phi\big(\lambda|u|^{\frac{n}{n-1}}\big)}{F^0(x)^\beta}\,\mathrm{d}x<+\infty.$$
The proof is complete, as desired.
\end{proof}
Next, we begin by stating the following important result, which will be repeatedly used in the proof of Theorem \ref{thm1.2}.
\begin{lemma}\label{lem3.3}
 Assume that $n\geq2$, $0\leq\beta<n$ and $q\geq n$. Then for all $u\in W^{1,n}(\mathbb{R}^n)$ with $\|F(\nabla u)\|_n\leq1$, there holds
 \begin{align}\label{eq3.7}
     \int_{\mathbb{R}^n}\frac{|u|^q}{F^0(x)^\beta}\,\mathrm{d}x\leq C\bigg(\int_{\mathbb{R}^n}|u|^n\,\mathrm{d}x\bigg)^{1-\frac{\beta}{n}},
 \end{align}
 where $C=C(q,n,\beta)>0$ is a constant depending only on $q$, $n$ and $\beta$.
\end{lemma}
\begin{proof}
    By using density of $C_0^\infty(\mathbb{R}^n)$ in $W^{1,n}(\mathbb{R}^n)$, it suffices to prove \eqref{eq3.7} for all $u\in C_0^\infty(\mathbb{R}^n)\setminus\{0\}$, $u\geq0$ and $\|F(\nabla u)\|_n\leq1$. Define the set
    ${D}:=\{x\in \mathbb{R}^n\colon \,u(x)>1\}$.
    This infers that
    \begin{align}\label{eq3.8}
        |{D}|=\int_D\,\mathrm{d}x\leq \int_D|u|^n\,\mathrm{d}x\leq\int_{\mathbb{R}^n}|u|^n\,\mathrm{d}x.
    \end{align}
    Splitting the integral over $\mathbb{R}^n\setminus {D}$ and ${D}$, we can write
    $$ \int_{\mathbb{R}^n}\frac{|u|^q}{F^0(x)^\beta}\,\mathrm{d}x=\bigg(\int_{\mathbb{R}^n\setminus {D}}+\int_{{D}}\bigg)\frac{|u|^q}{F^0(x)^\beta}\,\mathrm{d}x=:J_1+J_2.$$
    To estimate $J_1$, one can notice that $u(x)\leq 1$ for all $x\in\mathbb{R}^n\setminus {D}$ and thus, we have
    \begin{align*}
        J_1&\leq \int_{\{x\in \mathbb{R}^n, u(x)\leq 1,F^0(x)\geq\|u\|_n\}}\frac{|u|^n}{F^0(x)^\beta}\,\mathrm{d}x+\int_{\{x\in \mathbb{R}^n, u(x)\leq 1,F^0(x)<\|u\|_n\}}\frac{|u|^n}{F^0(x)^\beta}\,\mathrm{d}x\\
        &\leq\frac{1}{\|u\|_n^\beta}\int_{\mathbb{R}^n}|u|^n\,\mathrm{d}x+\int_{0}^{\|u\|_n}\frac{1}{\eta^\beta}\int_{\partial\mathcal{W}_\eta}\frac{1}{|\nabla F^0(x)|}\,\mathrm{d}\sigma\,\mathrm{d}\eta\\
        &=\|u\|_n^{n\big(1-\frac{\beta}{n}\big)}+n\kappa_n \int_{0}^{\|u\|_n}\eta^{n-\beta-1}\,\mathrm{d}\eta=\bigg(1+\frac{n\kappa_n}{n-\beta}\bigg)\|u\|_n^{n\big(1-\frac{\beta}{n}\big)}\\
        &\leq C(n,\beta)\|u\|_n^{n\big(1-\frac{\beta}{n}\big)},
    \end{align*}
where $C(n,\beta)>0$ is a constant depending on $n$ and $\beta$. Next, we define a function $w\colon D\to\mathbb{R}$ by $w(x)=u(x)-1$ for all $x\in{D}$. This shows that $\|F(\nabla w)\|_{L^n({D})}\leq1$ and $w\in W^{1,n}_0({D})$. It directly implies that
\begin{align*}
  |u|^q&\leq (|w|+1)^q\leq 2^{q-1}(|w|^q+1)\\
  &\leq 2^{q-1}\big(1+C(n,\beta)\big)\operatorname{\exp}(\lambda_{\beta,n}|w|^{\frac{n}{n-1}})\quad\text{in}\quad{D}.
\end{align*}
In virtue of the anisotropic singular Moser--Trudinger inequality in the bounded domain (see Zhu \cite[Theorem 2.1]{Zhu-2020}) and \eqref{eq3.8}, we have
\begin{align*}
    J_2&\leq 2^{q-1}\big(1+C(n,\beta)\big)\int_{{D}}\frac{\operatorname{\exp}(\lambda_{\beta,n}|w|^{\frac{n}{n-1}})}{F^0(x)^\beta}\,\mathrm{d}x\\
    &\leq C(q,n,\beta)|{D}|^{1-\frac{\beta}{n}}\leq C(q,n,\beta)\|u\|_n^{n\big(1-\frac{\beta}{n}\big)},
\end{align*}
where $C=C(q,n,\beta)>0$ is a constant depending only on $q$, $n$ and $\beta$. The conclusion now follows by combining the estimates of $J_1$ and $J_2$, and therefore \eqref{eq3.7} is established. This finishes the proof.
\end{proof}
Now, we are fully equipped to establish Theorem \ref{thm1.2}.
\subsection{{Proof of Theorem \ref{thm1.2}}}
Let $0<\lambda<\lambda_{\beta,n}$ and $u\in W^{1,n}(\mathbb{R}^n)$ be such that $\|F(\nabla u)\|_n\leq1$. Moreover, the integral admits the following splitting
\begin{align*}
\int_{\mathbb{R}^n}\frac{\operatorname{exp}\big(\lambda|u|^{\frac{n}{n-1}}\big)|u|^n}{F^0(x)^\beta}\,\mathrm{d}x=:J_1+J_2,    
\end{align*}
where $$J_1:=\int_{\mathbb{R}^n}\frac{\Phi\big(\lambda|u|^{\frac{n}{n-1}}\big)|u|^n}{F^0(x)^\beta}\,\mathrm{d}x\quad\text{and}\quad J_2:=\int_{\mathbb{R}^n}\sum_{j=0}^{n-2}\bigg(\frac{\lambda^j}{j!}\bigg)\frac{|u|^{(\frac{n}{n-1})(n-1+j)}}{F^0(x)^\beta}\,\mathrm{d}x. $$
Let $\vartheta>1$ be such that $\frac{1}{\vartheta}+\frac{1}{\vartheta^\prime}=1$. Take $\vartheta>1$ close to $1$ such that $\lambda\vartheta<\lambda_{\beta,n}$. By using H\"{o}lder's inequality, the subcritical anisotropic singular Moser--Trudinger inequality established by Song--Li--Zhu \cite[Theorem 3]{Song-Li-Zhu-2021}, Lemmas \ref{lem2.4} and \ref{lem3.3}, we get
\begin{equation}\label{eq3.9}
\begin{aligned}
J_1&\leq\Bigg(\int_{\mathbb{R}^n}\frac{\Phi\big(\lambda\vartheta|u|^{\frac{n}{n-1}}\big)}{F^0(x)^\beta}\,\mathrm{d}x\Bigg)^{\frac{1}{\vartheta}}\Bigg(\int_{\mathbb{R}^n}\frac{|u|^{n\vartheta^\prime}}{F^0(x)^\beta}\,\mathrm{d}x\Bigg)^{\frac{1}{\vartheta^\prime}}\\
&\lesssim \bigg(\int_{\mathbb{R}^n}|u|^n\,\mathrm{d}x\bigg)^{\frac{1}{\vartheta}\big(1-\frac{\beta}{n}\big)}\bigg(\int_{\mathbb{R}^n}|u|^n\,\mathrm{d}x\bigg)^{\frac{1}{\vartheta^\prime}\big(1-\frac{\beta}{n}\big)}\\
&=\bigg(\int_{\mathbb{R}^n}|u|^n\,\mathrm{d}x\bigg)^{1-\frac{\beta}{n}}.
\end{aligned}
\end{equation}
Once again by employing Lemma \ref{lem3.3}, we have 
\begin{equation}\label{eq3.10}
\begin{aligned}
        J_2&\leq\sum_{j=0}^{n-2}\bigg(\frac{\lambda^j}{j!}\bigg)\int_{\mathbb{R}^n}\frac{|u|^{(\frac{n}{n-1})(n-1+j)}}{F^0(x)^\beta}\,\mathrm{d}x\\
        &\lesssim\sum_{j=0}^{n-2}\bigg(\frac{\lambda^j}{j!}\bigg)\bigg(\int_{\mathbb{R}^n}|u|^n\,\mathrm{d}x\bigg)^{1-\frac{\beta}{n}}\\
        &\lesssim \bigg(\int_{\mathbb{R}^n}|u|^n\,\mathrm{d}x\bigg)^{1-\frac{\beta}{n}}.
    \end{aligned}
\end{equation}
It follows from \eqref{eq3.9} and \eqref{eq3.10} that \eqref{eq1.13} holds. To finish the proof, it remains only to establish the sharpness of $\lambda_{\beta,n}$. Let $\lambda\geq \lambda_{\beta,n}$ be given. Here, we only discuss the proof for the case $\lambda> \lambda_{\beta,n}$, as the equality case can be handled analogously. Thus, there exists $\varepsilon_0> 0$ such that $\lambda=(1+\varepsilon_0)\lambda_{\beta,n}$. Next, we consider the same Moser-type sequence of functions $\{v_k\}_k$ introduced in the proof of Theorem \ref{thm1.1}. This sequence enjoys the following estimates
\begin{align*}
\|F(\nabla v_k)\|^n_n=1 \qquad \text{and}\qquad \|v_k\|_n^n=O\Big(\frac{1}{\ln k}\Big)\quad\text{as}\quad k\to\infty.    
\end{align*}
Bringing together the preceding estimates, we obtain 
\begin{align*}
&\int_{\mathbb{R}^n}\frac{\operatorname{exp}\big(\lambda|v_k|^{\frac{n}{n-1}}\big)|v_k|^n}{F^0(x)^\beta}\,\mathrm{d}x \\ &\geq\int_{\{x\in\mathbb{R}^n\colon F^0(x)\leq sk^{-1}\}}\frac{\operatorname{exp}\big(\lambda|v_k|^{\frac{n}{n-1}}\big)|v_k|^n}{F^0(x)^\beta}\,\mathrm{d}x\\
&=\operatorname{\exp}\Big(\lambda\Big|(n\kappa_n)^{-\frac{1}{n}}(\ln k)^{\frac{n-1}{n}}\Big|^{\frac{n}{n-1}}\Big)(n\kappa_n)^{-1}(\ln k)^{n-1}\int_{0}^{sk^{-1}}\frac{1}{\eta^\beta}\int_{\partial\mathcal{W}_\eta}\frac{1}{|\nabla F^0(x)|}\,\mathrm{d}\sigma\,\mathrm{d}\eta\\
&=\operatorname{\exp}\Big(\lambda\Big|(n\kappa_n)^{-\frac{1}{n}}(\ln k)^{\frac{n-1}{n}}\Big|^{\frac{n}{n-1}}\Big)(\ln k)^{n-1}\int_{0}^{sk^{-1}}\eta^{n-\beta-1}\,\mathrm{d}\eta\\
&\gtrsim \operatorname{\exp}\Big(\lambda\Big|(n\kappa_n)^{-\frac{1}{n}}(\ln k)^{\frac{n-1}{n}}\Big|^{\frac{n}{n-1}}\Big)(\ln k)^{n-1} k^{-(n-\beta)}=k^{\varepsilon_0(n-\beta)}(\ln k)^{n-1}.
\end{align*}
This yields that
\begin{align*}
    \sup_{u\in W^{1,n}(\mathbb{R}^n),~\|F(\nabla u)\|_n\leq 1}&\frac{1}{\|u\|_n^{n\big(1-\frac{\beta}{n}\big)}}\int_{\mathbb{R}^n}\frac{\operatorname{exp}\big(\lambda|u|^{\frac{n}{n-1}}\big)|u|^n}{F^0(x)^\beta}\,\mathrm{d}x\\ &\geq\frac{1}{\|v_k\|_n^{n\big(1-\frac{\beta}{n}\big)}}\int_{\mathbb{R}^n}\frac{\operatorname{exp}\big(\lambda|v_k|^{\frac{n}{n-1}}\big)|v_k|^n}{F^0(x)^\beta}\,\mathrm{d}x\\
    &\gtrsim k^{\varepsilon_0(n-\beta)}(\ln k)^{n-\frac{n}{\beta}}\to+\infty\quad\text{as}\quad k\to\infty.
\end{align*}
Hence, the desired sharpness of $\lambda_{\beta,n}$ follows, and the proof of Theorem \ref{thm1.2} is thus completed. \qed

To conclude this section, the following subcritical anisotropic singular Moser--Trudinger inequality is stated without proof. For a detailed proof, the reader can proceed with Das--Tiwari \cite[Lemma 5.2]{Das-Tiwari-2026} and requires only minor adjustments. This inequality is particularly relevant to variational problems associated with the $n$-Finsler-Laplace operator, where the nonlinear term exhibits exponential growth and involves a singular weight.
\begin{lemma}[\bf The Subcritical Anisotropic Singular Moser--Trudinger Inequality] 
	For all $n\geq 2$, $0\leq \beta<n$, $\lambda>0$ and $u\in W^{1,n}(\mathbb{R}^n)$, there holds
	\begin{align*}
		\int_{\mathbb{R}^n}\frac{\Phi\big(\lambda|u|^{\frac{n}{n-1}}\big)}{F^0(x)^\beta}\,\mathrm{d}x<+\infty.
	\end{align*}
	Moreover, if for all $u\in W^{1,n}(\mathbb{R}^n)$ satisfying $\|F(\nabla u)\|_n\leq1$, $\|u\|_n\leq M<+\infty$ and $\lambda<\lambda_{\beta,n}$, then there exists a constant $C=C(n,M,\lambda)>0$ which depends only on $n,M$ and $\lambda$ such that
	\begin{align*}
		\int_{\mathbb{R}^n}\frac{\Phi\big(\lambda|u|^{\frac{n}{n-1}}\big)}{F^0(x)^\beta}\,\mathrm{d}x\leq C(n,M,\lambda).
	\end{align*}
\end{lemma}
\section{The Anisotropic Moser--Trudinger Inequality with Exact Growth and Its Consequence}\label{sec4}
In this section, motivated by the pioneering contributions of Ibrahim--Masmoudi--Nakanishi \cite{Ibrahim-Masmoudi-Nakanishi-2015} and Masmoudi--Sani \cite{Masmoudi-Sani-2015} in the isotropic framework, we develop the anisotropic counterpart of the exact-growth Moser--Trudinger inequality of Masmoudi--Sani \cite[Theorem 1.4]{Masmoudi-Sani-2015}. It is worth noting that Liu \cite[Theorem 1.1]{Liu-2022} subsequently established a more general anisotropic exact-growth Moser--Trudinger inequality involving the denominator $1+b|u|^{\frac{n}{n-1}}$, with $b>0$. For each fixed $b>0$, this denominator is comparable with $(1+|u|)^{\frac{n}{n-1}}$, and consequently the finiteness assertion in Theorem \ref{thm1.3} also follows from Liu's result, up to a multiplicative constant. However, the establishment of this inequality follows the traditional method of Masmoudi--Sani \cite{Masmoudi-Sani-2015} and is precisely tailored to the quasi-conformal-type transformation employed in the next section to derive the anisotropic singular Moser--Trudinger inequality with the exact growth condition.  Moreover, Theorem \ref{thm1.4} and Corollary \ref{cor4.2} also show that these inequalities provide a genuine strengthening of the anisotropic critical and subcritical Moser--Trudinger inequalities.

The following elementary result will be needed repeatedly in the subsequent proofs.
\begin{proposition}\label{prop4.1}
 Let $u\in W^{1,n}(\mathbb{R}^n)$ be a nonnegative, Wulff radially symmetric, and nonincreasing function such that $u(r)>1$ and $\int_{\mathbb{R}^n\setminus\mathcal{W}_r}F(\nabla u)^n\,\mathrm{d}x\leq n\kappa_n\ell$ for some $r,\ell>0$. Then the following holds
 $$\frac{\operatorname{\exp\Big(\frac{n u^{\frac{n}{n-1}}(r)}{\ell^{\frac{1}{n-1}}}\Big)}}{u^{\frac{n}{n-1}}(r)}~\ell^{\frac{n}{n-1}} r^n\leq C\int_{\mathbb{R}^n\setminus\mathcal{W}_r}|u|^n\,\mathrm{d}x,$$
 where $C:=C(n)>0$ is a constant depending only on $n$.
\end{proposition}
\begin{proof}
Due to rescaling, it suffices to show that for all nonnegative, Wulff radially symmetric, and nonincreasing function $u\in W^{1,n}(\mathbb{R}^n)$ with $u(1)>1$ and $\int_{\mathbb{R}^n\setminus\mathcal{W}_1}F(\nabla u)^n\,\mathrm{d}x\leq n\kappa_n$, there exists a constant $C=C(n)>0$ such that 
$$\frac{\operatorname{\exp\big(n u^\frac{n}{n-1}(1)\big)}}{u^{\frac{n}{n-1}}(1)}\leq C\int_{\mathbb{R}^n\setminus\mathcal{W}_1}|u|^n\,\mathrm{d}x.$$
For each $k\in\mathbb{N}\cup\{0\}$, we define
$$f_k=u\big(\operatorname{\exp}(k)\big)\quad\text{and}\quad \xi_k=f_k-f_{k+1}.$$
It follows that $\{\xi_k\}_k$ with $k\in\mathbb{N}\cup\{0\}$ is a nonnegative sequence and 
$$\|\xi\|_1=\sum_{k=0}^{\infty}|\xi_k|=\sum_{k=0}^{\infty}(f_k-f_{k+1})=f_0=u(1)>1. $$
 By employing H\"{o}lder's inequality, we obtain for all $k\geq0$ that
\begin{align*}
    \xi_k&=\int_{\operatorname{\exp}(k)}^{\operatorname{\exp}(k+1)}-u^\prime(\eta)\,\mathrm{d}\eta\leq\bigg(\int_{\operatorname{\exp}(k)}^{\operatorname{\exp}(k+1)}|u^\prime(\eta)|^n\eta^{n-1}\,\mathrm{d}\eta\bigg)^{\frac{1}{n}}\bigg(\int_{\operatorname{\exp}(k)}^{\operatorname{\exp}(k+1)}\frac{1}{\eta}\bigg)^{\frac{n-1}{n}}\\
    &=\bigg(\int_{\operatorname{\exp}(k)}^{\operatorname{\exp}(k+1)}|u^\prime(\eta)|^n\eta^{n-1}\,\mathrm{d}\eta\bigg)^{\frac{1}{n}}.
\end{align*}
This yields from the above estimate that
\begin{align*}
    \|\xi\|_n&=\bigg(\sum_{k=0}^{\infty}|\xi_k|^n\bigg)^{\frac{1}{n}}\leq \bigg(\sum_{k=0}^{\infty}\int_{\operatorname{\exp}(k)}^{\operatorname{\exp}(k+1)}|u^\prime(\eta)|^n\eta^{n-1}\,\mathrm{d}\eta\bigg)^{\frac{1}{n}}\\
    &=\bigg(\int_{1}^{\infty}|u^\prime(\eta)|^n\eta^{n-1}\,\mathrm{d}\eta\bigg)^{\frac{1}{n}}=\bigg(\frac{1}{n\kappa_n}\int_{\mathbb{R}^n\setminus\mathcal{W}_1}F(\nabla u)^n\,\mathrm{d}x\bigg)^{\frac{1}{n}}\leq 1.
\end{align*}
Moreover, we deduce that
\begin{align*}
\int_{\mathbb{R}^n\setminus\mathcal{W}_1}|u|^n\,\mathrm{d}x&=n\kappa_n  \int_{1}^{\infty}|u(\eta)|^n\eta^{n-1}\,\mathrm{d}\eta =n\kappa_n  \sum_{k=0}^{\infty}\int_{\operatorname{\exp}(k)}^{\operatorname{\exp}(k+1)}|u(\eta)|^n\eta^{n-1}\,\mathrm{d}\eta\\
&\geq n\kappa_n \sum_{k=0}^{\infty}|u\big(\operatorname{\exp}(k+1)\big)|^n\int_{\operatorname{\exp}(k)}^{\operatorname{\exp}(k+1)}\eta^{n-1}\,\mathrm{d}\eta\\
&=\kappa_n \sum_{k=0}^{\infty}|u\big(\operatorname{\exp}(k+1)\big)|^n\big(\operatorname{\exp}\big((k+1)n\big)-\operatorname{\exp}\big(kn\big)\big)\\
&=\kappa_n \sum_{k=0}^{\infty}|u\big(\operatorname{\exp}(k+1)\big)|^n\operatorname{\exp}\big((k+1)n\big)\big(1-\operatorname{\exp}(-n)\big)\\
&=\kappa_n \big(1-\operatorname{\exp}(-n)\big)\sum_{k=1}^{\infty}|u\big(\operatorname{\exp}(k)\big)|^n\operatorname{\exp}(kn)\\
&=\kappa_n \big(1-\operatorname{\exp}(-n)\big)\sum_{k=1}^{\infty}|f_k|^n\operatorname{\exp}(kn)\\
&\geq \kappa_n \big(1-\operatorname{\exp}(-n)\big)\sum_{k=1}^{\infty}|\xi_k|^n\operatorname{\exp}(kn).
\end{align*}
The above estimates ensure that
\begin{equation}\label{eq4.1}
\|\xi\|^n_{\operatorname{\exp}}=\xi_0^n+\sum_{k=1}^{\infty}|\xi_k|^n\operatorname{\exp}(kn)\leq f_0^n+C(n)\int_{\mathbb{R}^n\setminus\mathcal{W}_1}|u|^n\,\mathrm{d}x,
\end{equation}
where $C(n)>0$ is a constant depending on $n$. Next, we claim that for some constant $C(n)>0$, there holds
\begin{equation}\label{eq4.2}
f_0^n\leq C(n)\int_{\mathbb{R}^n\setminus\mathcal{W}_1}|u|^n\,\mathrm{d}x.
\end{equation}
Indeed, let $1<t<\operatorname{\exp}(\theta)$, where $\theta=2^{-\frac{n}{n-1}}$. Hence, we obtain from H\"{o}lder's inequality that
\begin{align*}
    f_0-u(t)&=\int_{1}^t-u^\prime(\eta)\,\mathrm{d}\eta\leq\bigg(\int_{1}^{\operatorname{\exp}(\theta)}|u^\prime(\eta)|^n\eta^{n-1}\,\mathrm{d}\eta\bigg)^{\frac{1}{n}}\bigg(\int_{1}^{\operatorname{\exp}(\theta)}\frac{1}{\eta}\bigg)^{\frac{n-1}{n}}\\
    &\leq\theta^{\frac{n-1}{n}}\bigg(\frac{1}{n\kappa_n}\int_{\mathbb{R}^n\setminus\mathcal{W}_1}F(\nabla u)^n\,\mathrm{d}x\bigg)^{\frac{1}{n}}\leq\frac{1}{2}<\frac{f_0}{2}.
\end{align*}
This shows that $\frac{f_0}{2}\leq u(t)$ for all $1<t<\operatorname{\exp}(\theta)$ and thus, we have
\begin{align*}
 \int_{\mathbb{R}^n\setminus\mathcal{W}_1}|u|^n\,\mathrm{d}x&\geq   n\kappa_n\int_{1}^{\operatorname{\exp}(\theta)}|u(\eta)|^n\eta^{n-1}\,\mathrm{d}\eta\geq\frac{(\operatorname{\exp}(n\theta)-1)\kappa_n}{2^n}f_0^n.
\end{align*}
Therefore, we conclude that \eqref{eq4.2} holds. In view of \eqref{eq4.1}, \eqref{eq4.2} and Lemma \ref{lem2.5}, we have
\begin{align*}
  C(n)\int_{\mathbb{R}^n\setminus\mathcal{W}_1}|u|^n\,\mathrm{d}x\geq \|\xi\|^n_{\operatorname{\exp}}\geq\mu(f_0)^n\sim \frac{\operatorname{\exp}(nf_0^{\frac{n}{n-1}})}{f_0^{\frac{n}{n-1}}}=\frac{\operatorname{\exp}(nu^{\frac{n}{n-1}}(1))}{u^{\frac{n}{n-1}}(1)}.
\end{align*}
The proof is thus complete.
\end{proof}
The preceding result allow us to proceed with the proof of Theorem \ref{thm1.3} as follows.
\subsection{{Proof of Theorem \ref{thm1.3}}}
Define the map $\Pi\colon[0,+\infty)\to[0,+\infty)$ by
$$\Pi(t):=\frac{\Phi(\lambda t^{\frac{n}{n-1}})}{(1+t)^{\frac{n}{n-1}}} \quad\text{for all}\quad t\geq 0\quad\text{and}\quad\lambda>0.$$
It is standard to check that $\Pi$ is continuous on $[0,+\infty)$. Now, we claim that $\Pi$ is increasing on $[0,+\infty)$. Indeed, one can notice that  
\begin{align*}
    \Pi^\prime(t)&=\frac{\big(\frac{n\lambda}{n-1}\big)t^{\frac{1}{n-1}}(1+t)^{\frac{n}{n-1}}\sum_{j= n-2}^{\infty}\frac{\lambda^j}{j!}t^{\frac{nj}{n-1}}-\big(\frac{n}{n-1}\big)(1+t)^{\frac{1}{n-1}}\sum_{j= n-2}^{\infty}\big(\frac{\lambda t^{\frac{n}{n-1}}}{j+1}\big)\frac{\lambda^j}{j!}t^{\frac{nj}{n-1}}}{(1+t)^{\frac{2n}{n-1}}}\\
    &\geq \frac{\big(\frac{n\lambda}{n-1}\big)t^{\frac{1}{n-1}}(1+t)^{\frac{n}{n-1}}\sum_{j= n-2}^{\infty}\frac{\lambda^j}{j!}t^{\frac{nj}{n-1}}-\big(\frac{n\lambda}{n-1}\big)(1+t)^{\frac{1}{n-1}}t^{\frac{n}{n-1}}\sum_{j= n-2}^{\infty}\frac{\lambda^j}{j!}t^{\frac{nj}{n-1}}}{(1+t)^{\frac{2n}{n-1}}}\\
    &=\frac{\big(\frac{n\lambda}{n-1}\big)t^{\frac{1}{n-1}}(1+t)^{\frac{1}{n-1}}\sum_{j= n-2}^{\infty}\frac{\lambda^j}{j!}t^{\frac{nj}{n-1}}}{(1+t)^{\frac{2n}{n-1}}}\geq0
\end{align*}
for all $t\in[0,+\infty)$. This shows the claim. In addition, since $\Pi(|u(x)|)=(\Pi \circ |u|)(x)$ for all $u\in W^{1,n}(\mathbb{R}^n)$, we obtain from \eqref{eq2.1} and Lemma \ref{lem2.91} that
$$ \int_{\mathbb{R}^n}\frac{\Phi\big(\lambda|u|^{\frac{n}{n-1}}\big)}{(1+|u|)^{\frac{n}{n-1}}}\,\mathrm{d}x=\int_{\mathbb{R}^n}\frac{\Phi\big(\lambda|u^\ast|^{\frac{n}{n-1}}\big)}{(1+|u^\ast|)^{\frac{n}{n-1}}}\,\mathrm{d}x,$$
and $$ \|u\|_n=\|u^\ast\|_n\quad\text{and}\qquad\|F(\nabla u^\ast)\|_n\leq \|F(\nabla u)\|_n. $$
By the preceding facts together with the density argument, it is enough to prove the desired inequality \eqref{eq1.14} for all nonnegative, Wulff radially symmetric, and nonincreasing function $u\in C_0^\infty(\mathbb{R}^n)\setminus\{0\}$ such that $\|F(\nabla u)\|_n\leq1$. Due to abuse notaion of radial symmetry, we write $u(x)=u(\eta)$, where $\eta=F^0(x)$. It is worth observing that the integral can be split as follows
$$ \int_{\mathbb{R}^n}\frac{\Phi\big(\lambda|u|^{\frac{n}{n-1}}\big)}{(1+|u|)^{\frac{n}{n-1}}}\,\mathrm{d}x=\bigg(\int_{\mathbb{R}^n\setminus\mathcal{W}_{r_0}}+\int_{\mathcal{W}_{r_0}}\bigg)\frac{\Phi\big(\lambda|u|^{\frac{n}{n-1}}\big)}{(1+|u|)^{\frac{n}{n-1}}}\,\mathrm{d}x=:J_1+J_2,$$
where $r_0=\inf\{\eta\colon~u(\eta)\leq 1\}\in[0,+\infty)$. It follows immediately that $u(\eta)\leq 1$ for all $\eta\geq r_0$ and thus, without loss of generality, we can take $r_0=r_0(u)>0$. Initially, we compute $J_1$. Hence, we obtain by using $u(x)\leq 1$ for all $x\in \mathbb{R}^n\setminus\mathcal{W}_{r_0}$ that
\begin{equation}
    \begin{aligned}\label{eq4.3}
       J_1&\leq \int_{\mathbb{R}^n\setminus\mathcal{W}_{r_0}}\Phi\big(\lambda|u|^{\frac{n}{n-1}}\big)\,\mathrm{d}x= \int_{\mathbb{R}^n\setminus\mathcal{W}_{r_0}}\sum_{j=n-1}^{\infty}\frac{\lambda^j}{j!}|u|^{\frac{nj}{n-1}}\,\mathrm{d}x\\
       &\leq \int_{\mathbb{R}^n\setminus\mathcal{W}_{r_0}}\sum_{j=n-1}^{\infty}\frac{\lambda^j}{j!}|u|^{n}\,\mathrm{d}x\leq \operatorname{\exp}(\lambda)\int_{\mathbb{R}^n\setminus\mathcal{W}_{r_0}}|u|^{n}\,\mathrm{d}x \leq C(n)\|u\|_n^n,
    \end{aligned}
\end{equation}
where the bound $0<\lambda\leq\lambda_n$ allows us to obtain the constant $C(n)>0$ depending only on $n$. Let $\theta\in(0,1)$ be an arbitrarily fixed real number. Then there exists $r_1=r_1(u)>0$ such that we obtain from Lemma \ref{lem2.90} that
\begin{align}\label{eq4.4}
 \int_{\mathcal{W}_{r_1}}F(\nabla u)^n\,\mathrm{d}x=n\kappa_n\int_{0}^{r_1}|u^\prime(\eta)|^n\eta^{n-1}\,\mathrm{d}\eta\leq \theta  
\end{align}
and \begin{align}\label{eq4.5}
 \int_{\mathbb{R}^n\setminus\mathcal{W}_{r_1}}F(\nabla u)^n\,\mathrm{d}x=n\kappa_n\int_{r_1}^{\infty}|u^\prime(\eta)|^n\eta^{n-1}\,\mathrm{d}\eta\leq 1-\theta. 
 \end{align}
 To complete the proof, it suffices to show that 
 \begin{align}\label{eq4.6}
J_2=\int_{\mathcal{W}_{r_0}}\frac{\Phi\big(\lambda|u|^{\frac{n}{n-1}}\big)}{(1+|u|)^{\frac{n}{n-1}}}\,\mathrm{d}x\leq C\|u\|_n^n,
 \end{align}
 for some constant $C>0$ independent on the choice of $u$. In virtue of H\"{o}lder's inequality and \eqref{eq4.4}, one has
 \begin{equation}\label{eq4.7}
     \begin{aligned}
         u(\eta_1)&=u(\eta_2)+\int_{\eta_1}^{\eta_2}-u^\prime(\eta)\,\mathrm{d}\eta\\
         &\leq u(\eta_2)+\bigg(\int_{\eta_1}^{\eta_2}|u^\prime(\eta)|^n\eta^{n-1}\,\mathrm{d}\eta\bigg)^{\frac{1}{n}}\bigg(\int_{\eta_1}^{\eta_2}\frac{1}{\eta}\bigg)^{\frac{n-1}{n}}\\
         &\leq u(\eta_2)+\bigg(\int_{0}^{r_1}|u^\prime(\eta)|^n\eta^{n-1}\,\mathrm{d}\eta\bigg)^{\frac{1}{n}}\bigg(\ln \bigg(\frac{\eta_2}{\eta_1}\bigg)\bigg)^{\frac{n-1}{n}}\\
         &\leq u(\eta_2)+\bigg(\frac{\theta}{n\kappa_n}\bigg)^{\frac{1}{n}}\bigg(\ln \bigg(\frac{\eta_2}{\eta_1}\bigg)\bigg)^{\frac{n-1}{n}}\quad\text{for all}\quad 0<\eta_1\leq\eta_2\leq r_1.
     \end{aligned}
 \end{equation}
 Likewise, once again by  H\"{o}lder's inequality and \eqref{eq4.5} leads to
 \begin{equation}\label{eq4.8}
         u(\eta_1)\leq u(\eta_2)+\bigg(\frac{1-\theta}{n\kappa_n}\bigg)^{\frac{1}{n}}\bigg(\ln \bigg(\frac{\eta_2}{\eta_1}\bigg)\bigg)^{\frac{n-1}{n}}\quad\text{for all}\quad 0<r_1\leq\eta_1\leq\eta_2.
 \end{equation}
To compute $J_2$, we have possibly two situations: $r_1\geq r_0$ and $r_1< r_0$, which are discussed below.

\textbf{Case-1:} Let $0<r_0\leq r_1$ be fixed. Then for all $0<\eta\leq r_0$, we obtain from \eqref{eq4.7} that
$$ u(\eta)\leq 1+\bigg(\frac{\theta}{n\kappa_n}\bigg)^{\frac{1}{n}}\bigg(\ln \bigg(\frac{r_0}{\eta}\bigg)\bigg)^{\frac{n-1}{n}}.$$
Pick $\delta=1-\theta^{\frac{1}{n-1}}$. It follows that $(1+\delta)\theta^{\frac{1}{n-1}}<1$. Recalling that $\lambda\leq\lambda_n$, one can deduce from the above inequality and \eqref{eq3.6} that
\begin{align*}
    J_2&\leq \int_{\mathcal{W}_{r_0}}\operatorname{\exp}\big(\lambda_n|u|^{\frac{n}{n-1}}\big)\,\mathrm{d}x=\int_{0}^{r_0}\operatorname{\exp}\big(\lambda_n|u(\eta)|^{\frac{n}{n-1}}\big)\int_{\partial\mathcal{W}_\eta}\frac{1}{|\nabla F^0(x)|}\,\mathrm{d}\sigma\,\mathrm{d}\eta\\
    &\leq n\kappa_n\int_{0}^{r_0}\operatorname{\exp}\Bigg( \lambda_n(1+\delta)\bigg(\frac{\theta}{n\kappa_n}\bigg)^{\frac{1}{n-1}}\ln \bigg(\frac{r_0}{\eta}\bigg)+\lambda_n C_\delta\Bigg)\eta^{n-1}\,\mathrm{d}\eta\\
    &=n\kappa_n\operatorname{\exp}(\lambda_n C_\delta)\int_{0}^{r_0}\bigg(\frac{r_0}{\eta}\bigg)^{n(1+\delta)\theta^{\frac{1}{n-1}}}\eta^{n-1}\,\mathrm{d}\eta\\
    &=\frac{\kappa_n\operatorname{\exp}(\lambda_n C_\delta)}{1-(1+\delta)\theta^{\frac{1}{n-1}}}~r_0^n\\
    &=\frac{\operatorname{\exp}(\lambda_n C_\delta)}{1-(1+\delta)\theta^{\frac{1}{n-1}}}\int_{0}^{r_0}\int_{\partial\mathcal{W}_\eta}\frac{1}{|\nabla F^0(x)|}\,\mathrm{d}\sigma\,\mathrm{d}\eta\\
    &\leq \frac{\operatorname{\exp}(\lambda_n C_\delta)}{1-(1+\delta)\theta^{\frac{1}{n-1}}}\int_{0}^{r_0}|u(\eta)|^n\int_{\partial\mathcal{W}_\eta}\frac{1}{|\nabla F^0(x)|}\,\mathrm{d}\sigma\,\mathrm{d}\eta\\
    &=\frac{\operatorname{\exp}(\lambda_n C_\delta)}{1-(1+\delta)\theta^{\frac{1}{n-1}}}\int_{\mathcal{W}_{r_0}}|u|^n\,\mathrm{d}x\\
    &\leq C(n,\delta)\|u\|_n^n,
\end{align*}
where $C(n,\delta)>0$ is a constant depending only on $n$ and $\delta$.

\textbf{Case-2:} Let $0<r_1< r_0$ be fixed. Moreover, we decompose the integral over $\mathcal{W}_{r_0}$ as follows
\begin{align*}
J_2=\int_{\mathcal{W}_{r_0}}\frac{\Phi\big(\lambda|u|^{\frac{n}{n-1}}\big)}{(1+|u|)^{\frac{n}{n-1}}}\,\mathrm{d}x=\bigg(\int_{\mathcal{W}_{r_0}\setminus\mathcal{W}_{r_1}}+\int_{\mathcal{W}_{r_1}}\bigg)\frac{\Phi\big(\lambda|u|^{\frac{n}{n-1}}\big)}{(1+|u|)^{\frac{n}{n-1}}}\,\mathrm{d}x=:J_{a}+J_{b}.    
\end{align*}
It follows from \eqref{eq4.8} that for all $r_1<\eta<r_0$, there holds
$$  u(\eta)\leq 1+\bigg(\frac{1-\theta}{n\kappa_n}\bigg)^{\frac{1}{n}}\bigg(\ln \bigg(\frac{r_0}{\eta}\bigg)\bigg)^{\frac{n-1}{n}}.$$
Take $\delta=1-(1-\theta)^{\frac{1}{n-1}}$ and thus, one sees that $(1+\delta)(1-\theta)^{\frac{1}{n-1}}<1$. In addition, by proceeding as in Case-$1$, with minor modifications, we obtain from the above inequality and \eqref{eq3.6} that
$$ J_a=\int_{\mathcal{W}_{r_0}\setminus\mathcal{W}_{r_1}}\frac{\Phi\big(\lambda|u|^{\frac{n}{n-1}}\big)}{(1+|u|)^{\frac{n}{n-1}}}\,\mathrm{d}x\leq C(n,\delta)\|u\|_n^n,$$
where $C(n,\delta)>0$ is a constant depending only on $n$ and $\delta$. To evaluate $J_b$, we define $w\colon [0,r_1]\to\mathbb{R}$ by
 $w(\eta)=u(\eta)-u(r_1)$ for all $ 0\leq \eta\leq r_1$. It follows that $w\in W^{1,n}_0(\mathcal{W}_{r_1})$ and $F(\nabla w)=F(\nabla u)$ in $\mathcal{W}_{r_1}$. From the preceding arguments, we infer that
 $$u(r_1)>1\quad\text{and}\quad\int_{\mathbb{R}^n\setminus \mathcal{W}_{r_1}}F(\nabla u)^n\,\mathrm{d}x\leq 1-\theta.$$
 It follows from Proposition \ref{prop4.1} that
 $$\frac{\operatorname{\exp\bigg(\frac{\lambda_n u^{\frac{n}{n-1}}(r_1)}{(1-\theta)^{\frac{1}{n-1}}}\bigg)}}{u^{\frac{n}{n-1}}(r_1)}\leq \frac{C(n,\theta)}{r_1^n}\int_{\mathbb{R}^n\setminus\mathcal{W}_{r_1}}|u|^n\,\mathrm{d}x,$$
 where $C(n,\theta)>0$ is a constant depending only on $n$ and $\theta$. In view of \eqref{eq3.6}, we choose $\delta=\big(\frac{1}{\theta}\big)^{\frac{1}{n-1}}-1$, which yields $C_\delta=(1-\theta)^{-\frac{1}{n-1}}$. Define the auxiliary function $\zeta\colon [0,r_1]\to\mathbb{R}$ by $\zeta(\eta)=(1+\delta)^{\frac{n-1}{n}}w(\eta)$ for all $0\leq\eta\leq r_1$. It follows immediately that
 $$\int_{\mathcal{W}_{r_1}}F(\nabla \zeta)^n\,\mathrm{d}x=(1+\delta)^{n-1}\int_{\mathcal{W}_{r_1}}F(\nabla w)^n\,\mathrm{d}x\leq (1+\delta)^{n-1}\theta=1.$$
 The above observations, combined with \eqref{eq3.6} and the anisotropic Moser--Trudinger inequality in bounded domains established by Wang--Xia \cite[Theorem 3.1]{Wang-Xia-2012}, yields
\begin{align*}
J_b&\leq\int_{\mathcal{W}_{r_1}}\frac{\operatorname{\exp}\big(\lambda_n|u|^{\frac{n}{n-1}}\big)}{|u|^{\frac{n}{n-1}}}\,\mathrm{d}x\\
&\leq\frac{1}{u^{\frac{n}{n-1}}(r_1)}\int_{\mathcal{W}_{r_1}}\operatorname{\exp}\big(\lambda_n(1+\delta)|w|^{\frac{n}{n-1}}+\lambda_nC_\delta u^{\frac{n}{n-1}}(r_1)\big)\,\mathrm{d}x\\
&=\frac{\operatorname{\exp\bigg(\frac{\lambda_n u^{\frac{n}{n-1}}(r_1)}{(1-\theta)^{\frac{1}{n-1}}}\bigg)}}{u^{\frac{n}{n-1}}(r_1)}\int_{\mathcal{W}_{r_1}}\operatorname{\exp}\big(\lambda_n(1+\delta)|w|^{\frac{n}{n-1}}\big)\,\mathrm{d}x\\
&\leq \frac{C(n,\theta)}{r_1^n}\|u\|_n^n\int_{\mathcal{W}_{r_1}}\operatorname{\exp}\big(\lambda_n|\zeta|^{\frac{n}{n-1}}\big)\,\mathrm{d}x\\
&\leq \tilde{C}(n,\delta)\|u\|_n^n,
\end{align*}
where $\tilde{C}(n,\delta)>0$ is a constant depending only on $n$ and $\delta$.

The validity of \eqref{eq4.6} follows from the above discussion. Hence, the desired inequality \eqref{eq1.14} is obtained by combining \eqref{eq4.3} and \eqref{eq4.6}. The proof will be complete once the sharpness of the exponent $\frac{n}{n-1}$ and the constant $\lambda_n$ are established. We begin by addressing the sharpness of the exponent $\frac{n}{n-1}$. To proceed, we make use of the Moser-type sequence of functions $\{v_k\}_k$ introduced earlier in the proof of Theorem \ref{thm1.1}, which satisfies 
\begin{align*}
\|F(\nabla v_k)\|^n_n=1 \qquad \text{and}\qquad \|v_k\|_n^n=O\Big(\frac{1}{\ln k}\Big)\quad\text{as}\quad k\to\infty.    
\end{align*}
Thus, we obtain
\begin{align*}
    \int_{\mathbb{R}^n}\frac{\Phi\big(\lambda_n|v_k|^{\frac{n}{n-1}}\big)}{(1+|v_k|)^q}\,\mathrm{d}x&\geq \int_{\{x\in\mathbb{R}^n\colon F^0(x)\leq sk^{-1}\}}\frac{\Phi\big(\lambda_n|v_k|^{\frac{n}{n-1}}\big)}{(1+|v_k|)^q}\,\mathrm{d}x\\
    &=\int_{0}^{sk^{-1}}\frac{\Phi\Big(\lambda_n\Big|(n\kappa_n)^{-\frac{1}{n}}(\ln k)^{\frac{n-1}{n}}\Big|^{\frac{n}{n-1}}\Big)}{\Big(1+\Big|(n\kappa_n)^{-\frac{1}{n}}(\ln k)^{\frac{n-1}{n}}\Big|\Big)^q}\int_{\partial\mathcal{W}_\eta}\frac{1}{|\nabla F^0(x)|}\,\mathrm{d}\sigma\,\mathrm{d}\eta\\
    &\gtrsim \frac{\operatorname{\exp}\Big(\lambda_n\Big|(n\kappa_n)^{-\frac{1}{n}}(\ln k)^{\frac{n-1}{n}}\Big|^{\frac{n}{n-1}}\Big)}{\Big(1+\Big|(n\kappa_n)^{-\frac{1}{n}}(\ln k)^{\frac{n-1}{n}}\Big|\Big)^q}\int_{0}^{sk^{-1}}\eta^{n-1}\,\mathrm{d}\eta\\
    &\gtrsim \frac{\operatorname{\exp}\Big(\lambda_n\Big|(n\kappa_n)^{-\frac{1}{n}}(\ln k)^{\frac{n-1}{n}}\Big|^{\frac{n}{n-1}}\Big)}{\Big(1+\Big|(n\kappa_n)^{-\frac{1}{n}}(\ln k)^{\frac{n-1}{n}}\Big|\Big)^q}~(sk^{-1})^n\\ &=\frac{\operatorname{\exp}(n \ln k)}{\Big(1+\Big|(n\kappa_n)^{-\frac{1}{n}}(\ln k)^{\frac{n-1}{n}}\Big|\Big)^q}~(sk^{-1})^n\\
    & \gtrsim (\ln k)^{-q\big(\frac{n-1}{n}\big)}.
\end{align*}
It follows that
$$\frac{1}{\|v_k\|_n^n}\int_{\mathbb{R}^n}\frac{\Phi\big(\lambda_n|v_k|^{\frac{n}{n-1}}\big)}{(1+|v_k|)^q}\,\mathrm{d}x\gtrsim (\ln k)^{1-q\big(\frac{n-1}{n}\big)}\to+\infty\quad\text{as}\quad k\to\infty,$$
where the hypothesis $q<\frac{n}{n-1}$ has been used. This shows that the exponent $\frac{n}{n-1}$ in the denominator cannot be improved. It remains to verify the sharpness of the constant $\lambda_n$. To this end, let $\lambda>\lambda_n$. Then there exists $\varepsilon_0>0$ such that $\lambda=(1+\varepsilon_0)\lambda_n$. By arguing as before, we also deduce that
\begin{align*}
    \int_{\mathbb{R}^n}\frac{\Phi\big(\lambda|v_k|^{\frac{n}{n-1}}\big)}{(1+|v_k|)^q}\,\mathrm{d}x&\geq \int_{\{x\in\mathbb{R}^n\colon F^0(x)\leq sk^{-1}\}}\frac{\Phi\big(\lambda|v_k|^{\frac{n}{n-1}}\big)}{(1+|v_k|)^q}\,\mathrm{d}x\\
    &\gtrsim \frac{\operatorname{\exp}\Big(\lambda\Big|(n\kappa_n)^{-\frac{1}{n}}(\ln k)^{\frac{n-1}{n}}\Big|^{\frac{n}{n-1}}\Big)}{\Big(1+\Big|(n\kappa_n)^{-\frac{1}{n}}(\ln k)^{\frac{n-1}{n}}\Big|\Big)^q}\int_{0}^{sk^{-1}}\eta^{n-1}\,\mathrm{d}\eta\\
    &\gtrsim \frac{\operatorname{\exp}\Big((1+\varepsilon_0)\lambda_n\Big|(n\kappa_n)^{-\frac{1}{n}}(\ln k)^{\frac{n-1}{n}}\Big|^{\frac{n}{n-1}}\Big)}{\Big(1+\Big|(n\kappa_n)^{-\frac{1}{n}}(\ln k)^{\frac{n-1}{n}}\Big|\Big)^q}~(sk^{-1})^n\\ &=\frac{\operatorname{\exp}(n (1+\varepsilon_0)\ln k)}{\Big(1+\Big|(n\kappa_n)^{-\frac{1}{n}}(\ln k)^{\frac{n-1}{n}}\Big|\Big)^q}~(sk^{-1})^n\\
    & \gtrsim k^{n\varepsilon_0}(\ln k)^{-q\big(\frac{n-1}{n}\big)}.
\end{align*}
This yields
\begin{align*}
    \frac{1}{\|v_k\|_n^n}\int_{\mathbb{R}^n}\frac{\Phi\big(\lambda|v_k|^{\frac{n}{n-1}}\big)}{(1+|v_k|)^q}&\,\mathrm{d}x\gtrsim k^{n\varepsilon_0}(\ln k)^{1-q\big(\frac{n-1}{n}\big)}\\
    &=\operatorname{\exp}\Bigg(\ln k\bigg(n\varepsilon_0+\bigg(1-\frac{q(n-1)}{n}\bigg)\frac{\ln(\ln k)}{\ln k}\bigg)\Bigg)\to+\infty 
\end{align*} 
as $k\to\infty$, where we have used $q\geq 0$. This completes the proof of the sharpness of $\lambda_n$ and consequently of Theorem \ref{thm1.3}.\qed

The following result is an immediate consequence of Theorem \ref{thm1.3} and establishes an anisotropic counterpart to the classical Adachi--Tanaka type subcritical Moser–Trudinger inequality.
\begin{corollary}\label{cor4.2}
    For all $n\geq 2$, $0<\lambda<\lambda_n$ and $u\in W^{1,n}(\mathbb{R}^n)$, there holds
\begin{align}\label{eq4.9}
\sup_{u\in W^{1,n}(\mathbb{R}^n),~\|F(\nabla u)\|_n\leq 1}\frac{1}{\|u\|_n^n}\int_{\mathbb{R}^n}\Phi\big(\lambda|u|^{\frac{n}{n-1}}\big)\,\mathrm{d}x<+\infty.
\end{align}
The inequality above is sharp, that is, for $\lambda\geq \lambda_n$, the supremum \eqref{eq4.9} is no longer finite.
\end{corollary}
\begin{proof}
It is standard to check that for all $0<\lambda<\lambda_n$, there exists a constant $C>0$ independent of $t$ such that
$$\Phi\big(\lambda|t|^{\frac{n}{n-1}}\big)\leq C\frac{\Phi\big(\lambda_n|t|^{\frac{n}{n-1}}\big)}{(1+|t|)^{\frac{n}{n-1}}}\quad\text{for all}\quad t\in\mathbb{R}.$$
Let $u\in W^{1,n}(\mathbb{R}^n)$ be such that $\|F(\nabla u)\|_n\leq 1$. The above estimate, when combined with Theorem \ref{thm1.3}, immediately yields
$$\int_{\mathbb{R}^n}\Phi\big(\lambda|u|^{\frac{n}{n-1}}\big)\,\mathrm{d}x\leq C \int_{\mathbb{R}^n}\frac{\Phi\big(\lambda_n|u|^{\frac{n}{n-1}}\big)}{(1+|u|)^{\frac{n}{n-1}}}\,\mathrm{d}x\leq C\|u\|_n^n.$$
The desired inequality \eqref{eq4.9} follows from the above estimate. Next, we need to the proof of sharpness of $\lambda$. For this purpose, we take $\lambda\geq \lambda_n$. It is enough to treat the case $\lambda=\lambda_n$, as the case $\lambda> \lambda_n$ can be treated in an entirely analogous manner. To this end, we recall the Moser-type sequence of functions $\{v_k\}_k$ constructed earlier in the proof of Theorem \ref{thm1.1}, which satisfies
\begin{align*}
\|F(\nabla v_k)\|^n_n=1 \qquad \text{and}\qquad \|v_k\|_n^n=O\Big(\frac{1}{\ln k}\Big)\quad\text{as}\quad k\to\infty.    
\end{align*}
Notice that 
\begin{align*}
    \int_{\mathbb{R}^n}\Phi\big(\lambda_n|v_k|^{\frac{n}{n-1}}\big)\,\mathrm{d}x &\geq \int_{\{x\in\mathbb{R}^n\colon F^0(x)\leq sk^{-1}\}}\Phi\big(\lambda_n|v_k|^{\frac{n}{n-1}}\big)\,\mathrm{d}x\\
    &=\int_{0}^{sk^{-1}}\Phi\Big(\lambda_n\Big|(n\kappa_n)^{-\frac{1}{n}}(\ln k)^{\frac{n-1}{n}}\Big|^{\frac{n}{n-1}}\Big)\int_{\partial\mathcal{W}_\eta}\frac{1}{|\nabla F^0(x)|}\,\mathrm{d}\sigma\,\mathrm{d}\eta\\
    &\gtrsim \operatorname{\exp}\Big(\lambda_n\Big|(n\kappa_n)^{-\frac{1}{n}}(\ln k)^{\frac{n-1}{n}}\Big|^{\frac{n}{n-1}}\Big)\int_{0}^{sk^{-1}}\eta^{n-1}\,\mathrm{d}\eta\\
    &\gtrsim \operatorname{\exp}\Big(\lambda_n\Big|(n\kappa_n)^{-\frac{1}{n}}(\ln k)^{\frac{n-1}{n}}\Big|^{\frac{n}{n-1}}\Big)(sk^{-1})^n\\
    &\gtrsim\operatorname{\exp}(n \ln k)(k^{-1})^n
    =1.
\end{align*}
This shows that
\begin{align*}
    \sup_{u\in W^{1,n}(\mathbb{R}^n),~\|F(\nabla u)\|_n\leq 1}\frac{1}{\|u\|_n^n}\int_{\mathbb{R}^n}\Phi\big(\lambda_n|u|^{\frac{n}{n-1}}\big)\,\mathrm{d}x&\geq \frac{1}{\|v_k\|_n^n}\int_{\mathbb{R}^n}\Phi\big(\lambda_n|v_k|^{\frac{n}{n-1}}\big)\,\mathrm{d}x\\
    &\gtrsim \ln k\to+\infty\quad\text{as}\quad k\to\infty.
\end{align*}
Hence, the proof is completed.
\end{proof}
The result below will serve as a useful tool in the proof of Theorem \ref{thm1.4}.
\begin{corollary}\label{cor4.3}
    For all $n\geq 2$, $r\geq n$ and $u\in W^{1,n}(\mathbb{R}^n)$, the following inequality holds
\begin{align}\label{eq4.10}
\|u\|_r\leq c(n,r)\|u\|_n^{\frac{n}{r}} \|F(\nabla u)\|_n^{1-\frac{n}{r}},
\end{align}
where $c(n,r)>0$ is a constant depending on $n$ and $r$. Moreover, there exists $\tau>0$ such that $c(n,r)$ has the asymptotic behavior $c(n,r)\simeq \tau r^{\frac{1}{n^\prime}}$ as $r\to\infty$, where $n^\prime=\frac{n}{n-1}$.
\end{corollary}
\begin{proof}
    Denote $\lambda_0$ as the supremum of all $\lambda>0$ satisfying the inequality stated below
    $$\int_{\mathbb{R}^n}\Phi\big(\lambda|u|^{\frac{n}{n-1}}\big)\,\mathrm{d}x\leq C\|u\|_n^n $$
   for all $u\in W^{1,n}(\mathbb{R}^n)$ with $\|F(\nabla u)\|_n\leq 1$, where  $C>0$ is a constant depending only on $n$. It follows from Corollary \ref{cor4.2} that $\lambda_0=\lambda_n$. In addition, the above inequality ensures that for all $\lambda\in(0,\lambda_0)$, there exists a constant $C>0$ such that for all $u\in W^{1,n}(\mathbb{R}^n)\setminus\{0\}$ and all integers $\ell\geq n-1$, there holds
$$ C\Bigg(\frac{\|u\|_n}{\|F(\nabla u)\|_n}\Bigg)^n \geq \int_{\mathbb{R}^n}\Phi\Bigg(\lambda\Bigg|\frac{u}{\|F(\nabla u)\|_n}\Bigg|^{\frac{n}{n-1}}\Bigg)\,\mathrm{d}x\geq\frac{\lambda^\ell} {\ell!}\Bigg(\frac{\|u\|_{n^\prime\ell}}{\|F(\nabla u)\|_n}\Bigg)^{n^\prime \ell}.$$
This gives that
\begin{align}\label{eq4.11}
  \|u\|_{n^\prime\ell}\leq \bigg(\frac{\ell! C}{\lambda^\ell}\bigg)^{\frac{1}{n^\prime\ell}} \|u\|_n^{\frac{n}{n^\prime\ell}} \|F(\nabla u)\|_n^{1-\frac{n}{n^\prime\ell}} 
\end{align}
for all $u\in W^{1,n}(\mathbb{R}^n)$ and all integers $\ell\geq n-1$. It follows that for all $r\geq n$, we can find an integer $\ell\geq n-1$ such that $n^\prime\ell\leq r<n^\prime(\ell+1)$. Thus, without loss of generality, there exists $\vartheta\in(0,1]$ such that $\frac{1}{r}=\frac{\vartheta}{n^\prime\ell}+\frac{1-\vartheta}{n^\prime(\ell+1)}$. Now, by utilizing the  interpolation inequality, one has
\begin{align}\label{eq4.12}
  \|u\|_r\leq \|u\|_{n^\prime\ell}^\vartheta \|u\|_{n^\prime(\ell+1)}^{1-\vartheta}.
\end{align}
Hence, we deduce from \eqref{eq4.11} and \eqref{eq4.12} that for all $u\in W^{1,n}(\mathbb{R}^n)$ and $r\in[n,+\infty)$, there holds
\begin{align}\label{eq4.13}
  \|u\|_r\leq ((\ell+1)!)^{\frac{1}{r}}\lambda^{-\frac{1}{n^\prime}} C^{\frac{1}{r}}\|u\|_n^{\frac{n}{r}} \|F(\nabla u)\|_n^{1-\frac{n}{r}}.
\end{align}
It is standard to see that $$((\ell+1)!)^{\frac{1}{r}}\leq \bigg(\Gamma\Big(\frac{r}{n^\prime}+2\Big)\bigg)^{\frac{1}{r}}.$$  The desired inequality \eqref{eq4.10} follows directly from \eqref{eq4.13} by means of the estimate $$c(n,r)\approx \lambda^{-\frac{1}{n^\prime}} C^{\frac{1}{r}}\bigg(\Gamma\Big(\frac{r}{n^\prime}+2\Big)\bigg)^{\frac{1}{r}}.$$
Next, we claim that there exists $\tau>0$ such that $c(n,r)\simeq \tau r^{\frac{1}{n^\prime}}$ as $r\to\infty$. In other words, for all $r_0\geq n$, there exists $r\geq r_0$ such that 
\begin{align}\label{eq4.14}
\|u\|_r\leq \tau r^{\frac{1}{n^\prime}}\|u\|_n^{\frac{n}{r}} \|F(\nabla u)\|_n^{1-\frac{n}{r}}
\end{align}
for all $u\in W^{1,n}(\mathbb{R}^n)$. Indeed, by employing Stirling’s asymptotic formula, we have
$$ \bigg(\Gamma\Big(\frac{r}{n^\prime}+2\Big)\bigg)^{\frac{1}{r}}=(1+o_r(1))\bigg(\frac{r}{n^\prime\operatorname{\exp}(1)}\bigg)^{\frac{1}{n^\prime}}\quad\text{as}\quad r\to\infty.$$
This yields for all $u\in W^{1,n}(\mathbb{R}^n)$ that
$$\|u\|_r\leq (1+o_r(1))\bigg(\frac{r}{\lambda n^\prime\operatorname{\exp}(1)}\bigg)^{\frac{1}{n^\prime}} \|u\|_n^{\frac{n}{r}} \|F(\nabla u)\|_n^{1-\frac{n}{r}}\quad\text{as}\quad r\to\infty.$$
Observe that the claim follows from the preceding inequality, which in turn establishes \eqref{eq4.14}. Hence, the desired result is accomplished, completing the proof.
\end{proof}
The proof of Theorem \ref{thm1.4} can now be carried out using the results established above, which brings this section to a close.
\subsection{{Proof of Theorem \ref{thm1.4}}}
For the sake of argument, assume that Theorem \ref{thm1.3} holds. To complete the proof, it is enough to prove Theorem \ref{thm1.1} for $\beta=0$ and for all $0<\lambda\leq\lambda_n$, that is, there holds
\begin{align}\label{eq4.15}
    \sup_{u\in W^{1,n}(\mathbb{R}^n),\,\|u\|\leq 1}\int_{\mathbb{R}^n}\Phi\big(\lambda|u|^{\frac{n}{n-1}}\big)\,\mathrm{d}x<+\infty.
\end{align}
Next, we take $u\in W^{1,n}(\mathbb{R}^n),~\|u\|\leq 1$  and $\rho\in(0,1)$ satisfying $\|u\|_n^n=\rho$ and $\|F(\nabla u)\|_n^n\leq 1-\rho$. This leads to two possible situations, which we discuss separately below.

\textbf{Case-1:} In this case, we take $\rho\geq \frac{n-1}{n}$. Define the auxiliary function $w\colon\mathbb{R}^n \to \mathbb{R}$ by $w(x)=n^{\frac{1}{n}}u(x)$ for all $x\in\mathbb{R}^n$. It follows that $w\in W^{1,n}(\mathbb{R}^n)$ such that $\|w\|_n^n\leq n$ and $\|F(\nabla w)\|_n^n\leq 1$. Applying Corollary \ref{cor4.2}, we obtain for all $0<\lambda<\lambda_n$ that
\begin{align*}
 \int_{\mathbb{R}^n}\Phi\big(\lambda n^{\frac{1}{n-1}}|u|^{\frac{n}{n-1}}\big)\,\mathrm{d}x= \int_{\mathbb{R}^n}\Phi\big(\lambda |w|^{\frac{n}{n-1}}\big)\,\mathrm{d}x\leq C\|w\|_n^n \leq C, 
\end{align*}
where $C>0$ is a constant varies from step to step and depending only on  $\lambda$  and $n$. In addition, by assuming $\lambda n^{\frac{1}{n-1}}=\lambda_n$, one sees that $\lambda<\lambda_n$ and thus, the validity of \eqref{eq4.15} follows.

\textbf{Case-2:} In this case, we take $\rho<\frac{n-1}{n}$. Define the set $S:=\{x\in\mathbb{R}^n\colon~|u(x)|\geq 1\}$. Now, one can divide the whole integral as follows
$$\int_{\mathbb{R}^n}\Phi\big(\lambda|u|^{\frac{n}{n-1}}\big)\,\mathrm{d}x=\bigg(\int_{\mathbb{R}^n\setminus S}+\int_{S}\bigg)\Phi\big(\lambda|u|^{\frac{n}{n-1}}\big)\,\mathrm{d}x=:J_1+J_2.$$
By using $|u(x)|<1$ for all $x\in\mathbb{R}^n\setminus S$, we get
\begin{align*}
  J_1&= \int_{\mathbb{R}^n\setminus S}\sum_{j=n-1}^{\infty}\frac{\lambda^j}{j!}|u|^{\frac{nj}{n-1}}\,\mathrm{d}x\leq \int_{\mathbb{R}^n\setminus S}\sum_{j=n-1}^{\infty}\frac{\lambda^j}{j!}|u|^{n}\,\mathrm{d}x\\
  &\leq \operatorname{\exp}(\lambda)\int_{\mathbb{R}^n\setminus S}|u|^{n}\,\mathrm{d}x\leq \operatorname{\exp}(\lambda)\rho\leq C(\lambda,n),
\end{align*}
where $C(\lambda,n)>0$ is a constant depending on $\lambda$ and $n$. To evaluate $J_2$, choose $r=\frac{n-1}{n-1-\rho}>1$ such that $\frac{1}{r}+\frac{1}{r^\prime}=1$. It follows that $\frac{n^\prime}{r-1}>n$. Thus, by using Corollary \ref{cor4.3}, one has
$$\|u\|_{\frac{n^\prime}{r-1}}\lesssim\bigg(\frac{n^\prime}{r-1} \bigg)^{\frac{1}{n^\prime}}\|u\|_n^{(n-1)(r-1)}$$
for all $u\in W^{1,n}(\mathbb{R}^n)$ with $\|F(\nabla u)\|_n\leq 1$. This together with H\"{o}lder inequality and Lemma \ref{lem2.4} yields
\begin{align*}
    J_2&\leq \bigg(\int_S(1+|u|)^{\frac{n^\prime}{r-1}}\,\mathrm{d}x\bigg)^{\frac{1}{r^\prime}}\bigg(\int_S\frac{\Phi\big(\lambda r|u|^{\frac{n}{n-1}}\big)}{(1+|u|)^{\frac{n}{n-1}}}\,\mathrm{d}x\bigg)^{\frac{1}{r}}\\
    &\leq 2^{\frac{n^\prime}{r}}\|u\|_{\frac{n^\prime}{r-1}}^{\frac{n^\prime}{r}}\bigg(\int_S\frac{\Phi\big(\lambda r|u|^{\frac{n}{n-1}}\big)}{(1+|u|)^{\frac{n}{n-1}}}\,\mathrm{d}x\bigg)^{\frac{1}{r}}\\
    &\lesssim 2^{\frac{n^\prime}{r}}\bigg(\frac{n^\prime}{r-1} \bigg)^{\frac{1}{r}}\|u\|_{n}^{\frac{n}{r^\prime}}\bigg(\int_{\mathbb{R}^n}\frac{\Phi\big(\lambda r|u|^{\frac{n}{n-1}}\big)}{(1+|u|)^{\frac{n}{n-1}}}\,\mathrm{d}x\bigg)^{\frac{1}{r}}.
\end{align*}
Define an auxiliary function $w\colon\mathbb{R}^n\to\mathbb{R}$ by $w(x)=r^{\frac{1}{n^\prime}}u(x)$ for all $x\in \mathbb{R}^n$. Then, we can see that $$\|w\|_n=r^{\frac{1}{n^\prime}}\|u\|_n$$ and 
\begin{align*}
    \|F(\nabla w)\|_n&=r^{\frac{1}{n^\prime}}\|F(\nabla u)\|_n\leq \bigg(\frac{n-1}{n-1-\rho}\bigg)^{\frac{1}{n^\prime}}(1-\rho)^{\frac{1}{n}}\\
    &\leq \Big(1-\frac{\rho}{n-1}\Big)^{-\frac{1}{n^\prime}}\Big(1-\frac{\rho}{n-1}\Big)^{\frac{1}{n^\prime}}=1,
\end{align*}
where we have used $(1-x)^q\leq 1-q x$ for all $x\in[0,1]$ and $q\in[0,1]$. It follows that $w\in W^{1,n}(\mathbb{R}^n)$ and $\|F(\nabla w)\|_n\leq 1$. Hence, we obtain from Theorem \ref{thm1.3} that
\begin{align*}
\bigg(\int_{\mathbb{R}^n}\frac{\Phi\big(\lambda r|u|^{\frac{n}{n-1}}\big)}{(1+|u|)^{\frac{n}{n-1}}}\,\mathrm{d}x\bigg)^{\frac{1}{r}}\leq \bigg(r\int_{\mathbb{R}^n}\frac{\Phi\big(\lambda |w|^{\frac{n}{n-1}}\big)}{(1+|w|)^{\frac{n}{n-1}}}\,\mathrm{d}x\bigg)^{\frac{1}{r}}\lesssim(r\|w\|_n^n)^{\frac{1}{r}}=r^{\frac{n}{r}} \|u\|_n^{\frac{n}{r}}.
\end{align*}
By combining the above results, we obtain
\begin{align*}
    J_2&\lesssim 2^{\frac{n^\prime}{r}}\bigg(\frac{n^\prime}{r-1} \bigg)^{\frac{1}{r}}r^{\frac{n}{r}}\|u\|_{n}^n= (2^{n^\prime}n^\prime)^{\frac{1}{r}} \bigg(\frac{r}{r-1} \bigg)^{\frac{1}{r}}r^{\frac{n-1}{r}}\rho\\
    &\lesssim 2^{n^\prime}n^\prime\bigg(\frac{n-1}{\rho}\bigg)^{1-\frac{\rho}{n-1}}\bigg(1-\frac{\rho}{n-1}\bigg)^{-(n-1-\rho)}\rho\\
    &\lesssim 2^{n^\prime}n^\prime(n-1)\rho^{\frac{\rho}{n-1}}\bigg(1-\frac{\rho}{n-1}\bigg)^{-(n-1-\rho)}\\
    &\lesssim2^{n^\prime}n\bigg(1-\frac{1}{n}\bigg)^{-(n-1)},
\end{align*}
where we apply $0<\rho<\frac{n-1}{n}$ and $r>1$. Hence, by combining the bounds for $J_1$ and $J_2$, we deduce that \eqref{eq4.15} is satisfied. It  follows from both cases that the desired inequality \eqref{eq4.15} holds. Thus, the proof is well-established.\qed

\section{The Anisotropic Singular Moser--Trudinger Inequality with Exact Growth-I and Its Consequence}\label{sec5}
The main objective of this section is to introduce anisotropic quasi-conformal-type transformations and apply them, together with Theorem \ref{thm1.3}, to establish Theorem \ref{thm1.5}. In the sequel, we establish certain relations between Theorems \ref{thm1.1} and  \ref{thm1.5} arising from Theorem \ref{thm1.5}. The discussion starts with the following definition.
\begin{definition}[Quasi-Conformal-Type Transformations]
Let $T\colon\mathbb{R}^n\to \mathbb{R}^n$ be a vector-valued map defined by
$$T(x):=F^0(x)^{\gamma-1}x\quad\text{for all}\quad x\in \mathbb{R}^n\quad\text{and}\quad\gamma>1.$$
The mapping $T$ is referred to as an anisotropic quasi-conformal-type transformation. Such transformations are particularly effective in the analysis of anisotropic Moser--Trudinger inequalities with singular weights and provide an effective replacement for the anisotropic P\'{o}lya--Szeg\"{o} inequality in the presence of power-type weights.
\end{definition}
It is not difficult to verify that the Jacobian matrix of the map $T$ is given by
$$J_T:=F^0(x)^{\gamma-1}I_{n\times n}+M, $$
where $I_{n\times n}$ is the identity matrix of order $n\times n$ and the matrix $M$ is defined by
$$ M:=(\gamma-1)F^0(x)^{\gamma-2}\bigg(x_i\frac{\partial F^0(x)}{\partial x_j}\bigg)_{1\leq i,j\leq n}.$$
Direct computation shows that $\textnormal{rank}(M)=1$. In addition, its trace can be evaluated as follows
\begin{align*}
\textnormal{tr}(M)&=(\gamma-1)F^0(x)^{\gamma-2}\sum_{i=1}^{n}x_i\frac{\partial F^0(x)}{\partial x_i}=(\gamma-1)F^0(x)^{\gamma-2}\langle{x,\nabla F^0(x)}\rangle\\
&=(\gamma-1)F^0(x)^{\gamma-2}F^0(x)=(\gamma-1)F^0(x)^{\gamma-1},
\end{align*}
where we have used Lemma \ref{lem2.90}. The eigenvalues of $J_T$ are $\gamma F^0(x)^{\gamma-1}$ with multiplicity $1$ and $F^0(x)^{\gamma-1}$ with multiplicity $n-1$. Thus, the determinant of $J_T$ is given by
$$ \textnormal{det}(J_T)=\gamma F^0(x)^{n(\gamma-1)}.$$
For all $r>1$ and given a function $u\colon\mathbb{R}^n\to \mathbb{R}$, we define 
\begin{align}\label{eq5.1}
 w(x):=\bigg(\frac{1}{\gamma}\bigg)^\frac{r-1}{r} u\big(T(x)\big)=\bigg(\frac{1}{\gamma}\bigg)^\frac{r-1}{r} u\big(F^0(x)^{\gamma-1}x\big)\quad\text{for all}\quad x\in \mathbb{R}^n. \end{align}
Set $y=T(x)$. Now, we are interested in calculating $F(\nabla w)$. By employing multivariate chain rule, we have
$$ \begin{pmatrix}
\frac{\partial w}{\partial x_1}\\   
\frac{\partial w}{\partial x_2}\\ 
\vdots\\
\frac{\partial w}{\partial x_n}
\end{pmatrix}=\nabla w(x)=\bigg(\frac{1}{\gamma}\bigg)^\frac{r-1}{r} \nabla u\big(F^0(x)^{\gamma-1}x\big)=\bigg(\frac{1}{\gamma}\bigg)^\frac{r-1}{r} J_T^t\begin{pmatrix}
\frac{\partial u}{\partial y_1}\\  
\frac{\partial u}{\partial y_2}\\
\vdots\\
\frac{\partial u}{\partial y_n}
\end{pmatrix},$$
where $J_T^t$ denotes the transpose of the Jacobian matrix $J_T$. It follows that for all $i=1,2,\cdots,n$, one has
\begin{align*}
\frac{\partial w}{\partial x_i} = \bigg(\frac{1}{\gamma}\bigg)^\frac{r-1}{r}\bigg(F^0(x)^{\gamma-1}\frac{\partial u}{\partial y_i}+(\gamma-1) F^0(x)^{\gamma-2}\frac{\partial F^0(x)}{\partial x_i}\langle{x,\nabla u}\rangle\bigg).  
\end{align*}
In particular, there holds
\begin{align*}
\nabla w = \bigg(\frac{1}{\gamma}\bigg)^\frac{r-1}{r}\Big(F^0(x)^{\gamma-1}\nabla u+(\gamma-1) F^0(x)^{\gamma-2}\nabla F^0(x)\langle{x,\nabla u}\rangle\Big).  
\end{align*}
Recalling the anisotropic analogue of the classical Cauchy--Schwarz inequality $|\langle{x,\xi}\rangle|\leq F^0(x)F(\xi)$ for all $x,\xi\in\mathbb{R}^n$, as established in Alvino et al. \cite{Alvino-Ferone-Trombetti-Lions-1997}, we deduce from Lemma \ref{lem2.90} and \eqref{eq2.333} that
\begin{align*}
F(\nabla w)&= \bigg(\frac{1}{\gamma}\bigg)^\frac{r-1}{r}F \Big(F^0(x)^{\gamma-1}\nabla u+(\gamma-1) F^0(x)^{\gamma-2}\nabla F^0(x)\langle{x,\nabla u}\rangle\Big)\\
&\leq \bigg(\frac{1}{\gamma}\bigg)^\frac{r-1}{r} \Big(F^0(x)^{\gamma-1}F(\nabla u)+(\gamma-1) F^0(x)^{\gamma-2}|\langle{x,\nabla u}\rangle|\Big)\\
& \leq \bigg(\frac{1}{\gamma}\bigg)^\frac{r-1}{r} \Big(F^0(x)^{\gamma-1}F(\nabla u)+(\gamma-1) F^0(x)^{\gamma-2}F^0(x) F(\nabla u)\Big)\\
&=\gamma^{\frac{1}{r}} F^0(x)^{\gamma-1}F(\nabla u).
\end{align*}
Making the substitution $y=T(x)$, we obtain
\begin{align}\label{eq5.2}
 F^0(x)=F^0(y)^{\frac{1}{\gamma}}\quad\text{and}\quad \mathrm{d} x=\frac{1}{\gamma F^0(y)^{\frac{n(\gamma-1)}{\gamma}}} \,\mathrm{d}y.   
\end{align}
This shows that
\begin{equation}\label{eq5.3}
    \int_{\mathbb{R}^n}|w(x)|^r\,\mathrm{d}x=\left(\frac{1}{\gamma}\right)^r\int_{\mathbb{R}^n}\frac{|u(y)|^r}{F^0(y)^{\frac{n(\gamma-1)}{\gamma}}}\,\mathrm{d}y
\end{equation}
and \begin{align}\label{eq5.333}
     \int_{\mathbb{R}^n}F(\nabla w)^r\,\mathrm{d}x\leq \int_{\mathbb{R}^n}\frac{F(\nabla u)^r}{F^0(y)^{\frac{(n-r)(\gamma-1)}{\gamma}}}\,\mathrm{d}y.
\end{align}

In what follows, we establish an elementary result that will be useful in the subsequent analysis.
\begin{lemma}\label{lem5.1}
    Let $g\colon\mathbb{R}\to \mathbb{R}$ be a continuous function. Then for all $n\geq2$, $\gamma>1$, $r>1$ and $0\leq \beta<n$, there holds
    $$\int_{\mathbb{R}^n} \frac{g\big(\big(\frac{1}{\gamma}\big)^\frac{r-1}{r}u(y)\big)}{F^0(y)^\beta}\,\mathrm{d}y=\gamma\int_{\mathbb{R}^n} \frac{g(w(x))}{F^0(x)^{n-(n-\beta)\gamma}}\,\mathrm{d}x,$$
    where the function $w$ is defined in \eqref{eq5.1}. In particular, for $g(\cdot)=|\cdot|^p$ for any $p>1$, then
    $$ u\in L^p\big(\mathbb{R}^n,F^0(x)^{-\beta}\,\mathrm{d}x\big)\quad\text{if and only if}\quad w\in L^p\big(\mathbb{R}^n,F^0(x)^{-(n-(n-\beta)\gamma)}\,\mathrm{d}x\big).$$
\end{lemma}
\begin{proof}
    Let the given hypotheses be fulfilled. Defining $y=F^0(x)^{\gamma-1}x$, we obtain \eqref{eq5.2}. It follows that
    \begin{align*}
      \int_{\mathbb{R}^n} \frac{g(w(x))}{F^0(x)^{n-(n-\beta)\gamma}}\,\mathrm{d}x&= \int_{\mathbb{R}^n} \frac{g\big(\big(\frac{1}{\gamma}\big)^\frac{r-1}{r}u(y)\big)}{F^0(y)^{\frac{n-(n-\beta)\gamma}{\gamma}}} \bigg(\frac{1}{\gamma F^0(y)^{\frac{n(\gamma-1)}{\gamma}}}\bigg)\,\mathrm{d} y\\
      &=\frac{1}{\gamma}\int_{\mathbb{R}^n} \frac{g\big(\big(\frac{1}{\gamma}\big)^\frac{r-1}{r}u(y)\big)}{F^0(y)^\beta}\,\mathrm{d}y.
    \end{align*}
Thus, the desired result follows, completing the proof. 
\end{proof}
Next, we turn to the proof of Theorem \ref{thm1.5}, which proceeds as follows.
\subsection{{Proof of Theorem \ref{thm1.5}}} The assertion is already known when $\beta=0$, since in this case Theorem \ref{thm1.5} reduces to Theorem \ref{thm1.3}. It only remains to prove Theorem \ref{thm1.5} for $\beta\in(0,n)$. To this end, let $\beta\in(0,n)$, $0<\lambda\leq \lambda_{\beta,n}$, and $u\in W^{1,n}(\mathbb{R}^n)$ be such that $\|F(\nabla u)\|_n\leq 1$. In addition, we reconsider the same function  $w$ defined in \eqref{eq5.1} with $r=n$ and $\gamma=\frac{n}{n-\beta}$ and thus, we derive from \eqref{eq5.3} and \eqref{eq5.333} that
\begin{align}\label{eq5.4}
\int_{\mathbb{R}^n}|w(x)|^n\,\mathrm{d}x=\bigg(\frac{n-\beta}{n}\bigg)^n\int_{\mathbb{R}^n}\frac{|u(y)|^n}{F^0(y)^\beta}\,\mathrm{d}y,\quad \int_{\mathbb{R}^n}F(\nabla w)^n\,\mathrm{d}x\leq \int_{\mathbb{R}^n}F(\nabla u)^n\,\mathrm{d}y.    
\end{align}
It follows from \eqref{eq5.4} and Lemma \ref{lem2.6} that $w\in W^{1,n}(\mathbb{R}^n)$ and $\|F(\nabla w)\|_n\leq 1$. Hence, by combining \eqref{eq5.1}, \eqref{eq5.2}, and \eqref{eq5.4} with the monotone convergence theorem and Theorem \ref{thm1.3}, yields
\begin{align*}
\int_{\mathbb{R}^n}&\frac{\Phi\big(\lambda|u(y)|^{\frac{n}{n-1}}\big)}{(1+|u(y)|)^{\frac{n}{n-1}}F^0(y)^\beta}\,\mathrm{d}y\\&=\int_{\mathbb{R}^n}\sum_{j=n-1}^{\infty} \frac{\lambda^j}{j!}~\frac{|u(y)|^{\frac{nj}{n-1}}}{(1+|u(y)|)^{\frac{n}{n-1}}F^0(y)^\beta}\,\mathrm{d}y\\
& =\sum_{j=n-1}^{\infty} \int_{\mathbb{R}^n} \frac{\lambda^j}{j!} ~\frac{\big|\big(\frac{n}{n-\beta}\big)^{\frac{n-1}{n}}w(x)\big|^{\frac{nj}{n-1}}}{\Big(1+\big(\frac{n}{n-\beta}\big)^{\frac{n-1}{n}}|w(x)|\Big)^{\frac{n}{n-1}}F^0(x)^{\frac{n\beta}{n-\beta}}}~\Big(\frac{n}{n-\beta}\Big)F^0(x)^{\frac{n\beta}{n-\beta}}\,\mathrm{d}x\\
&=\sum_{j=n-1}^{\infty} \int_{\mathbb{R}^n} \frac{\lambda^j}{j!} \Big(\frac{n}{n-\beta}\Big)^{j+1}~\frac{|w(x)|^{\frac{nj}{n-1}}}{\Big(1+\big(\frac{n}{n-\beta}\big)^{\frac{n-1}{n}}|w(x)|\Big)^{\frac{n}{n-1}}}\,\mathrm{d}x\\
&\leq \sum_{j=n-1}^{\infty} \int_{\mathbb{R}^n} \frac{\lambda^j}{j!} \Big(\frac{n}{n-\beta}\Big)^{j+1}~\frac{|w(x)|^{\frac{nj}{n-1}}}{(1+|w(x)|)^{\frac{n}{n-1}}}\,\mathrm{d}x\\
&=\Big(\frac{n}{n-\beta}\Big)\sum_{j=n-1}^{\infty} \int_{\mathbb{R}^n} \frac{\big(\lambda\big(1-\frac{\beta}{n}\big)^{-1}\big)^j}{j!} \frac{|w(x)|^{\frac{nj}{n-1}}}{(1+|w(x)|)^{\frac{n}{n-1}}}\,\mathrm{d}x\\
&=\Big(\frac{n}{n-\beta}\Big)\int_{\mathbb{R}^n}\frac{\Phi\big(\lambda\big(1-\frac{\beta}{n}\big)^{-1}|w(x)|^{\frac{n}{n-1}}\big)}{(1+|w(x)|)^{\frac{n}{n-1}}}\,\mathrm{d}x\\
&\lesssim \Big(\frac{n}{n-\beta}\Big)\int_{\mathbb{R}^n}|w(x)|^n\,\mathrm{d}x=\Big(\frac{n}{n-\beta}\Big)^{1-n}\int_{\mathbb{R}^n}\frac{|u(y)|^n}{F^0(y)^\beta}\,\mathrm{d}y
\end{align*}
for all $u\in W^{1,n}(\mathbb{R}^n)$ with $\|F(\nabla u)\|_n\leq 1$. Thus, the desired inequality, namely \eqref{eq1.15}, now follows. To finish the proof, it remains to demonstrate sharpness of the exponent $\frac{n}{n-1}$ and the constant $\lambda_{\beta,n}$. To show the exponent $\frac{n}{n-1}$ is sharp, we revisit the Moser-type sequence of functions $\{v_k\}_k$ that used in the proof of Theorem \ref{thm1.1}. It enjoys the the following properties 
\begin{align*}
\|F(\nabla v_k)\|^n_n=1 \qquad \text{and}\qquad \int_{\mathbb{R}^n}\frac{|v_k|^n}{F^0(x)^\beta}\,\mathrm{d}x=O\Big(\frac{1}{\ln k}\Big)\quad\text{as}\quad k\to\infty.    
\end{align*}
It is worth observing that
\begin{align*}
&\int_{\mathbb{R}^n}\frac{\Phi\big(\lambda_{\beta,n}|v_k|^{\frac{n}{n-1}}\big)}{(1+|v_k|)^q F^0(x)^\beta}\mathrm{d}x\\
    &\geq \int_{\{x\in\mathbb{R}^n\colon F^0(x)\leq sk^{-1}\}}\frac{\Phi\big(\lambda_{\beta,n}|v_k|^{\frac{n}{n-1}}\big)}{(1+|v_k|)^qF^0(x)^\beta}\,\mathrm{d}x\\
    &=\frac{\Phi\Big(\lambda_{n}\big(1-\frac{\beta}{n}\big)\Big|(n\kappa_n)^{-\frac{1}{n}}(\ln k)^{\frac{n-1}{n}}\Big|^{\frac{n}{n-1}}\Big)}{\Big(1+\Big|(n\kappa_n)^{-\frac{1}{n}}(\ln k)^{\frac{n-1}{n}}\Big|\Big)^q}\int_{0}^{sk^{-1}}\frac{1}{\eta^\beta}\int_{\partial\mathcal{W}_\eta}\frac{1}{|\nabla F^0(x)|}\mathrm{d}\sigma\mathrm{d}\eta\\
    &\gtrsim \frac{\operatorname{\exp}\Big(\lambda_{n}\big(1-\frac{\beta}{n}\big)\Big|(n\kappa_n)^{-\frac{1}{n}}(\ln k)^{\frac{n-1}{n}}\Big|^{\frac{n}{n-1}}\Big)}{\Big(1+\Big|(n\kappa_n)^{-\frac{1}{n}}(\ln k)^{\frac{n-1}{n}}\Big|\Big)^q}\int_{0}^{sk^{-1}}\eta^{n-\beta-1}\,\mathrm{d}\eta\\
    &\gtrsim \frac{\operatorname{\exp}\Big(\lambda_n\big(1-\frac{\beta}{n}\big)\Big|(n\kappa_n)^{-\frac{1}{n}}(\ln k)^{\frac{n-1}{n}}\Big|^{\frac{n}{n-1}}\Big)}{\Big(1+\Big|(n\kappa_n)^{-\frac{1}{n}}(\ln k)^{\frac{n-1}{n}}\Big|\Big)^q}~(sk^{-1})^{n-\beta}\\ &=\frac{\operatorname{\exp}\Big(n\big(1-\frac{\beta}{n}\big) \ln k\Big)}{\Big(1+\Big|(n\kappa_n)^{-\frac{1}{n}}(\ln k)^{\frac{n-1}{n}}\Big|\Big)^q}~(sk^{-1})^{n-\beta}\\
    & \gtrsim (\ln k)^{-q\big(\frac{n-1}{n}\big)}.
\end{align*}
This in turn implies that
$$\bigg(\int_{\mathbb{R}^n}\frac{|v_k|^n}{F^0(x)^\beta}\,\mathrm{d}x\bigg)^{-1}\int_{\mathbb{R}^n}\frac{\Phi\big(\lambda_{\beta,n}|v_k|^{\frac{n}{n-1}}\big)}{(1+|v_k|)^qF^0(x)^\beta}\,\mathrm{d}x\gtrsim (\ln k)^{1-q\big(\frac{n-1}{n}\big)}\to+\infty$$
as $k\to\infty$, where we imposed the hypothesis $q<\frac{n}{n-1}$. This establishes the optimality of the exponent $\frac{n}{n-1}$ in the denominator, in the sense that for any $q<\frac{n}{n-1}$, the conclusion fails. To show $\lambda_{\beta,n}$ is sharp, fix $\lambda>\lambda_{\beta,n}$ and write $\lambda=(1+\delta_0)\lambda_{\beta,n}$ with $\delta_0>0$. Repeating the previous argument yields
\begin{align*}
    \int_{\mathbb{R}^n}\frac{\Phi\big(\lambda|v_k|^{\frac{n}{n-1}}\big)}{(1+|v_k|)^q F^0(x)^\beta}\,\mathrm{d}x&\geq \int_{\{x\in\mathbb{R}^n\colon F^0(x)\leq sk^{-1}\}}\frac{\Phi\big(\lambda|v_k|^{\frac{n}{n-1}}\big)}{(1+|v_k|)^q F^0(x)^\beta}\,\mathrm{d}x\\
    &\gtrsim \frac{\operatorname{\exp}\Big(\lambda\Big|(n\kappa_n)^{-\frac{1}{n}}(\ln k)^{\frac{n-1}{n}}\Big|^{\frac{n}{n-1}}\Big)}{\Big(1+\Big|(n\kappa_n)^{-\frac{1}{n}}(\ln k)^{\frac{n-1}{n}}\Big|\Big)^q}\int_{0}^{sk^{-1}}\eta^{n-\beta-1}\,\mathrm{d}\eta\\
    &\gtrsim \frac{\operatorname{\exp}\Big((1+\delta_0)\lambda_n\big(1-\frac{\beta}{n}\big)\Big|(n\kappa_n)^{-\frac{1}{n}}(\ln k)^{\frac{n-1}{n}}\Big|^{\frac{n}{n-1}}\Big)}{\Big(1+\Big|(n\kappa_n)^{-\frac{1}{n}}(\ln k)^{\frac{n-1}{n}}\Big|\Big)^q}~(sk^{-1})^{n-\beta}\\ 
    &=\frac{\operatorname{\exp}\Big(n (1+\delta_0)\big(1-\frac{\beta}{n}\big)\ln k\Big)}{\Big(1+\Big|(n\kappa_n)^{-\frac{1}{n}}(\ln k)^{\frac{n-1}{n}}\Big|\Big)^q}~(sk^{-1})^{n-\beta}\\
    & \gtrsim k^{(n-\beta)\delta_0}(\ln k)^{-q\big(\frac{n-1}{n}\big)}.
\end{align*}
Hence, we deduce that
\begin{align*}
    \bigg(\int_{\mathbb{R}^n}\frac{|v_k|^n}{F^0(x)^\beta}\,\mathrm{d}x\bigg)^{-1}\int_{\mathbb{R}^n}&\frac{\Phi\big(\lambda|v_k|^{\frac{n}{n-1}}\big)}{(1+|v_k|)^q F^0(x)^\beta}\,\mathrm{d}x\gtrsim k^{(n-\beta)\delta_0}(\ln k)^{1-q\big(\frac{n-1}{n}\big)}\\
    &=\operatorname{\exp}\Bigg(\ln k\bigg((n-\beta)\delta_0+\bigg(1-\frac{q(n-1)}{n}\bigg)\frac{\ln(\ln k)}{\ln k}\bigg)\Bigg)\to+\infty \
\end{align*} 
as $k\to\infty$, thank to the fact that $q\geq 0$. This confirms the sharpness of $\lambda_{\beta,n}$, thus concluding the proof of Theorem \ref{thm1.5}.\qed

As a consequence of Theorem \ref{thm1.5}, we obtain an anisotropic extension of the subcritical singular Moser--Trudinger inequality as follows.
\begin{corollary}\label{cor5.2}
    For all $n\geq 2$, $0\leq\beta<n$, $0<\lambda<\lambda_{\beta,n}$ and $u\in W^{1,n}(\mathbb{R}^n)$, there holds
\begin{align}\label{eq5.5}
\sup_{u\in W^{1,n}(\mathbb{R}^n),~\|F(\nabla u)\|_n\leq 1}\bigg(\int_{\mathbb{R}^n}\frac{|u|^n}{F^0(x)^\beta}\,\mathrm{d}x\bigg)^{-1}\int_{\mathbb{R}^n}\frac{\Phi\big(\lambda|u|^{\frac{n}{n-1}}\big)}{F^0(x)^\beta}\,\mathrm{d}x<+\infty.
\end{align}
In addition, the above inequality is sharp, in the sense that, for $\lambda\geq \lambda_{\beta,n}$, the supremum \eqref{eq5.5} is infinite.
\end{corollary}
\begin{proof}
Notice that for all $0<\lambda<\lambda_{\beta,n}$, we have
$$\Phi\big(\lambda|t|^{\frac{n}{n-1}}\big)\leq C\frac{\Phi\big(\lambda_{\beta,n}|t|^{\frac{n}{n-1}}\big)}{(1+|t|)^{\frac{n}{n-1}}}\quad\text{for all}\quad t\in\mathbb{R},$$
where $C>0$ for some constant independent of $t$. Let $u\in W^{1,n}(\mathbb{R}^n)$ be such that $\|F(\nabla u)\|_n\leq 1$. It follows from the above inequality and Theorem \ref{thm1.5} that
$$\int_{\mathbb{R}^n}\frac{\Phi\big(\lambda|u|^{\frac{n}{n-1}}\big)}{F^0(x)^\beta}\,\mathrm{d}x\lesssim \int_{\mathbb{R}^n}\frac{\Phi\big(\lambda_{\beta,n}|u|^{\frac{n}{n-1}}\big)}{(1+|u|)^{\frac{n}{n-1}}F^0(x)^\beta}\,\mathrm{d}x \lesssim\int_{\mathbb{R}^n}\frac{|u|^n}{F^0(x)^\beta}\,\mathrm{d}x.$$
This shows the validity of the desired inequality \eqref{eq5.5}. To complete the proof, we verify the sharpness for $\lambda\geq \lambda_{\beta,n}$. We only focus on $\lambda=\lambda_{\beta,n}$, as the case $\lambda>\lambda_{\beta,n}$ may be treated in precisely the same way. To proceed, we recall the Moser-type sequence of functions $\{v_k\}_k$ introduced in the proof of Theorem \ref{thm1.1}, which enjoys the following properties
\begin{align*}
\|F(\nabla v_k)\|^n_n=1 \qquad \text{and}\qquad \int_{\mathbb{R}^n}\frac{|v_k|^n}{F^0(x)^\beta}\,\mathrm{d}x=O\Big(\frac{1}{\ln k}\Big)\quad\text{as}\quad k\to\infty.    
\end{align*}
By the above construction, we have
\begin{align*}
\int_{\mathbb{R}^n}\frac{\Phi\big(\lambda_{\beta,n}|v_k|^{\frac{n}{n-1}}\big)}{F^0(x)^\beta}\,\mathrm{d}x&\geq \int_{\{x\in\mathbb{R}^n\colon F^0(x)\leq sk^{-1}\}}\frac{\Phi\big(\lambda_{\beta,n}|v_k|^{\frac{n}{n-1}}\big)}{F^0(x)^\beta}\,\mathrm{d}x\\
    &=\int_{0}^{sk^{-1}}\frac{\Phi\Big(\lambda_{\beta,n}\Big|(n\kappa_n)^{-\frac{1}{n}}(\ln k)^{\frac{n-1}{n}}\Big|^{\frac{n}{n-1}}\Big)}{\eta^\beta}\int_{\partial\mathcal{W}_\eta}\frac{1}{|\nabla F^0(x)|}\,\mathrm{d}\sigma\,\mathrm{d}\eta\\
    &\gtrsim \operatorname{\exp}\bigg(\lambda_n\Big(1-\frac{\beta}{n}\Big)\Big|(n\kappa_n)^{-\frac{1}{n}}(\ln k)^{\frac{n-1}{n}}\Big|^{\frac{n}{n-1}}\bigg)\int_{0}^{sk^{-1}}\eta^{n-\beta-1}\,\mathrm{d}\eta\\
    &\gtrsim \operatorname{\exp}\bigg(\lambda_n\Big(1-\frac{\beta}{n}\Big)\Big|(n\kappa_n)^{-\frac{1}{n}}(\ln k)^{\frac{n-1}{n}}\Big|^{\frac{n}{n-1}}\bigg)(sk^{-1})^{n-\beta}\\
    &\gtrsim\operatorname{\exp}\bigg(n \Big(1-\frac{\beta}{n}\Big)\ln k\bigg)(k^{-1})^{n-\beta}=1.
\end{align*}
This shows that
\begin{align*}
    \sup_{u\in W^{1,n}(\mathbb{R}^n),~\|F(\nabla u)\|_n\leq 1}&\bigg(\int_{\mathbb{R}^n}\frac{|u|^n}{F^0(x)^\beta}\,\mathrm{d}x\bigg)^{-1}\int_{\mathbb{R}^n}\frac{\Phi\big(\lambda_{\beta,n}|u|^{\frac{n}{n-1}}\big)}{F^0(x)^\beta}\,\mathrm{d}x\\ &\geq \bigg(\int_{\mathbb{R}^n}\frac{|v_k|^n}{F^0(x)^\beta}\,\mathrm{d}x\bigg)^{-1}\int_{\mathbb{R}^n}\frac{\Phi\big(\lambda_{\beta,n}|v_k|^{\frac{n}{n-1}}\big)}{F^0(x)^\beta}\,\mathrm{d}x\gtrsim \ln k\to+\infty
\end{align*}
as $k\to\infty$. The proof is thus complete.
\end{proof}
In virtue of Lemma \ref{lem2.6}, the following result is well justified and will be a key ingredient in the proof of Theorem \ref{thm1.6}.
\begin{corollary}\label{cor5.3}
    For all $n\geq 2$, $0\leq\beta<n$, $p\geq n$ and $u\in W^{1,n}(\mathbb{R}^n)$, there holds
\begin{align}\label{eq5.6}
\bigg(\int_{\mathbb{R}^n}\frac{|u|^p}{F^0(x)^\beta}\,\mathrm{d}x\bigg)^{\frac{1}{p}}\leq c(n,p,\beta)\bigg(\int_{\mathbb{R}^n}\frac{|u|^n}{F^0(x)^\beta}\,\mathrm{d}x\bigg)^{\frac{1}{p}}\bigg(\int_{\mathbb{R}^n} F(\nabla u)^n\,\mathrm{d}x\bigg)^{\frac{p-n}{np}},
\end{align}
where $c(n,p,\beta)>0$ is a constant depending on $n$, $p$ and $\beta$. Also, there exists $\tilde{\zeta}>0$ such that $c(n,p,\beta)$ has the asymptotic behavior $c(n,p,\beta)\simeq \tilde{\zeta} p^{\frac{1}{n^\prime}}$ as $p\to\infty$, where $n^\prime=\frac{n}{n-1}$.
\end{corollary}
\begin{proof}
    Let $\lambda_0$ be the supremum of all $\lambda>0$ such that for all $u\in W^{1,n}(\mathbb{R}^n)$ and $\|F(\nabla u)\|_n\leq 1$, we have
$$\int_{\mathbb{R}^n}\frac{\Phi\big(\lambda|u|^{\frac{n}{n-1}}\big)}{F^0(x)^\beta}\,\mathrm{d}x\leq C \int_{\mathbb{R}^n}\frac{|u|^n}{F^0(x)^\beta}\,\mathrm{d}x,$$
   where  $C>0$ is a constant depending only on $n$ and $\beta$. By using Corollary \ref{cor5.2}, one has $\lambda_0=\lambda_{\beta,n}$. The above estimate further yields that for each $\lambda\in(0,\lambda_0)$, there exists a positive constant $C$ for which
   \begin{align*}
     C\bigg(\int_{\mathbb{R}^n}\frac{|u|^n}{F^0(x)^\beta}\,\mathrm{d}x\bigg)\bigg(\int_{\mathbb{R}^n} F(\nabla u)^n\,\mathrm{d}x\bigg)^{-1} &\geq \int_{\mathbb{R}^n}\Phi\Bigg(\lambda\Bigg|\frac{u}{\|F(\nabla u)\|_n}\Bigg|^{\frac{n}{n-1}}\Bigg)\frac{\,\mathrm{d}x}{F^0(x)^\beta}\\
     &\geq\frac{\lambda^j} {j!} \bigg(\int_{\mathbb{R}^n}\frac{|u|^{n^\prime j}}{F^0(x)^\beta}\,\mathrm{d}x\bigg)\bigg(\int_{\mathbb{R}^n} F(\nabla u)^n\,\mathrm{d}x\bigg)^{-\frac{n^\prime j}{n}}
   \end{align*}
holds for every $u\in W^{1,n}(\mathbb{R}^n)\setminus\{0\}$ and all integers $j\geq n-1$. Consequently, we deduce that 
\begin{align}\label{eq5.7}
  \bigg(\int_{\mathbb{R}^n}\frac{|u|^{n^\prime j}}{F^0(x)^\beta}\,\mathrm{d}x\bigg)^{\frac{1}{n^\prime j}}\leq \bigg(\frac{j! C}{\lambda^j}\bigg)^{\frac{1}{n^\prime j}} \bigg(\int_{\mathbb{R}^n}\frac{|u|^{n}}{F^0(x)^\beta}\,\mathrm{d}x\bigg)^{\frac{1}{n^\prime j}} \bigg(\int_{\mathbb{R}^n} F(\nabla u)^n\,\mathrm{d}x\bigg)^{\frac{1}{n}\big(1-\frac{n}{n^\prime j}\big)}
\end{align}
for any $u\in W^{1,n}(\mathbb{R}^n)$ and all integers $j\geq n-1$. Thus, there exists an integer $j\geq n-1$ with $n^\prime j\leq p<n^\prime(j+1)$ for all $p\geq n$. In addition, we can find $\eta\in(0,1]$ such that $\frac{1}{p}=\frac{\eta}{n^\prime j}+\frac{1-\eta}{n^\prime(j+1)}$. Due to H\"{o}lder's inequality, we get
\begin{align}\label{eq5.8}
\bigg(\int_{\mathbb{R}^n}\frac{|u|^p}{F^0(x)^\beta}\,\mathrm{d}x\bigg)^{\frac{1}{p}}\leq\bigg(\int_{\mathbb{R}^n}\frac{|u|^{n^\prime j}}{F^0(x)^\beta}\,\mathrm{d}x\bigg)^{\frac{\eta}{n^\prime j}}\bigg(\int_{\mathbb{R}^n}\frac{|u|^{n^\prime (j+1)}}{F^0(x)^\beta}\,\mathrm{d}x\bigg)^{\frac{1-\eta}{n^\prime(j+1)}}. 
\end{align}
In virtue of \eqref{eq5.7} and \eqref{eq5.8}, it follows for all $u\in W^{1,n}(\mathbb{R}^n)$ and $p\in[n,+\infty)$ that
\begin{align}\label{eq5.9}
  \bigg(\int_{\mathbb{R}^n}\frac{|u|^p}{F^0(x)^\beta}\,\mathrm{d}x\bigg)^{\frac{1}{p}}\leq ((j+1)!)^{\frac{1}{p}}\lambda^{-\frac{1}{n^\prime}} C^{\frac{1}{p}}\bigg(\int_{\mathbb{R}^n}\frac{|u|^{n}}{F^0(x)^\beta}\,\mathrm{d}x\bigg)^{\frac{1}{p}}\bigg(\int_{\mathbb{R}^n} F(\nabla u)^n\,\mathrm{d}x\bigg)^{\frac{p-n}{np}}.
\end{align}
The condition $p\geq n^\prime j$ implies that
$$((j+1)!)^{\frac{1}{p}}\leq \bigg(\Gamma\Big(\frac{p}{n^\prime}+2\Big)\bigg)^{\frac{1}{p}}.$$ 
Moreover, the estimate 
$$c(n,p,\beta)\approx \lambda^{-\frac{1}{n^\prime}} C^{\frac{1}{p}}\bigg(\Gamma\Big(\frac{p}{n^\prime}+2\Big)\bigg)^{\frac{1}{p}}$$
together with \eqref{eq5.9} establishes \eqref{eq5.6}. To finish the prove, it suffices to show that there exists $\tilde{\zeta}>0$ such that $c(n,p,\beta)\simeq \tilde{\zeta} p^{\frac{1}{n^\prime}}$ as $p\to\infty$, that is, for all $p_0\geq n$, there exists $p\geq p_0$ such that 
\begin{align}\label{eq5.10}
\bigg(\int_{\mathbb{R}^n}\frac{|u|^p}{F^0(x)^\beta}\,\mathrm{d}x\bigg)^{\frac{1}{p}}\leq \tilde{\zeta} p^{\frac{1}{n^\prime}}\bigg(\int_{\mathbb{R}^n}\frac{|u|^n}{F^0(x)^\beta}\,\mathrm{d}x\bigg)^{\frac{1}{p}}\bigg(\int_{\mathbb{R}^n} F(\nabla u)^n\,\mathrm{d}x\bigg)^{\frac{p-n}{np}}
\end{align}
for all $u\in W^{1,n}(\mathbb{R}^n)$. By employing similar procedure developed as in Corollary \ref{cor4.3}, we obtain for all $u\in W^{1,n}(\mathbb{R}^n)$ that
\begin{align*}
 \bigg(\int_{\mathbb{R}^n}\frac{|u|^p}{F^0(x)^\beta}\mathrm{d}x\bigg)^{\frac{1}{p}} & \leq (1+o_p(1))\bigg(\frac{p}{\lambda n^\prime\operatorname{\exp}(1)}\bigg)^{\frac{1}{n^\prime}} \bigg(\int_{\mathbb{R}^n}\frac{|u|^n}{F^0(x)^\beta}\mathrm{d}x\bigg)^{\frac{1}{p}}\\
 &\qquad\times\bigg(\int_{\mathbb{R}^n}F(\nabla u)^n\mathrm{d}x\bigg)^{\frac{p-n}{np}}\quad\text{as}\quad p\to\infty.   
\end{align*}
 Hence, the inequality \eqref{eq5.10} is obtained, which concludes the proof.
\end{proof}
To end this section, it suffices to establish Theorem \ref{thm1.6}, whose proof is given below.
\subsection{{Proof of Theorem \ref{thm1.6}}} Suppose that Theorem \ref{thm1.5} holds. It remains to establish \eqref{eq1.12} for all $0\leq \beta<n$ and $0<\lambda\leq \lambda_{\beta,n}$. Let $u\in W^{1,n}(\mathbb{R}^n)$ be such that $\|u\|\leq1$. In addition, we fix $\eta\in(0,1)$ in such a way that $\|u\|_n^n=\eta$ and $\|F(\nabla u)\|_n^n\leq 1-\eta$. This gives rise to the following two cases.

\textbf{Case-1:} Let $\eta\geq\frac{n-1}{n}$. By defining $v(x)=n^\frac{1}{n}u(x)$ for all $x\in \mathbb{R}^n$, one sees that $v\in W^{1,n}(\mathbb{R}^n)$ and there hold
$$\|F(\nabla v)\|_n^n=n \|F(\nabla u)\|_n^n\leq n(1-\eta)\leq 1 $$
and
$$\int_{\mathbb{R}^n}\frac{|v|^n}{F^0(x)^\beta}\,\mathrm{d}x=n \int_{\mathbb{R}^n}\frac{|u|^n}{F^0(x)^\beta}\,\mathrm{d}x\lesssim \|u\|^n\lesssim 1,$$
where we apply Lemma \ref{lem2.6}. Consequently, for all $0<\lambda< \lambda_{\beta,n}$, it follows from Corollary \ref{cor5.2} that
$$ \int_{\mathbb{R}^n}\frac{\Phi\big(\lambda n^{\frac{1}{n-1}}|u|^{\frac{n}{n-1}}\big)}{F^0(x)^\beta}\,\mathrm{d}x=\int_{\mathbb{R}^n}\frac{\Phi\big(\lambda|v|^{\frac{n}{n-1}}\big)}{F^0(x)^\beta}\,\mathrm{d}x\lesssim \int_{\mathbb{R}^n}\frac{|v|^n}{F^0(x)^\beta}\,\mathrm{d}x\lesssim 1.$$
Pick $\lambda n^{\frac{1}{n-1}}=\lambda_{\beta,n}$. This shows that $\lambda< \lambda_{\beta,n}$, which immediately yields \eqref{eq1.12}.

\textbf{Case-2:} Let $\eta<\frac{n-1}{n}$ and $D:=\{x\in \mathbb{R}^n\colon ~|u(x)|\geq 1\}$. Now, we divide the whole integral as follows
$$\int_{\mathbb{R}^n}\frac{\Phi\big(\lambda|u|^{\frac{n}{n-1}}\big)}{F^0(x)^\beta}\,\mathrm{d}x=\bigg(\int_{\mathbb{R}^n\setminus D}+\int_{D}\bigg)\frac{\Phi\big(\lambda|u|^{\frac{n}{n-1}}\big)}{F^0(x)^\beta}\,\mathrm{d}x=:I+J.$$
Recalling that $|u(x)|<1$ for all $x\in\mathbb{R}^n\setminus D$ and $\|u\|\leq 1$, we obtain from Lemma \ref{lem2.6} that
\begin{align*}
  I&= \int_{\mathbb{R}^n\setminus D}\sum_{j=n-1}^{\infty}\frac{\lambda^j}{j!}\frac{|u|^{\frac{nj}{n-1}}}{F^0(x)^\beta}\,\mathrm{d}x\leq \int_{\mathbb{R}^n\setminus D}\sum_{j=n-1}^{\infty}\frac{\lambda^j}{j!}\frac{|u|^{n}}{F^0(x)^\beta}\,\mathrm{d}x\\
  &\leq \operatorname{\exp}(\lambda)\int_{\mathbb{R}^n\setminus D}\frac{|u|^{n}}{F^0(x)^\beta}\,\mathrm{d}x\leq C(n,\beta)\|u\|^n\leq C(n,\beta),
\end{align*}
where $C(n,\beta)$ is a constant depending on $n$ and $\beta$, thanks to $\lambda\leq \lambda_{\beta,n}$. Let $q,\,q^\prime>1$ be such that $\frac{1}{q}+\frac{1}{q^\prime}=1$, where $q=\frac{n-1}{n-1-\eta}$. In virtue of $\frac{n^\prime}{q-1}>n$ and Corollary \ref{cor5.3}, we have
$$\bigg(\int_{\mathbb{R}^n}\frac{|u|^{\frac{n^\prime}{q-1}}}{F^0(x)^\beta}\,\mathrm{d}x\bigg)^{\frac{q-1}{n^\prime}}\lesssim \bigg(\frac{n^\prime}{q-1}\bigg)^{\frac{1}{n^\prime}}\bigg(\int_{\mathbb{R}^n}\frac{|u|^n}{F^0(x)^\beta}\,\mathrm{d}x\bigg)^{\frac{q-1}{n^\prime}}$$
for all $u\in W^{1,n}(\mathbb{R}^n)$ satisfying $\|F(\nabla u)\|_n\leq 1$. It follows from H\"{o}lder's inequality and Lemma \ref{lem2.4} that
\begin{align*}
    J&\leq \bigg(\int_{D}\frac{(1+|u|)^{\frac{n^\prime}{q-1}}}{F^0(x)^\beta}\,\mathrm{d}x\bigg)^{\frac{1}{q^\prime}}\bigg(\int_{D}\frac{\Phi(\lambda q|u|^{\frac{n}{n-1}})}{(1+|u|)^{\frac{n}{n-1}}F^0(x)^\beta}\,\mathrm{d}x\bigg)^{\frac{1}{q}}\\
    & \leq 2^{\frac{n^\prime}{q}}\bigg(\int_{D}\frac{|u|^{\frac{n^\prime}{q-1}}}{F^0(x)^\beta}\,\mathrm{d}x\bigg)^{\frac{1}{q^\prime}}\bigg(\int_{D}\frac{\Phi(\lambda q|u|^{\frac{n}{n-1}})}{(1+|u|)^{\frac{n}{n-1}}F^0(x)^\beta}\,\mathrm{d}x\bigg)^{\frac{1}{q}}\\
    &\lesssim 2^{\frac{n^\prime}{q}}\bigg(\frac{n^\prime}{q-1}\bigg)^{\frac{1}{q}}\bigg(\int_{\mathbb{R}^n}\frac{|u|^n}{F^0(x)^\beta}\,\mathrm{d}x\bigg)^{\frac{1}{q^\prime}}\bigg(\int_{\mathbb{R}^n}\frac{\Phi(\lambda q|u|^{\frac{n}{n-1}})}{(1+|u|)^{\frac{n}{n-1}}F^0(x)^\beta}\,\mathrm{d}x\bigg)^{\frac{1}{q}}.
\end{align*}
Set $v(x):=q^{\frac{1}{n^\prime}}u(x)$ for all $x\in\mathbb{R}^n$. By following similar arguments used in Theorem \ref{thm1.4}, one sees that $v\in W^{1,n}(\mathbb{R}^n)$ with $\|F(\nabla u)\|_n\leq 1$, and also there holds
$$\int_{\mathbb{R}^n}\frac{|v|^n}{F^0(x)^\beta}\,\mathrm{d}x =q^{\frac{n}{n^\prime}}\int_{\mathbb{R}^n}\frac{|u|^n}{F^0(x)^\beta}\,\mathrm{d}x.$$
The above estimates and Theorem \ref{thm1.5} lead to
\begin{align*}
\bigg(\int_{\mathbb{R}^n}\frac{\Phi(\lambda q|u|^{\frac{n}{n-1}})}{(1+|u|)^{\frac{n}{n-1}}F^0(x)^\beta}\,\mathrm{d}x\bigg)^{\frac{1}{q}}&\leq \bigg(q\int_{\mathbb{R}^n}\frac{\Phi(\lambda |v|^{\frac{n}{n-1}})}{(1+|v|)^{\frac{n}{n-1}}F^0(x)^\beta}\,\mathrm{d}x\bigg)^{\frac{1}{q}}\\
&\lesssim\bigg(q\int_{\mathbb{R}^n}\frac{|v|^n}{F^0(x)^\beta}\,\mathrm{d}x\bigg)^{\frac{1}{q}}=q^{\frac{n}{q}}\bigg(\int_{\mathbb{R}^n}\frac{|u|^n}{F^0(x)^\beta}\,\mathrm{d}x\bigg)^{\frac{1}{q}}.
\end{align*}
Taking into account that $q>1$ and $0<\eta<\frac{n-1}{n}$, and combining Lemma \ref{lem2.6} with the above estimates, we argue as in the proof of Theorem \ref{thm1.4} to obtain that
\begin{align*}
    J&\lesssim 2^{\frac{n^\prime}{q}}\bigg(\frac{n^\prime}{q-1}\bigg)^{\frac{1}{q}}q^{\frac{n}{q}}\int_{\mathbb{R}^n}\frac{|u|^n}{F^0(x)^\beta}\,\mathrm{d}x\lesssim 2^{\frac{n^\prime}{q}}\bigg(\frac{n^\prime}{q-1}\bigg)^{\frac{1}{q}}q^{\frac{n}{q}}\|u\|^n\\
    &\lesssim (2^{n^\prime} n^\prime)^{\frac{1}{q}}\bigg(\frac{q}{q-1}\bigg)^{\frac{1}{q}}q^{\frac{n-1}{q}}\lesssim 2^{n^\prime}n^\prime\bigg(\frac{n-1}{\eta}\bigg)^{1-\frac{\eta}{n-1}}\bigg(1-\frac{\eta}{n-1}\bigg)^{-(n-1-\eta)}\\
    &\lesssim2^{n^\prime}n\bigg(1-\frac{1}{n}\bigg)^{-(n-1)}.
\end{align*}
By the above discussion, we infer that \eqref{eq1.12} holds. The conclusion now follows from both cases. This finishes the proof.\qed

\section{The Anisotropic Singular Moser--Trudinger Inequality with Exact Growth-II and III}\label{sec6} 
The aim of this section is to derive some variants of anisotropic singular Moser–Trudinger inequalities with exact growth condition in the whole $\mathbb{R}^n$. More precisely, we use the quasi-conformal transformations introduced in Section \ref{sec5} to prove Theorems \ref{thm1.7} and \ref{thm1.8}. 

Now, we begin by proving Theorem \ref{thm1.7} as follows.
\subsection{{Proof of Theorem \ref{thm1.7}}} Let $u\in W^{1,n}(\mathbb{R}^n)$ be such that $\|F(\nabla u)\|_n\leq 1$. In addition, we consider the function $w$ that appeared in \eqref{eq5.1}. Now, by taking $r=n$ and $\gamma=\frac{n}{n-\beta}$, we obtain from \eqref{eq5.3}, \eqref{eq5.333} and Lemma \ref{lem5.1} that
\begin{align}\label{eq6.1}
    \int_{\mathbb{R}^n}\frac{\big|\big(\frac{n-\beta}{n}\big)^{\frac{n-1}{n}} u(y)\big|^n}{F^0(y)^\beta}\mathrm{d}y=\frac{n}{n-\beta} \int_{\mathbb{R}^n}|w(x)|^n\mathrm{d}y,\, \int_{\mathbb{R}^n}F(\nabla w)^n\mathrm{d}x\leq \int_{\mathbb{R}^n}F(\nabla u)^n\mathrm{d}y.  
\end{align}
In virtue of \eqref{eq6.1} and Lemma \ref{lem2.6}, one has $w\in W^{1,n}(\mathbb{R}^n)$ such that $\|F(\nabla w)\|_n\leq 1$. Moreover, by using \eqref{eq5.2}, we get
\begin{align*}
\int_{\mathbb{R}^n}&\frac{\Phi\big(\lambda\big(1-\frac{\beta}{n}\big)|u(y)|^{\frac{n}{n-1}}\big)}{\Big(1+\big(1-\frac{\beta}{n}\big)|u(y)|^{\frac{n}{n-1}}\Big)F^0(y)^\beta}\,\mathrm{d}y \\ &= \int_{\mathbb{R}^n}\frac{\Phi\big(\lambda\big(1-\frac{\beta}{n}\big)\big|\big(\frac{n}{n-\beta}\big)^{\frac{n-1}{n}}w(x)\big|^{\frac{n}{n-1}}\big)}{\Big(1+\big(1-\frac{\beta}{n}\big)\big|\big(\frac{n}{n-\beta}\big)^{\frac{n-1}{n}}w(x)\big|^{\frac{n}{n-1}}\Big)F^0(x)^{\frac{n\beta}{n-\beta}}}~\Big(\frac{n}{n-\beta}\Big)F^0(x)^{\frac{n\beta}{n-\beta}}\,\mathrm{d}x \\
&=\Big(\frac{n}{n-\beta}\Big)\int_{\mathbb{R}^n}\frac{\Phi\big(\lambda|w(x)|^{\frac{n}{n-1}}\big)}{1+|w(x)|^{\frac{n}{n-1}}}\,\mathrm{d}x.
\end{align*}
This together with \eqref{eq6.1} and  Liu \cite[Theorem 1.1]{Liu-2022} yields
\begin{align*}
    \sup_{u\in W^{1,n}(\mathbb{R}^n),~\|F(\nabla u)\|_n\leq 1} \bigg(&\int_{\mathbb{R}^n}\frac{|u|^n}{F^0(x)^\beta}\,\mathrm{d}x\bigg)^{-1}\int_{\mathbb{R}^n}\frac{\Phi\big(\lambda\big(1-\frac{\beta}{n}\big)|u|^{\frac{n}{n-1}}\big)}{\Big(1+\big(1-\frac{\beta}{n}\big)|u|^{\frac{n}{n-1}}\Big)F^0(x)^\beta}\,\mathrm{d}x\\
    &\lesssim \sup_{w\in W^{1,n}(\mathbb{R}^n),~\|F(\nabla w)\|_n\leq 1} \frac{1}{\|w\|_n^n}\int_{\mathbb{R}^n}\frac{\Phi\big(\lambda|w(x)|^{\frac{n}{n-1}}\big)}{1+|w(x)|^{\frac{n}{n-1}}}\,\mathrm{d}x<+\infty.
\end{align*}
This shows the validity of \eqref{eq1.16}. The sharpness of the exponent $\frac{n}{n-1}$ and the constant $\lambda_n$ follows by arguments similar to those used in Theorem \ref{thm1.5}, and hence we omit the details. This completes the proof.\qed

To end this section, we are now fully equipped to prove Theorem \ref{thm1.8} as follows.
\subsection{{Proof of Theorem \ref{thm1.8}}} Let $u\in W^{1,n}(\mathbb{R}^n)$ be such that $\|F(\nabla u)\|_n\leq 1$. Then, by applying the  the anisotropic Caffarelli--Kohn--Nirenberg inequality established by Cheng--Zhang--Zhou \cite{Cheng-Zhang-Zhou-2023}, we obtain for all $0\leq s\leq \beta<n$ that
\begin{align}\label{eq6.2}
 \int_{\mathbb{R}^n}\frac{|u|^n}{F^0(x)^\beta}\,\mathrm{d}x \lesssim \bigg(\int_{\mathbb{R}^n}&\frac{|u|^n}{F^0(x)^s}\,\mathrm{d}x\bigg)^{\frac{n-\beta}{n-s}}.
\end{align}
Define
$$ w(x):=\Big(\frac{n-\beta}{n}\Big)^{\frac{n-1}{n}}u(F^0(x)^{\frac{\beta}{n-\beta}}x) \quad\text{for all}\quad x\in\mathbb{R}^n.$$
It follows from Theorem \ref{thm1.7} that \eqref{eq6.1} holds and $w\in W^{1,n}(\mathbb{R}^n)$ with $\|F(\nabla w)\|_n\leq 1$. Now, by proceeding as in Theorem \ref{thm1.7}, one deduces that 
\begin{align*}
    \int_{\mathbb{R}^n}&\frac{\Phi\big(\lambda\big(1-\frac{\beta}{n}\big)|u(y)|^{\frac{n}{n-1}}\big)}{\Big(1+\big(1-\frac{\beta}{n}\big)^{\big(\frac{n-\beta}{n-s}\big)}|u(y)|^{\frac{n}{n-1}\big(\frac{n-\beta}{n-s}\big)}\Big)F^0(y)^\beta}\,\mathrm{d}y\\&=\Big(\frac{n}{n-\beta}\Big)\int_{\mathbb{R}^n}\frac{\Phi\big(\lambda|w(x)|^{\frac{n}{n-1}}\big)}{1+|w(x)|^{\frac{n}{n-1}\big(\frac{n-\beta}{n-s}\big)}}\,\mathrm{d}x\leq \Big(\frac{n}{n-\beta}\Big)\int_{\mathbb{R}^n}\Phi\big(\lambda|w(x)|^{\frac{n}{n-1}}\big)\,\mathrm{d}x.
\end{align*}
By combining \eqref{eq6.1}, \eqref{eq6.2}, and the above estimate with Corollary \ref{cor4.2}, yields
\begin{align*}
    \sup_{\substack{u\in W^{1,n}(\mathbb{R}^n),\\  \|F(\nabla u)\|_n\leq 1}} &\bigg(\int_{\mathbb{R}^n}\frac{|u|^n}{F^0(x)^s}\,\mathrm{d}x\bigg)^{-\big(\frac{n-\beta}{n-s}\big)}\int_{\mathbb{R}^n}\frac{\Phi\big(\lambda\big(1-\frac{\beta}{n}\big)|u|^{\frac{n}{n-1}}\big)}{\Big(1+\big(1-\frac{\beta}{n}\big)^{\big(\frac{n-\beta}{n-s}\big)}|u|^{\frac{n}{n-1}\big(\frac{n-\beta}{n-s}\big)}\Big)F^0(x)^\beta}\,\mathrm{d}x\\
    &\lesssim   \sup_{w\in W^{1,n}(\mathbb{R}^n),~\|F(\nabla w)\|_n\leq 1} \frac{1}{\|w\|_n^n}\int_{\mathbb{R}^n}\Phi\big(\lambda|w|^{\frac{n}{n-1}}\big)\,\mathrm{d}x<+\infty.
\end{align*}
This shows the validity of the desired inequality \eqref{1.999}. In what follows, we introduce the Moser-type sequence functions $\{v_k\}_k$ defined in \eqref{eq3.4}, which possesses the following properties
\begin{align*}
\|F(\nabla v_k)\|^n_n=1 \qquad \text{and}\qquad \int_{\mathbb{R}^n}\frac{|v_k|^n}{F^0(x)^s}\,\mathrm{d}x=O\Big(\frac{1}{\ln k}\Big)\quad\text{as}\quad k\to\infty.    
\end{align*}
To finish the proof, it is enough to show that
$$J_k:=\bigg(\int_{\mathbb{R}^n}\frac{|v_k|^n}{F^0(x)^s}\mathrm{d}x\bigg)^{-\big(\frac{n-\beta}{n-s}\big)}\int_{\mathbb{R}^n}\frac{\Phi\big(\lambda_n\big(1-\frac{\beta}{n}\big)|v_k|^{\frac{n}{n-1}}\big)}{\Big(1+\big(1-\frac{\beta}{n}\big)^{\big(\frac{n-\beta}{n-s}\big)}|v_k|^{q\big(\frac{n-\beta}{n-s}\big)}\Big)F^0(x)^\beta}\mathrm{d}x\to+\infty$$
as $k\to\infty$, where the exponent $q$ satisfies $q<\frac{n}{n-1}$. Observe that by using the above information, we have
\begin{align*}
\int_{\mathbb{R}^n}&\frac{\Phi\big(\lambda_n\big(1-\frac{\beta}{n}\big)|v_k|^{\frac{n}{n-1}}\big)}{\Big(1+\big(1-\frac{\beta}{n}\big)^{\big(\frac{n-\beta}{n-s}\big)}|v_k|^{q\big(\frac{n-\beta}{n-s}\big)}\Big)F^0(x)^\beta}\mathrm{d}x\\
&\geq \int_{\{x\in\mathbb{R}^n\colon F^0(x)\leq sk^{-1}\}}\frac{\Phi\big(\lambda_n\big(1-\frac{\beta}{n}\big)|v_k|^{\frac{n}{n-1}}\big)}{\Big(1+\big(1-\frac{\beta}{n}\big)^{\big(\frac{n-\beta}{n-s}\big)}|v_k|^{q\big(\frac{n-\beta}{n-s}\big)}\Big)F^0(x)^\beta}\mathrm{d}x\\
&
\gtrsim\frac{\operatorname{\exp}\bigg(\lambda_n\big(1-\frac{\beta}{n}\big)\Big|(n\kappa_n)^{-\frac{1}{n}}(\ln k)^{\frac{n-1}{n}}\Big|^{\frac{n}{n-1}}\bigg)}{1+\big(1-\frac{\beta}{n}\big)^{\big(\frac{n-\beta}{n-s}\big)}\Big|(n\kappa_n)^{-\frac{1}{n}}(\ln k)^{\frac{n-1}{n}}\Big|^{q\big(\frac{n-\beta}{n-s}\big)}}\int_{0}^{sk^{-1}}\eta^{n-\beta-1}\,\mathrm{d}\eta\\
&\gtrsim \frac{\operatorname{\exp}\big((n-\beta)\ln k\big)}{1+\big(1-\frac{\beta}{n}\big)^{\big(\frac{n-\beta}{n-s}\big)}\Big|(n\kappa_n)^{-\frac{1}{n}}(\ln k)^{\frac{n-1}{n}}\Big|^{q\big(\frac{n-\beta}{n-s}\big)}}~(sk^{-1})^{n-\beta}\\
&\gtrsim (\ln k)^{-q\big(\frac{n-1}{n}\big)\big(\frac{n-\beta}{n-s}\big)}.
\end{align*}
This leads to
$$ J_k\gtrsim (\ln k)^{\big(\frac{n-\beta}{n-s}\big)\big(1-q\big(\frac{n-1}{n}\big)\big)}\to+\infty\quad\text{as}\quad k\to\infty.$$
The proof is thereby completed.\qed

\section{The Anisotropic Singular Moser--Trudinger Inequality with Exact Growth-IV and Its Consequence}\label{sec7}
The purpose of this section is twofold. The proof of Theorem \ref{thm1.9} is given first. Traditionally, the study of existence and nonexistence of maximizers for Moser functionals relies on sophisticated blow-up analysis of the corresponding Euler-Lagrange functionals. In proving Theorem \ref{thm1.10}, we adopt a more direct and transparent strategy that completely avoids blow-up arguments. Instead, we develop a direct variational approach that allows us to characterize precisely the existence and nonexistence of maximizers for the supremum defined in Theorem \ref{thm1.9} when $q=n$ in the subcritical case under different parameter regimes.
\subsection{{Proof of Theorem \ref{thm1.9}}}
It suffices to prove the desired inequality \eqref{eq1.17} for the case $q=n$. By the hypothesis, we have $d>0$ and $\beta\in[0,n)$. Let $\Theta\colon[0,+\infty)\to[0,+\infty)$ be defined by
$$\Theta(s):=\frac{\Phi(\lambda s^{\frac{n}{n-1}})}{1+ds^{\frac{n}{n-1}\big(1-\frac{\beta}{n}\big)}} \quad\text{for all}\quad s\geq 0\quad\text{and}\quad\lambda >0.$$
Notice that  $\Theta$ is continuous on $[0,+\infty)$. We claim that $\Theta$ is s monotone increasing on $[0,+\infty)$. To see this, observe that for all $s\in[0,+\infty)$, we have
\begin{align*}
    &\Theta^\prime(s)\\&\geq\frac{\sum_{j= n-2}^{\infty}\frac{\lambda^j}{j!}s^{\frac{nj}{n-1}}\Big[\big(\frac{n\lambda}{n-1}\big)s^{\frac{1}{n-1}}\big(1+ds^{\frac{n}{n-1}(1-\frac{\beta}{n})}\big)-\big(\frac{nd}{n-1}\big)\big(1-\frac{\beta}{n}\big)s^{\frac{n}{n-1}\big(1-\frac{\beta}{n}\big)-1}\lambda s^{\frac{n}{n-1}}\Big]}{\Big(1+ds^{\frac{n}{n-1}\big(1-\frac{\beta}{n}\big)}\Big)^2}\\
    &\geq\frac{\sum_{j= n-2}^{\infty}\frac{\lambda^j}{j!}s^{\frac{nj}{n-1}}\Big[\big(\frac{n\lambda}{n-1}\big)\big(1-\frac{\beta}{n}\big)s^{\frac{1}{n-1}}\big(1+ds^{\frac{n}{n-1}(1-\frac{\beta}{n})}\big)-\big(\frac{nd}{n-1}\big)\big(1-\frac{\beta}{n}\big)s^{\frac{n}{n-1}\big(1-\frac{\beta}{n}\big)}\lambda s^{\frac{1}{n-1}}\Big]}{\Big(1+ds^{\frac{n}{n-1}\big(1-\frac{\beta}{n}\big)}\Big)^2}\\
    &=\frac{\big(\frac{n\lambda}{n-1}\big)\big(1-\frac{\beta}{n}\big)s^{\frac{1}{n-1}}\sum_{j= n-2}^{\infty}\frac{\lambda^j}{j!}s^{\frac{nj}{n-1}}\Big[1+ds^{\frac{n}{n-1}\big(1-\frac{\beta}{n}\big)}-ds^{\frac{n}{n-1}\big(1-\frac{\beta}{n}\big)}\Big]}{\Big(1+ds^{\frac{n}{n-1}\big(1-\frac{\beta}{n}\big)}\Big)^2}
    \\ 
    &=\frac{\big(\frac{n\lambda}{n-1}\big)\big(1-\frac{\beta}{n}\big)s^{\frac{1}{n-1}}\sum_{j= n-2}^{\infty}\frac{\lambda^j}{j!}s^{\frac{nj}{n-1}}}{\Big(1+ds^{\frac{n}{n-1}\big(1-\frac{\beta}{n}\big)}\Big)^2}\geq0.
\end{align*}
 The claim is thereby established. By using the facts that $\big(F^0(x)^{-\beta}\big)^\ast=F^0(x)^{-\beta}$ and $\Theta(|u(x)|)=(\Theta \circ |u|)(x)$ for all $u\in W^{1,n}(\mathbb{R}^n)$, it follows from \eqref{eq2.2}, Lemma \ref{lem2.91} and Lemma \ref{lem2.92} that
$$ \int_{\mathbb{R}^n}\frac{\Phi\big(\lambda|u|^{\frac{n}{n-1}}\big)}{\Big(1+d|u|^{\frac{n}{n-1}\big(1-\frac{\beta}{n}\big)}\Big)F^0(x)^\beta}\,\mathrm{d}x\leq \int_{\mathbb{R}^n}\frac{\Phi\big(\lambda|u^\ast|^{\frac{n}{n-1}}\big)}{\Big(1+d|u^\ast|^{\frac{n}{n-1}\big(1-\frac{\beta}{n}\big)}\Big)F^0(x)^\beta}\,\mathrm{d}x,$$
and 
 $$\|u^\ast\|_n=\|u\|_n\quad\text{and}\quad \|F(\nabla u^\ast)\|_n\leq \|F(\nabla u)\|_n .$$
Now, from the above estimates together with the density argument, the proof reduces to establishing \eqref{eq1.17} for every nonnegative, Wulff radially symmetric, nonincreasing function $u\in C_0^\infty(\mathbb{R}^n)\setminus\{0\}$ such that $\|F(\nabla u)\|_n\leq1$. Hereafter, with a slight abuse of notation, one can again write $u(x)=u(\eta)$, where $\eta=F^0(x)$. Consequently, we decompose the whole the integral in the form
$$ \int_{\mathbb{R}^n}\frac{\Phi\big(\lambda|u|^{\frac{n}{n-1}}\big)}{\Big(1+d|u|^{\frac{n}{n-1}\big(1-\frac{\beta}{n}\big)}\Big)F^0(x)^\beta}\,\mathrm{d}x=:I+J,$$
with $$I=\int_{\mathbb{R}^n\setminus\mathcal{W}_{r_0}}\frac{\Phi\big(\lambda|u|^{\frac{n}{n-1}}\big)}{\Big(1+d|u|^{\frac{n}{n-1}\big(1-\frac{\beta}{n}\big)}\Big)F^0(x)^\beta}\,\mathrm{d}x,\,\, J=\int_{\mathcal{W}_{r_0}}\frac{\Phi\big(\lambda|u|^{\frac{n}{n-1}}\big)}{\Big(1+d|u|^{\frac{n}{n-1}\big(1-\frac{\beta}{n}\big)}\Big)F^0(x)^\beta}\,\mathrm{d}x,$$
where $r_0$ is defined by $r_0:=\inf\{\eta\colon~u(\eta)\leq 1\}\in[0,+\infty)$. Since $u(\eta)\leq 1$ for all $\eta\geq r_0$, we may, without loss of generality, assume that $r_0=r_0(u)>0$. We first deal with $J_1$. By exploiting the fact that $u(x)\leq 1$ for all $x\in \mathbb{R}^n\setminus\mathcal{W}_{r_0}$ and arguing similarly as in Lemma \ref{lem3.3}, we obtain 
\begin{equation}\label{eq7.1}
\begin{aligned}
I & \leq \int_{\{x\in \mathbb{R}^n\colon~ u(x)\leq1\}}\frac{1}{F^0(x)^\beta}\sum_{j=n-1}^\infty \frac{\lambda^j}{j!}|u|^{\frac{nj}{n-1}}\,\mathrm{d}x\\
&\leq \int_{\{x\in \mathbb{R}^n\colon~ u(x)\leq1\}}\frac{1}{F^0(x)^\beta}\sum_{j=n-1}^\infty \frac{\lambda^j}{j!}|u|^{n}\,\mathrm{d}x \\ 
 & \leq \operatorname{\exp}(\lambda)\Bigg( \int_{\{x\in \mathbb{R}^n\colon~ u(x)\leq1,~F^0(x)\geq \|u\|_n\}}\frac{|u|^{n}}{F^0(x)^\beta}\,\mathrm{d}x\\
 &\qquad+  \int_{\{x\in \mathbb{R}^n\colon~u(x)\leq 1,~F^0(x)<\|u\|_n\}}\frac{|u|^{n}}{F^0(x)^\beta}\,\mathrm{d}x\Bigg)\\ 
& \leq \operatorname{\exp}(\lambda)\bigg(1+\frac{n\kappa_n}{n-\beta}\bigg)\|u\|_n^{n\big(1-\frac{\beta}{n}\big)}\\
&\leq C(n,\beta)\|u\|_n^{n\big(1-\frac{\beta}{n}\big)},
\end{aligned}
\end{equation}
thanks to $\lambda\leq\lambda_{\beta,n}$, where $C(n, \beta)>0$ is a constant depending solely on $n$ and $\beta$. We now fix an arbitrary real number $\theta\in(0,1)$. In addition, one may choose $r_1=r_1(u)$ in such a way that \eqref{eq4.4} and \eqref{eq4.5} are satisfied. It is worth noting that \eqref{eq4.7} and \eqref{eq4.8} still hold. Thus, to complete the proof, it remains to show that
\begin{align}\label{eq7.2}
 J=\int_{\mathcal{W}_{r_0}}\frac{\Phi\big(\lambda|u|^{\frac{n}{n-1}}\big)}{\Big(1+d|u|^{\frac{n}{n-1}\big(1-\frac{\beta}{n}\big)}\Big)F^0(x)^\beta}\,\mathrm{d}x\leq C\|u\|_n^{n\big(1-\frac{\beta}{n}\big)},   
\end{align}
where $C>0$ is a constant independent of $u$. To evaluate the integral over Wulff ball $\mathcal{W}_{r_0}$, we distinguish two cases: $r_1 \geq r_0$ and $r_1 < r_0$, which are treated separately below.

\textbf{Case-1:} Fix $0<r_0\leq r_1$. Then, for all $0<\eta\leq r_0$, it follows from \eqref{eq4.7} that
$$ u(\eta)\leq 1+\bigg(\frac{\theta}{n\kappa_n}\bigg)^{\frac{1}{n}}\bigg(\ln \bigg(\frac{r_0}{\eta}\bigg)\bigg)^{\frac{n-1}{n}}.$$
We now set $\delta = 1 - \theta^{\frac{1}{n-1}}$. This yields $(1+\delta)\theta^{\frac{1}{n-1}} < 1$. By using $\lambda \leq \lambda_{\beta,n}$ and \eqref{eq3.6}, the preceding estimate ensures that
\begin{align*}
    J&\leq \int_{\mathcal{W}_{r_0}}\frac{\operatorname{\exp}\Big(\lambda_n\Big(1-\frac{\beta}{n}\Big)|u|^{\frac{n}{n-1}}\Big)}{F^0(x)^\beta}\,\mathrm{d}x\\
    &=\int_{0}^{r_0}\frac{\operatorname{\exp}\Big(\lambda_n\Big(1-\frac{\beta}{n}\Big)|u(\eta)|^{\frac{n}{n-1}}\Big)}{\eta^\beta}\int_{\partial\mathcal{W}_\eta}\frac{1}{|\nabla F^0(x)|}\,\mathrm{d}\sigma\,\mathrm{d}\eta\\
    &\leq n\kappa_n\int_{0}^{r_0}\operatorname{\exp}\Bigg( \lambda_n\Big(1-\frac{\beta}{n}\Big)(1+\delta)\bigg(\frac{\theta}{n\kappa_n}\bigg)^{\frac{1}{n-1}}\ln \bigg(\frac{r_0}{\eta}\bigg)+\lambda_n \Big(1-\frac{\beta}{n}\Big)C_\delta\Bigg)\eta^{n-\beta-1}\,\mathrm{d}\eta\\
    &=n\kappa_n\operatorname{\exp}\Big(\lambda_n \Big(1-\frac{\beta}{n}\Big)C_\delta\Big)\int_{0}^{r_0}\bigg(\frac{r_0}{\eta}\bigg)^{n\big(1-\frac{\beta}{n}\big)(1+\delta)\theta^{\frac{1}{n-1}}}\eta^{n-\beta-1}\,\mathrm{d}\eta\\
    &=\frac{n\kappa_n\operatorname{\exp}\Big(\lambda_n \Big(1-\frac{\beta}{n}\Big)C_\delta\Big)}{(n-\beta)\big(1-(1+\delta)\theta^{\frac{1}{n-1}}\big)}~r_0^{n\big(1-\frac{\beta}{n}\big)}\\
    &=\frac{n\kappa_n\operatorname{\exp}\Big(\lambda_n \Big(1-\frac{\beta}{n}\Big)C_\delta\Big)}{(n-\beta)\big(1-(1+\delta)\theta^{\frac{1}{n-1}}\big)}~\Bigg(\frac{1}{\kappa_n}\int_{0}^{r_0}\int_{\partial\mathcal{W}_\eta}\frac{1}{|\nabla F^0(x)|}\,\mathrm{d}\sigma\,\mathrm{d}\eta\Bigg)^{1-\frac{\beta}{n}}\\
    &\leq \frac{n\kappa_n^{\frac{\beta}{n}}\operatorname{\exp}\Big(\lambda_n \Big(1-\frac{\beta}{n}\Big)C_\delta\Big)}{(n-\beta)\big(1-(1+\delta)\theta^{\frac{1}{n-1}}\big)}~\Bigg(\int_{0}^{r_0}|u(\eta)|^n\int_{\partial\mathcal{W}_\eta}\frac{1}{|\nabla F^0(x)|}\,\mathrm{d}\sigma\,\mathrm{d}\eta\Bigg)^{1-\frac{\beta}{n}}\\
    &=\frac{n\kappa_n^{\frac{\beta}{n}}\operatorname{\exp}\Big(\lambda_n \Big(1-\frac{\beta}{n}\Big)C_\delta\Big)}{(n-\beta)\big(1-(1+\delta)\theta^{\frac{1}{n-1}}\big)}~\Bigg(\int_{\mathcal{W}_{r_0}}|u|^n\,\mathrm{d}x\Bigg)^{1-\frac{\beta}{n}}\\
    &\leq C(n,\beta,\delta)\|u\|_n^{n\big(1-\frac{\beta}{n}\big)},
\end{align*}
where $C(n,\beta,\delta)>0$ is a constant depending only on $n$, $\beta$ and $\delta$.

\textbf{Case-2:} Fix $0<r_1< r_0$. We then split the integral over $\mathcal{W}_{r_0}$ in the following manner
\begin{align*}
 J&=\int_{\mathcal{W}_{r_0}}\frac{\Phi(\lambda|u|^{\frac{n}{n-1}})}{\big(1+d|u|^{\frac{n}{n-1}(1-\frac{\beta}{n})}\big)F^0(x)^\beta}\mathrm{d}x\\&=\bigg(\int_{\mathcal{W}_{r_0}\setminus\mathcal{W}_{r_1}}+\int_{\mathcal{W}_{r_1}}\bigg)\frac{\Phi(\lambda|u|^{\frac{n}{n-1}})}{\big(1+d|u|^{\frac{n}{n-1}(1-\frac{\beta}{n})}\big)F^0(x)^\beta}\mathrm{d}x=:J_1+J_2.   
\end{align*}
In virtue of \eqref{eq4.8}, it holds for all $r_1<\eta<r_0$ that
$$  u(\eta)\leq 1+\bigg(\frac{1-\theta}{n\kappa_n}\bigg)^{\frac{1}{n}}\bigg(\ln \bigg(\frac{r_0}{\eta}\bigg)\bigg)^{\frac{n-1}{n}}.$$
By choosing $\delta=1-(1-\theta)^{\frac{1}{n-1}}$, we obtain  $(1+\delta)(1-\theta)^{\frac{1}{n-1}}<1$. Repeating the arguments used in Case-$1$, after minor adjustments, yields from the above inequality and \eqref{eq3.6} that
$$J_1:=\int_{\mathcal{W}_{r_0}\setminus\mathcal{W}_{r_1}}\frac{\Phi\big(\lambda|u|^{\frac{n}{n-1}}\big)}{\Big(1+d|u|^{\frac{n}{n-1}\big(1-\frac{\beta}{n}\big)}\Big)F^0(x)^\beta}\,\mathrm{d}x\leq C(n,\beta,\delta)\|u\|_n^{n\big(1-\frac{\beta}{n}\big)},$$
where $C(n,\beta,\delta)>0$ is a constant depending only on $n$, $\beta$ and $\delta$. To proceed, we introduce the auxiliary function
$\xi\colon [0,r_1]\to\mathbb{R}$ by
 $\xi(\eta)=u(\eta)-u(r_1)$ for all $ 0\leq \eta\leq r_1$, which will be used to estimate $J_2$. Thus, it is immediate that $\xi\in W^{1,n}_0(\mathcal{W}_{r_1})$ and $F(\nabla \xi)=F(\nabla u)$ in $\mathcal{W}_{r_1}$. We further observe that
 $$u(r_1)>1\quad\text{and}\quad\int_{\mathbb{R}^n\setminus \mathcal{W}_{r_1}}F(\nabla u)^n\,\mathrm{d}x\leq 1-\theta.$$
 Consequently, we obtain from Proposition \ref{prop4.1} that
 $$\frac{\operatorname{\exp\bigg(\frac{\lambda_n(1-\frac{\beta}{n}) u^{\frac{n}{n-1}}(r_1)}{(1-\theta)^{\frac{1}{n-1}}}\bigg)}}{u^{\frac{n}{n-1}(1-\frac{\beta}{n})}(r_1)}\leq \frac{C(n,\beta,\theta)}{r_1^{n(1-\frac{\beta}{n})}}\Bigg(\int_{\mathbb{R}^n\setminus\mathcal{W}_{r_1}}|u|^n\,\mathrm{d}x\Bigg)^{1-\frac{\beta}{n}},$$
 where $C(n,\beta,\theta)>0$ is a constant depending only on $n$, $\beta$ and $\theta$. By setting $\delta=\big(\frac{1}{\theta}\big)^{\frac{1}{n-1}}-1$, we deduce from \eqref{eq3.6} that $C_\delta=(1-\theta)^{-\frac{1}{n-1}}$. Define $\psi\colon [0,r_1]\to\mathbb{R}$ by $\psi(\eta)=(1+\delta)^{\frac{n-1}{n}}\xi(\eta)$ for all $0\leq\eta\leq r_1$. This yields at once that 
 $$\int_{\mathcal{W}_{r_1}}F(\nabla \psi)^n\,\mathrm{d}x=(1+\delta)^{n-1}\int_{\mathcal{W}_{r_1}}F(\nabla \xi)^n\,\mathrm{d}x\leq (1+\delta)^{n-1}\theta=1.$$
 By gathering the above facts, we obtain from \eqref{eq3.6} and the anisotropic singular Moser--Trudinger inequality in bounded domains (see Zhu \cite[Theorem 2.1]{Zhu-2020}) that
\begin{align*}
J_2&\leq\int_{\mathcal{W}_{r_1}}\frac{\operatorname{\exp}\Big(\lambda_n\big(1-\frac{\beta}{n}\big)|u|^{\frac{n}{n-1}}\Big)}{d|u|^{\frac{n}{n-1}(1-\frac{\beta}{n})} F^0(x)^\beta}\,\mathrm{d}x\\
&\leq\frac{1}{du^{\frac{n}{n-1}(1-\frac{\beta}{n})}(r_1)}\int_{\mathcal{W}_{r_1}}\frac{\operatorname{\exp}\Big(\lambda_n\big(1-\frac{\beta}{n}\big)(1+\delta)|\xi|^{\frac{n}{n-1}}+\lambda_n \big(1-\frac{\beta}{n}\big)C_\delta u^{\frac{n}{n-1}}(r_1)\Big)}{F^0(x)^\beta}\,\mathrm{d}x\\
&=\frac{\operatorname{\exp\bigg(\frac{\lambda_n(1-\frac{\beta}{n}) u^{\frac{n}{n-1}}(r_1)}{(1-\theta)^{\frac{1}{n-1}}}\bigg)}}{du^{\frac{n}{n-1}(1-\frac{\beta}{n})}(r_1)}\int_{\mathcal{W}_{r_1}}\frac{\operatorname{\exp}\Big(\lambda_n\big(1-\frac{\beta}{n}\big)(1+\delta)|\xi|^{\frac{n}{n-1}}\Big)}{F^0(x)^\beta}\,\mathrm{d}x\\
&\leq \frac{C(n,\beta,\theta)}{d r_1^{n(1-\frac{\beta}{n})}}\|u\|_n^{n\big(1-\frac{\beta}{n}\big)}\int_{\mathcal{W}_{r_1}}\frac{\operatorname{\exp}\Big(\lambda_n\big(1-\frac{\beta}{n}\big)|\psi|^{\frac{n}{n-1}}\Big)}{F^0(x)^\beta}\,\mathrm{d}x\\
&\leq \hat{C}(n,\beta,\delta)\|u\|_n^{n\big(1-\frac{\beta}{n}\big)},
\end{align*}
where $\hat{C}(n,\beta,\delta)>0$ is a constant depending only on $n$, $\beta$ and $\delta$.

Due to the above discussion, one can notice that \eqref{eq7.2} holds. Moreover, the desired inequality \eqref{eq1.17} follows immediately from \eqref{eq7.1} and \eqref{eq7.2}. The proof will be complete after showing that both the exponent $q$ and the constant $\lambda_{\beta,n}$ are sharp. We first verify the sharpness of the exponent $q$. For this, we invoke the Moser-type sequence of functions $\{v_k\}_k$ introduced previously in the proof of Theorem \ref{thm1.1}, enjoying the properties that
\begin{align*}
\|F(\nabla v_k)\|^n_n=1 \qquad \text{and}\qquad \|v_k\|_n^n=O\Big(\frac{1}{\ln k}\Big)\quad\text{as}\quad k\to\infty.   
\end{align*}
Accordingly, we obtain
\begin{align*}
\int_{\mathbb{R}^n}&\frac{\Phi\big(\lambda_{\beta,n}|v_k|^{\frac{n}{n-1}}\big)}{\Big(1+d|v_k|^{\frac{q}{n-1}\big(1-\frac{\beta}{n}\big)}\Big)F^0(x)^\beta}\,\mathrm{d}x\\ &\geq \int_{\{x\in\mathbb{R}^n\colon F^0(x)\leq sk^{-1}\}}\frac{\Phi\big(\lambda_{\beta,n}|v_k|^{\frac{n}{n-1}}\big)}{\Big(1+d|v_k|^{\frac{q}{n-1}\big(1-\frac{\beta}{n}\big)}\Big)F^0(x)^\beta}\,\mathrm{d}x\\
&\gtrsim \frac{\operatorname{\exp}\Big(\lambda_n\big(1-\frac{\beta}{n}\big)\Big|(n\kappa_n)^{-\frac{1}{n}}(\ln k)^{\frac{n-1}{n}}\Big|^{\frac{n}{n-1}}\Big)}{1+d\Big|(n\kappa_n)^{-\frac{1}{n}}(\ln k)^{\frac{n-1}{n}}\Big|^{\frac{q}{n-1}\big(1-\frac{\beta}{n}\big)}}\int_{0}^{sk^{-1}}\eta^{n-\beta-1}\,\mathrm{d}\eta\\
&\gtrsim \frac{\operatorname{\exp}\Big(n\big(1-\frac{\beta}{n}\big)\ln k\Big)}{1+d\Big|(n\kappa_n)^{-\frac{1}{n}}(\ln k)^{\frac{n-1}{n}}\Big|^{\frac{q}{n-1}\big(1-\frac{\beta}{n}\big)}}~(sk^{-1})^{n-\beta}\\ 
& \gtrsim (\ln k)^{-\frac{q}{n}\big(1-\frac{\beta}{n}\big)}.
\end{align*}
This gives that
$$\frac{1}{\|v_k\|_n^{n\big(1-\frac{\beta}{n}\big)}}\int_{\mathbb{R}^n}\frac{\Phi\big(\lambda_{\beta,n}|v_k|^{\frac{n}{n-1}}\big)}{\Big(1+d|v_k|^{\frac{q}{n-1}\big(1-\frac{\beta}{n}\big)}\Big)F^0(x)^\beta}\,\mathrm{d}x\gtrsim (\ln k)^{\big(1-\frac{\beta}{n}\big)\big(1-\frac{q}{n}\big)}\to+\infty$$
as $k\to\infty$, thanks to $q<n$. This verifies the sharpness of the exponent $q$. Next, our aim is to show the constant $\lambda_{\beta,n}$ is also sharp. To do this, let $\lambda>\lambda_{\beta,n}$. Then there exists $\delta_0>0$ such that $\lambda=(1+\delta_0)\lambda_{\beta,n}$. Consequently, we obtain
\begin{align*}
    \int_{\mathbb{R}^n}&\frac{\Phi\big(\lambda|v_k|^{\frac{n}{n-1}}\big)}{\Big(1+d|v_k|^{\frac{q}{n-1}\big(1-\frac{\beta}{n}\big)}\Big)F^0(x)^\beta}\,\mathrm{d}x\\ 
    &\geq \int_{\{x\in\mathbb{R}^n\colon F^0(x)\leq sk^{-1}\}}\frac{\Phi\big(\lambda|v_k|^{\frac{n}{n-1}}\big)}{\Big(1+d|v_k|^{\frac{q}{n-1}\big(1-\frac{\beta}{n}\big)}\Big)F^0(x)^\beta}\,\mathrm{d}x\\
    &\gtrsim \frac{\operatorname{\exp}\Big(\lambda\Big|(n\kappa_n)^{-\frac{1}{n}}(\ln k)^{\frac{n-1}{n}}\Big|^{\frac{n}{n-1}}\Big)}{1+d\Big|(n\kappa_n)^{-\frac{1}{n}}(\ln k)^{\frac{n-1}{n}}\Big|^{\frac{q}{n-1}\big(1-\frac{\beta}{n}\big)}}\int_{0}^{sk^{-1}}\eta^{n-\beta-1}\,\mathrm{d}\eta\\
    &\gtrsim \frac{\operatorname{\exp}\Big((1+\delta_0)n\big(1-\frac{\beta}{n}\big)\ln k\Big)}{1+d\Big|(n\kappa_n)^{-\frac{1}{n}}(\ln k)^{\frac{n-1}{n}}\Big|^{\frac{q}{n-1}\big(1-\frac{\beta}{n}\big)}}~(sk^{-1})^{n-\beta}\\ 
    & \gtrsim k^{\delta_0(n-\beta)}(\ln k)^{-\frac{q}{n}\big(1-\frac{\beta}{n}\big)}.
\end{align*}
This leads to
\begin{align*}
    \frac{1}{\|v_k\|_n^{n\big(1-\frac{\beta}{n}\big)}}\int_{\mathbb{R}^n}&\frac{\Phi\big(\lambda|v_k|^{\frac{n}{n-1}}\big)}{\Big(1+d|v_k|^{\frac{q}{n-1}\big(1-\frac{\beta}{n}\big)}\Big)F^0(x)^\beta}\,\mathrm{d}x\gtrsim  k^{\delta_0(n-\beta)}(\ln k)^{\big(1-\frac{\beta}{n}\big)\big(1-\frac{q}{n}\big)}\\
    &=\operatorname{\exp}\Bigg(\Big(1-\frac{\beta}{n}\Big)\ln k\bigg(\delta_0 n+\Big(1-\frac{q}{n}\Big)\frac{\ln(\ln k)}{\ln k}\bigg)\Bigg)\to+\infty 
\end{align*} 
as $k\to\infty$, thanks to $q\geq 0$. This proves the sharpness of $\lambda_{\beta,n}$, which in turn yields Theorem \ref{thm1.9}. Hence, the proof is completed.\qed

The next result establishes that the supremum defining $\textnormal{ASTME}_{n,d,\beta}(\lambda)$ may equivalently be taken over the class of Sobolev functions satisfying $\|u\|_n=1$.
\begin{lemma}\label{lem7.1}
For all $n\geq2$, $d>0$, $0\leq\beta<n$ and $0<\lambda<\lambda_{\beta,n}$, there holds
$$\textnormal{ASTME}_{n,d,\beta}(\lambda)=\sup_{u\in W^{1,n}(\mathbb{R}^n),~\|F(\nabla u)\|_n\leq 1,~\|u\|_n=1} \int_{\mathbb{R}^n}\frac{\Phi\big(\lambda|u|^{\frac{n}{n-1}}\big)}{\Big(1+d|u|^{\frac{n}{n-1}\big(1-\frac{\beta}{n}\big)}\Big)F^0(x)^\beta}\,\mathrm{d}x.$$
\end{lemma}
\begin{proof}
 Let $u\in W^{1,n}(\mathbb{R}^n)$ be such that $\|F(\nabla u)\|_n\leq 1$. Define the function $w\colon \mathbb{R}^n\to \mathbb{R}$ by $w(x)=u(\theta x)$ for all $x\in \mathbb{R}^n$, where $\theta=\|u\|_n>0$. By direct calculations, we see that $w\in W^{1,n}(\mathbb{R}^n)$ such that
 $$\|F(\nabla w)\|_n=\|F(\nabla u)\|_n\leq 1 \quad\text{and}\quad \|w\|_n=\theta^{-1}\|u\|_n=1.$$
 This shows that
 \begin{align*}
     \int_{\mathbb{R}^n}&\frac{\Phi\big(\lambda|w(x)|^{\frac{n}{n-1}}\big)}{\Big(1+d|w(x)|^{\frac{n}{n-1}\big(1-\frac{\beta}{n}\big)}\Big)F^0(x)^\beta}\,\mathrm{d}x\\&=\int_{\mathbb{R}^n}\frac{\Phi\big(\lambda|u(\theta x)|^{\frac{n}{n-1}}\big)}{\Big(1+d|u(\theta x)|^{\frac{n}{n-1}\big(1-\frac{\beta}{n}\big)}\Big)F^0(x)^\beta}\,\mathrm{d}x\\
     &=\frac{1}{\theta^{n-\beta}}\int_{\mathbb{R}^n}\frac{\Phi\big(\lambda|u(x)|^{\frac{n}{n-1}}\big)}{\Big(1+d|u(x)|^{\frac{n}{n-1}\big(1-\frac{\beta}{n}\big)}\Big)F^0(x)^\beta}\,\mathrm{d}x\\
     &=\frac{1}{\|u\|_n^{n\big(1-\frac{\beta}{n}\big)}}\int_{\mathbb{R}^n}\frac{\Phi\big(\lambda|u|^{\frac{n}{n-1}}\big)}{\Big(1+d|u|^{\frac{n}{n-1}\big(1-\frac{\beta}{n}\big)}\Big)F^0(x)^\beta}\,\mathrm{d}x.
 \end{align*}
 Thus, we deduce that
 \begin{align*}
     \textnormal{ASTME}_{n,d,\beta}(\lambda)&\leq\sup_{\substack{w\in W^{1,n}(\mathbb{R}^n),\\  \|F(\nabla w)\|_n\leq 1,~\|w\|_n=1}}\int_{\mathbb{R}^n}\frac{\Phi\big(\lambda|w|^{\frac{n}{n-1}}\big)}{\Big(1+d|w|^{\frac{n}{n-1}\big(1-\frac{\beta}{n}\big)}\Big)F^0(x)^\beta}\,\mathrm{d}x\\
     &\leq\sup_{\substack{u\in W^{1,n}(\mathbb{R}^n),\\  \|F(\nabla u)\|_n\leq 1}} \,\,\frac{1}{\|u\|_n^{n\big(1-\frac{\beta}{n}\big)}}\int_{\mathbb{R}^n}\frac{\Phi\big(\lambda|u|^{\frac{n}{n-1}}\big)}{\Big(1+d|u|^{\frac{n}{n-1}\big(1-\frac{\beta}{n}\big)}\Big)F^0(x)^\beta}\,\mathrm{d}x\\
     &=\textnormal{ASTME}_{n,d,\beta}(\lambda).
 \end{align*}
Hence, the proof is completed.
\end{proof}
The following lemma is devoted to deriving lower bounds for $\textnormal{ASTME}_{n,d,\beta}(\lambda)$ in the special case $\beta=0$. The complete proof has already been provided by Liu \cite{Liu-2022} and is therefore omitted; see, for instance, the proof of Theorem $1.2(i)$.
\begin{lemma}\label{lem7.2}
  Assume that  $n\geq2$, $d>0$, $\beta=0$ and $0<\lambda<\lambda_{n}$. Then the following assertions are hold
  \begin{itemize}
      \item [\textnormal{(a)}] $\textnormal{ASTME}_{n,d,0}(\lambda)\geq\frac{\lambda^{n-1}}{(n-1)!}$.
      \item [\textnormal{(b)}] $\textnormal{ASTME}_{n,d,0}(\lambda)>\frac{\lambda^{n-1}}{(n-1)!}\quad\text{for all}\quad 0<d\leq\frac{\lambda}{n}$.
  \end{itemize}
\end{lemma}
\begin{remark}
\textnormal{We point out that Liu \cite{Liu-2022} determined the exact value of the supremums and established several existence and nonexistence results for maximizers in the special case $\beta=0$. However, there appears to be a gap in the proof Theorem $1.2$ to the estimate the integral over the region
$$\{x\in\mathbb{R}^n\colon~ F^0(x)\leq R_\varepsilon,~ |v_k|\geq T_\varepsilon\}.$$
Indeed, to justify the estimate
\begin{align*}
 &\int_{\{x\in\mathbb{R}^n\colon\, F^0(x)\leq R_\varepsilon,\, |v_k|\geq T_\varepsilon\}}\left(
\frac{\Phi(\lambda|v_k|^{\frac{n}{n-1}})}{1+b|v_k|^{\frac{n}{n-1}}}-\frac{\lambda^{n-1}}{(n-1)!}|v_k|^n
\right)\mathrm{d}x\\&
\leq\varepsilon\int_{\mathbb{R}^n}\Phi(\lambda|v_k|^{\frac{n}{n-1}})\mathrm{d}x\leq\varepsilon C,   
\end{align*}
one would require an appropriate Moser--Trudinger type inequality, which does not seem to be available in the existing literature. The above argument was based on the estimate
$$(1+b|v_k|^{\frac{n}{n-1}})^{-1}\leq \varepsilon
\quad \text{for}\quad |v_k|\geq T_\varepsilon,$$
under the assumptions
$$
0<\lambda\leq\lambda_n,\quad\|F(\nabla v_k)\|_n\leq1\quad\text{and}\quad\|v_k\|_n=1.$$
Hence, we revisit the special case $\beta=0$ separately and provide a corrected argument for completeness.}
\end{remark}
Finally, we end this section by establishing Theorem \ref{thm1.10}.
\subsection{{Proof of Theorem \ref{thm1.10}}}
Let $\{u_k\}_k\subset W^{1,n}(\mathbb{R}^n)\setminus\{0\}$ be a maximizing sequence for $\mathrm{ASTME}_{n,d,\beta}(\lambda)$. By Lemma \ref{lem7.1}, we may assume that
 $\|u_k\|_n=1$, $\|F(\nabla u_k)\|_n\leq 1$ and 
$$\lim_{k\to\infty}\int_{\mathbb{R}^n}\frac{\Phi\big(\lambda|u_k|^{\frac{n}{n-1}}\big)}{\Big(1+d|u_k|^{\frac{n}{n-1}\big(1-\frac{\beta}{n}\big)}\Big)F^0(x)^\beta}\,\mathrm{d}x=\textnormal{ASTME}_{n,d,\beta}(\lambda). $$
Moreover, arguing as in the proof of Theorem \ref{thm1.9}, we may further assume that $\{u_k\}_k$ is a sequence of nonnegative, Wulff radially symmetric, and nonincreasing functions in $W^{1,n}(\mathbb{R}^n)$. Note that $\{u_k\}_k$ is a bounded sequence in $W^{1,n}(\mathbb{R}^n)$. Hence, passing to a subsequence if necessary, still denoted by itself, there exist $u\in W^{1,n}(\mathbb{R}^n)$ and $q\geq 1$ such that Lemma \ref{lem2.3} yields that
\begin{align}\label{eq7.3}
u_k\rightharpoonup u~~\text{in}~~ W^{1,n}(\mathbb{R}^n),\,
u_k\to u ~~\text{in}~~ L^q _{\textit{loc}}(\mathbb{R}^n),~~\text{and}~~
u_k\to u~~\text{a.e. in}~~\mathbb{R}^n 
\end{align}
as $k\to\infty$. In view of \eqref{eq7.3}, again by exploiting Lemma \ref{lem2.3} and Cheng--Zhang--Zhou \cite[Lemma 6.2]{Cheng-Zhang-Zhou-2023}, we also have $\|u\|_n\leq 1$ and $\|F(\nabla u)\|_n\leq 1$. It is worth noting that the rest of the proof splits naturally into the following two cases.

\textbf{Case-1:} In this case, we discuss about $\beta>0$. To start with, we fix $r>0$. By using $\|u_k\|_n=1$, we obtain from Lemma \ref{lem2.2} that
$$u_k(r)\leq \bigg(\frac{1}{\kappa_n}\bigg)^{\frac{1}{n}} r^{-1}\quad\text{for all}\quad k\in\mathbb{N}.$$
Let $\varepsilon>0$ be arbitrary and  denote $r_\varepsilon=\kappa_n^{-\frac{1}{n}}\varepsilon^{-1}$. By the preceding inequality, we infer that $u_k(r)\leq \varepsilon$ for all $r\geq r_\varepsilon$. Next, we divide the the whole integral as follows
\begin{align*}
 \int_{\mathbb{R}^n}\frac{\Phi(\lambda|u_k|^{\frac{n}{n-1}})}{\big(1+d|u_k|^{\frac{n}{n-1}(1-\frac{\beta}{n})}\big)F^0(x)^\beta}\,\mathrm{d}x&=\bigg(\int_{\mathbb{R}^n\setminus \mathcal{W}_{r_\varepsilon}}+\int_{\mathcal{W}_{r_\varepsilon}}\bigg)\frac{\Phi(\lambda|u_k|^{\frac{n}{n-1}})}{\big(1+d|u_k|^{\frac{n}{n-1}(1-\frac{\beta}{n})}\big)F^0(x)^\beta}\,\mathrm{d}x\\
 &=:I+J.   
\end{align*}
To estimate $I$, we employ Theorem \ref{thm1.2}, from which it follows that
\begin{align*}
    I&\leq \int_{\mathbb{R}^n\setminus \mathcal{W}_{r_\varepsilon}}\frac{\Phi\big(\lambda|u_k|^{\frac{n}{n-1}}\big)}{F^0(x)^\beta}\,\mathrm{d}x\lesssim\frac{1}{r_\varepsilon^{\frac{\beta}{2}}}\int_{\mathbb{R}^n\setminus \mathcal{W}_{r_\varepsilon}}\frac{\operatorname{\exp}\big(\lambda|u_k|^{\frac{n}{n-1}}\big)|u_k|^n}{F^0(x)^{\frac{\beta}{2}}}\,\mathrm{d}x\\
    &\lesssim \frac{1}{r_\varepsilon^{\frac{\beta}{2}}}\|u_k\|_n^{n\big(1-\frac{\beta}{2n}\big)}= \frac{1}{r_\varepsilon^{\frac{\beta}{2}}}\to 0\quad\text{as}\quad \varepsilon\to 0.
\end{align*}
Notice that
$$\frac{1}{1+d|t|^{\frac{n}{n-1}\big(1-\frac{\beta}{n}\big)}}\to 0 \quad\text{as}\quad |t|\to \infty.$$
It follows that for each $\varepsilon>0$, we can find $t_\varepsilon>0$ large enough such that
$$\frac{1}{1+d|t|^{\frac{n}{n-1}\big(1-\frac{\beta}{n}\big)}}\leq \varepsilon\quad\text{for all}\quad |t|\geq t_\varepsilon.$$
Moreover, let us write $J$ as follows
\begin{align*}
    J&=\Bigg(\int_{\{x\in \mathcal{W}_{r_\varepsilon}\colon~|u_k|\geq t_\varepsilon\}}+\int_{\{x\in \mathcal{W}_{r_\varepsilon}\colon~|u_k|< t_\varepsilon\}}\Bigg)\frac{\Phi\big(\lambda|u_k|^{\frac{n}{n-1}}\big)}{\Big(1+d|u_k|^{\frac{n}{n-1}\big(1-\frac{\beta}{n}\big)}\Big)F^0(x)^\beta}\,\mathrm{d}x\\
    &=:J_{1,\varepsilon}+J_{2,\varepsilon}.
\end{align*}
Again, by applying Theorem \ref{thm1.2}, we obtain 
\begin{align*}
    J_{1,\varepsilon}&\leq \varepsilon \int_{\{x\in \mathcal{W}_{r_\varepsilon}\colon ~|u_k|\geq t_\varepsilon\}}\frac{\Phi\big(\lambda|u_k|^{\frac{n}{n-1}}\big)}{F^0(x)^\beta}\,\mathrm{d}x\lesssim\varepsilon\int_{\{x\in \mathcal{W}_{r_\varepsilon}\colon~|u_k|\geq t_\varepsilon\}}\frac{\operatorname{\exp}\big(\lambda|u_k|^{\frac{n}{n-1}}\big)|u_k|^n}{F^0(x)^{\beta}}\,\mathrm{d}x\\
    &\lesssim \varepsilon\|u_k\|_n^{n\big(1-\frac{\beta}{n}\big)}= \varepsilon\to 0\quad\text{as}\quad \varepsilon\to 0.
\end{align*}
We now turn to $J_{2,\varepsilon}$. To this end, by using \eqref{eq7.3}, we obtain 
$$ \frac{\Phi\big(\lambda|u_k|^{\frac{n}{n-1}}\big)\chi_{\{|u_k|<t_\varepsilon\}}}{\Big(1+d|u_k|^{\frac{n}{n-1}\big(1-\frac{\beta}{n}\big)}\Big)F^0(x)^\beta}\to \frac{\Phi\big(\lambda|u|^{\frac{n}{n-1}}\big)\chi_{\{|u|<t_\varepsilon\}}}{\Big(1+d|u|^{\frac{n}{n-1}\big(1-\frac{\beta}{n}\big)}\Big)F^0(x)^\beta}\quad\text{a.e. in}\quad \mathcal{W}_{r_\varepsilon} $$
as $k\to\infty$, and 
$$\frac{\Phi\big(\lambda|u_k|^{\frac{n}{n-1}}\big)\chi_{\{|u_k|<t_\varepsilon\}}}{\Big(1+d|u_k|^{\frac{n}{n-1}\big(1-\frac{\beta}{n}\big)}\Big)F^0(x)^\beta}\leq \frac{\Phi\big(\lambda t_\varepsilon^{\frac{n}{n-1}}\big)}{F^0(x)^\beta}\in L^1(\mathcal{W}_{r_\varepsilon}).$$
Invoking the Lebesgue Dominated Convergence Theorem, we obtain
$$\lim_{k\to\infty}J_{2,\varepsilon} =\int_{\{x\in \mathcal{W}_{r_\varepsilon}\colon~|u|< t_\varepsilon\}}\frac{\Phi\big(\lambda|u|^{\frac{n}{n-1}}\big)}{\Big(1+d|u|^{\frac{n}{n-1}\big(1-\frac{\beta}{n}\big)}\Big)F^0(x)^\beta}\,\mathrm{d}x.$$
Passing to the limit as $\varepsilon \to 0$, we deduce from the above information that
$$ \textnormal{ASTME}_{n,d,\beta}(\lambda)\leq \int_{\mathbb{R}^n}\frac{\Phi\big(\lambda|u|^{\frac{n}{n-1}}\big)}{\Big(1+d|u|^{\frac{n}{n-1}\big(1-\frac{\beta}{n}\big)}\Big)F^0(x)^\beta}\,\mathrm{d}x.$$
Consequently, $u\neq 0$ and there holds
\begin{align*}
  \textnormal{ASTME}_{n,d,\beta}(\lambda)&\leq \int_{\mathbb{R}^n}\frac{\Phi\big(\lambda|u|^{\frac{n}{n-1}}\big)}{\Big(1+d|u|^{\frac{n}{n-1}\big(1-\frac{\beta}{n}\big)}\Big)F^0(x)^\beta}\,\mathrm{d}x\\
  &\leq \frac{1}{\|u\|_n^{n\big(1-\frac{\beta}{n}\big)}} \int_{\mathbb{R}^n}\frac{\Phi\big(\lambda|u|^{\frac{n}{n-1}}\big)}{\Big(1+d|u|^{\frac{n}{n-1}\big(1-\frac{\beta}{n}\big)}\Big)F^0(x)^\beta}\,\mathrm{d}x\\
  &\leq \textnormal{ASTME}_{n,d,\beta}(\lambda). 
\end{align*}
This shows that $\|u\|_n=1$ and thus, $u$ is a maximizer of $\textnormal{ASTME}_{n,d,\beta}(\lambda)$.

\textbf{Case-2:} In this case, we discuss about $\beta=0$. Define the function $G\colon\mathbb{R}\to \mathbb{R}$ by
$$G(s):=\frac{\Phi(\lambda|s|^{\frac{n}{n-1}})}{1+d|s|^{\frac{n}{n-1}}}-\frac{\lambda^{n-1}}{(n-1)!}|s|^n\quad\text{for all}\quad s\in \mathbb{R}. $$
It follows that
$$\textnormal{ASTME}_{n,d,0}(\lambda)=\lim_{k\to\infty} \int_{\mathbb{R}^n}\frac{\Phi(\lambda|u_k|^{\frac{n}{n-1}})}{1+d|u_k|^{\frac{n}{n-1}}}\,\mathrm{d}x=\frac{\lambda^{n-1}}{(n-1)!}+\lim_{k\to\infty} \int_{\mathbb{R}^n}G(u_k)\,\mathrm{d}x,$$
thanks to $\|u_k\|_n=1$. Divide the whole integral as follows
$$\int_{\mathbb{R}^n}G(u_k)\,\mathrm{d}x=\bigg(\int_{\mathbb{R}^n\setminus\mathcal{W}_{r_\varepsilon}}+\int_{\mathcal{W}_{r_\varepsilon}}\bigg)G(u_k)\,\mathrm{d}x:=I+II.$$
Recall that $u_k(x)\leq \varepsilon$ for all $x\in \mathbb{R}^n\setminus \mathcal{W}_{r_\varepsilon}$. Applying Theorem \ref{thm1.2} leads to
\begin{align*}
I&\leq \int_{\mathbb{R}^n\setminus\mathcal{W}_{r_\varepsilon}}\bigg(\frac{\Phi(\lambda|u_k|^{\frac{n}{n-1}})}{1+d|u_k|^{\frac{n}{n-1}}}-\frac{\lambda^{n-1}}{(n-1)!}\frac{|u_k|^n}{1+d|u_k|^{\frac{n}{n-1}}}\bigg)\,\mathrm{d}x\\
&=\int_{\mathbb{R}^n\setminus\mathcal{W}_{r_\varepsilon}}\sum_{j=n}^{\infty}\frac{\lambda^j}{j!}\frac{|u_k|^{\frac{nj}{n-1}}}{1+d|u_k|^{\frac{n}{n-1}}}\,\mathrm{d}x\\
&\leq \int_{\mathbb{R}^n\setminus\mathcal{W}_{r_\varepsilon}}\lambda|u_k|^{\frac{n}{n-1}}\sum_{j=n-1}^{\infty}\frac{\lambda^j}{j!}\frac{|u_k|^{\frac{nj}{n-1}}}{1+d|u_k|^{\frac{n}{n-1}}}\,\mathrm{d}x\\
&\leq \lambda\varepsilon^{\frac{n}{n-1}}\int_{\mathbb{R}^n\setminus\mathcal{W}_{r_\varepsilon}}\frac{\Phi(\lambda|u_k|^{\frac{n}{n-1}})}{1+d|u_k|^{\frac{n}{n-1}}}\,\mathrm{d}x\\ &\lesssim \lambda\varepsilon^{\frac{n}{n-1}} \int_{\mathbb{R}^n\setminus \mathcal{W}_{r_\varepsilon}}\operatorname{\exp}\big(\lambda|u_k|^{\frac{n}{n-1}}\big)|u_k|^n\,\mathrm{d}x\\
&\lesssim \lambda\varepsilon^{\frac{n}{n-1}}\|u_k\|_n^n\lesssim\varepsilon^{\frac{n}{n-1}}\to 0\quad\text{as}\quad \varepsilon\to 0.
\end{align*}
Moreover, one has
$$\frac{1}{1+d|t|^{\frac{n}{n-1}}}\to 0 \quad\text{as}\quad |t|\to \infty.$$
Thus, for each $\varepsilon>0$, one can find $t_\varepsilon>0$ large enough such that
$$\frac{1}{1+d|t|^{\frac{n}{n-1}}}\leq \varepsilon\quad\text{for all}\quad |t|\geq t_\varepsilon.$$
As a result, we further split $II$ as follows
$$ II=\Bigg(\int_{\{x\in \mathcal{W}_{r_\varepsilon}\colon~|u_k|\geq t_\varepsilon\}}+\int_{\{x\in \mathcal{W}_{r_\varepsilon}\colon~|u_k|< t_\varepsilon\}}\Bigg)G(u_k)\,\mathrm{d}x=:II_{a}+II_{b}.$$
In virtue of Theorem \ref{thm1.2}, we also obtain
\begin{align*}
  II_{a}&\leq  \int_{\{x\in \mathcal{W}_{r_\varepsilon}\colon~|u_k|\geq t_\varepsilon\}} \frac{\Phi(\lambda|u_k|^{\frac{n}{n-1}})}{1+d|u_k|^{\frac{n}{n-1}}}\,\mathrm{d}x\\
  &\lesssim \varepsilon \int_{\{x\in \mathcal{W}_{r_\varepsilon}\colon~|u_k|\geq t_\varepsilon\}}\operatorname{\exp}\big(\lambda|u_k|^{\frac{n}{n-1}}\big)|u_k|^n\,\mathrm{d}x\\
  &\lesssim \lambda\varepsilon\|u_k\|_n^n=\lambda\varepsilon\to 0\quad\text{as}\quad \varepsilon\to 0.
\end{align*}
Now, by arguing similarly to Case-$1$, we infer from the Lebesgue dominated convergence theorem that
$$\lim_{k\to\infty}II_{b} =\int_{\{x\in \mathcal{W}_{r_\varepsilon}\colon~|u|< t_\varepsilon\}}\left(\frac{\Phi(\lambda|u|^{\frac{n}{n-1}})}{1+d|u|^{\frac{n}{n-1}}}-\frac{\lambda^{n-1}}{(n-1)!}|u|^n\right)\,\mathrm{d}x.$$
Taking the limit as $\varepsilon \to 0$, it follows from the above estimates that
$$ \textnormal{ASTME}_{n,d,0}(\lambda)\leq\frac{\lambda^{n-1}}{(n-1)!}+ \int_{\mathbb{R}^n}G(u)\,\mathrm{d}x.$$
Notice that if we take $\textnormal{ASTME}_{n,d,0}(\lambda)>\frac{\lambda^{n-1}}{(n-1)!}$, then $u\neq 0$ and there holds
\begin{align*}
\textnormal{ASTME}_{n,d,0}(\lambda)&\leq\frac{\lambda^{n-1}}{(n-1)!}+ \int_{\mathbb{R}^n}G(u)\,\mathrm{d}x\leq \frac{\lambda^{n-1}}{(n-1)!}+\frac{1}{\|u\|_n^n} \int_{\mathbb{R}^n}G(u)\,\mathrm{d}x\\
&=\frac{1}{\|u\|_n^n} \int_{\mathbb{R}^n}\frac{\Phi\big(\lambda|u|^{\frac{n}{n-1}}\big)}{1+d|u|^{\frac{n}{n-1}}}\,\mathrm{d}x\leq \textnormal{ASTME}_{n,d,0}(\lambda).
\end{align*}
Hence, we conclude that $\|u\|_n=1$ and $u$ is a maximizer of $\textnormal{ASTME}_{n,d,0}(\lambda)$. It immediately follows that $\textnormal{ASTME}_{n,d,0}(\lambda)$ is attained only if $$\textnormal{ASTME}_{n,d,0}(\lambda)=\frac{\lambda^{n-1}}{(n-1)!}.$$
In addition, by using Lemma \ref{lem7.1}, one sees from the above discussion that $\textnormal{ASTME}_{n,d,0}(\lambda)$ can be attained for all $d\leq \frac{\lambda}{n}$. This finishes the proof.\qed

\section{Existence and Nonexistence of Ground State Solutions}\label{sec8} 
In this section, all the assumptions required for Theorem \ref{thm1.13} are understood to hold, unless explicitly stated otherwise. More precisely, this section is devoted to the study of existence and nonexistence phenomena for ground state solutions of \eqref{main problem} associated with appropriate ranges of the constant potential. It is evident that \eqref{main problem} has a variational formulation. Next, we define the corresponding Euler-Lagrange energy functional $J_\zeta \colon W^{1,n}(\mathbb{R}^n)\to \mathbb{R}$ by
\begin{align}\label{eq8.1}
   J_\zeta(u):=\frac{1}{n}\left(\int_{\mathbb{R}^n}F(\nabla u)^n\,\mathrm{d}x+\zeta\int_{\mathbb{R}^n}|u|^n\,\mathrm{d}x\right)-\int_{\mathbb{R}^n}G(u)\,\mathrm{d}x
\end{align}
for all $u\in W^{1,n}(\mathbb{R}^n)$. It easy to see that $J_\zeta$ is well-defined, of class $C^1(W^{1,n}(\mathbb{R}^n),\mathbb{R})$ and its G\^ateaux derivative is given by
\begin{align*}
    \left\langle{J_\zeta^\prime(u),v}\right\rangle=\int_{\mathbb{R}^n} F(\nabla u)^{n-1}F_\xi(\nabla u)\nabla v\,\mathrm{d}x+\zeta\int_{\mathbb{R}^n} |u|^{n-2} uv\,\mathrm{d}x-\int_{\mathbb{R}^n}g(u)v\,\mathrm{d}x
\end{align*}
for all $u,v\in W^{1,n}(\mathbb{R}^n)$. It follows from standard arguments that weak solutions of \eqref{main problem} are precisely the critical points of $J_\zeta$. To analyze the positive ground state solutions of \eqref{main problem}, we introduce the Nehari manifold corresponding to $J_\zeta$ as follows
$$\mathcal{N}_\zeta:=\left\{u\in W^{1,n}(\mathbb{R}^n)\setminus\{0\}\colon  \left\langle{J_\zeta^\prime(u),u}\right\rangle=0\right\}.$$
It can be seen that if $u\in \mathcal{N}_\zeta$, then
\begin{align}\label{eq8.2}
 J_\zeta(u)=\frac{1}{n}\int_{\mathbb{R}^n}\big(g(u)u-nG(u)\big)\,\mathrm{d}x.   
\end{align}
Notice that $\mathcal{N}_\zeta$ contains all nontrivial critical points of $J_\zeta$. Moreover, we say that $v$ is a ground state solution of \eqref{main problem}, if there holds $J_\zeta(v)=m_\zeta$, where 
$$ m_\zeta:=\inf\left\{J_\zeta(u)\colon\, u\in\mathcal{N}_\zeta\right\}. $$

The lemma below establishes that all nontrivial weak solutions of \eqref{main problem} are nonnegative.
\begin{lemma}\label{lem8.1}
     Every nontrivial critical point of the energy functional $J_\zeta$ is nonnegative.
 \end{lemma}
\begin{proof}
    Let $u\in W^{1,n}(\mathbb{R}^n)\setminus\{0\}$ be a critical point of $J_\zeta$. Define $u=u^+-u^-$, where $u^{\pm}=\max\{\pm u,0\}$. Then by taking $u^-$ as a test function in $W^{1,n}(\mathbb{R}^n)$,  we obtain from $\langle{J_\zeta^\prime(u),u^-}\rangle=0$ and \eqref{g1} that
    $$\|F(\nabla u^-)\|^n_n+\zeta\|u^-\|^n_n=0. $$
    It follows that $\|u^-\|_\zeta=0$ and thus, we have $u^-=0$ a.e. in $\mathbb{R}^n$. Hence, we conclude that $u=u^+\geq 0$  a.e. in $\mathbb{R}^n$. This finishes the proof.
\end{proof}
In what follows, we prove that the Nehari manifold $\mathcal{N}_\zeta$ is not empty.
\begin{lemma}\label{lem8.2}
   For every $u\in W^{1,n}(\mathbb{R}^n)\setminus\{0\}$ satisfying $u\geq 0$ a.e. in $\mathbb{R}^n$, there exists a unique $t_u>0$ such that $t_u u\in\mathcal{N}_\zeta$. In particular, if such a function $u\in\mathcal{N}_\zeta$, then 
    $$ J_\zeta(u)=\max_{t\geq 0} J_\zeta(tu).$$
\end{lemma}
\begin{proof}
 Let $u\in W^{1,n}(\mathbb{R}^n)\setminus\{0\}$ be such that $u\geq 0$ a.e. in $\mathbb{R}^n$. Define the map $\Psi\colon [0,+\infty)\to\mathbb{R}$ by $\Psi(t)=J_\zeta(tu)$ for all $t\geq 0$. It follows that $\Psi(0)=0$. Take $\sigma,\sigma^\prime>1$ such that $\frac{1}{\sigma}+\frac{1}{\sigma^\prime}=1$. Next, we choose $t>0$ sufficiently small in such a way that $\lambda\sigma^\prime t^\frac{n}{n-1}\|u\|_\zeta^\frac{n}{n-1}<\lambda_n$. In view of \eqref{eq1.2}, we deduce from H\"{o}lder's inequality, Lemmas \ref{lem2.3}, \ref{lem2.4} and Corollary \ref{cor3.1} that
 \begin{align*}
 \int_{\mathbb{R}^n}G(tu)\,\mathrm{d}x&\leq \varepsilon t^n\|u\|_n^n+D_\varepsilon\left(\int_{\mathbb{R}^n}|tu|^{\mu\sigma}\,\mathrm{d}x\right)^{\frac{1}{\sigma}}\left(\int_{\mathbb{R}^n}\Phi\left(\lambda\sigma^\prime|tu|^{\frac{n}{n-1}}\right)\,\mathrm{d}x\right)^{\frac{1}{\sigma^\prime}}\\
 &\leq \varepsilon t^n S^{-n}_n\|u\|_\zeta^n+ D_\varepsilon t^\mu S^{-\mu}_{\mu\sigma}\|u\|_\zeta^\mu \left(\int_{\mathbb{R}^n}\Phi\left(\lambda\sigma^\prime t^\frac{n}{n-1}\|u\|_\zeta^\frac{n}{n-1}\big|u/\|u\|_\zeta\big|^{\frac{n}{n-1}}\right)\mathrm{d}x\right)^{\frac{1}{\sigma^\prime}}\\
 &\leq C\left(\varepsilon t^n \|u\|_\zeta^n+ t^\mu \|u\|_\zeta^\mu\right),
 \end{align*}
 where 
 \begin{align*}
     C:=\max\left\{S^{-n}_n,D_\varepsilon S^{-\mu}_{\mu\sigma}\left(\int_{\mathbb{R}^n}\Phi\left(\lambda\sigma^\prime t^\frac{n}{n-1}\|u\|_\zeta^\frac{n}{n-1}\big|u/\|u\|_\zeta\big|^{\frac{n}{n-1}}\right)\,\mathrm{d}x\right)^{\frac{1}{\sigma^\prime}}\right\}<+\infty.
 \end{align*}
 Consequently, we obtain from \eqref{eq8.1} that
 \begin{align*}
     \Psi(t)\geq \frac{t^n}{n}\|u\|_\zeta^n-C\left(\varepsilon t^n \|u\|_\zeta^n+ t^\mu \|u\|_\zeta^\mu\right).
 \end{align*}
Setting $\rho:=\|u\|_\zeta>0$ and choosing
$0<\varepsilon<\frac{1}{nC}$, one sees from the above inequality that
 \begin{align*}
     \Psi(t)\geq\left(\frac{1}{n}-\varepsilon C\right)t^n\rho^n-C t^\mu \rho^\mu>0
 \end{align*}
for sufficiently small $t>0$, thanks to $\mu>n$. Define another map $\Theta\colon [1,+\infty)\to\mathbb{R}$ by $\Theta(t):=t^{-\mu}G(tu)-G(u)$ for all $t\geq1$. By using \eqref{g3}, we have
\begin{align*}
 \Theta^\prime(t)&=-\mu t^{-\mu-1} G(tu)+t^{-\mu}g(tu)u=t^{-\mu-1}\left( g(tu)tu-\mu G(tu)\right)\geq 0.
\end{align*}
This shows that the map $t\mapsto\Theta(t)$ is monotone increasing on $[1,+\infty)$ and thus, we deduce that $G(tu)\geq t^\mu G(u)$ for all $t\geq1$. Hence, it follows from \eqref{eq8.1} that
\begin{align*}
    \Psi(t)\leq \frac{t^n}{n}\|u\|_\zeta^n-t^\mu\int_{\mathbb{R}^n} G(u)\,\mathrm{d}x<0
\end{align*}
for sufficiently large $t>0$, thanks to $\mu>n$. This infers that $\max_{t\geq0}\Psi(t)$ is achieved at some point $t=t_u>0$ such that $\Psi^\prime(t_u)=0$ and $t_u u\in\mathcal{N}_\zeta$. To finish the proof, we only need to check that $t_u$ is a unique critical point of $\Psi$. Indeed, if not, let there exist $0<t_1<t_2$ such that $\Psi^\prime(t_1)=\Psi^\prime(t_2)=0$. For $i=1,2$, we obtain that
\begin{align*}
0=\Psi'(t_i)
&=\frac{1}{t_i}\left(t_i^n\|u\|_\zeta^n
-\int_{\mathbb{R}^n}g(t_i u)t_i u\,\mathrm{d}x\right)\\
&=\frac{1}{t_i}\left(t_i^n\|u\|_\zeta^n
-\int_{\{u>0\}}g(t_i u)t_i u\,\mathrm{d}x\right).
\end{align*}
Dividing by $t_i^n$, we have
$$\|u\|_\zeta^n=\int_{\{u>0\}}
\frac{g(t_i u)}{(t_i u)^{n-1}}u^n\,\mathrm{d}x.$$
Moreover, it follows from \eqref{g5} that
\begin{align*}
    0=\int_{\{u>0\}}\left(\frac{g(t_1 u)}{(t_1 u)^{n-1}}-\frac{g(t_2 u)}{(t_2 u)^{n-1}}\right)u^n\,\mathrm{d}x<0,
\end{align*}
which is a contradiction. Hence, the uniqueness of $t_u$ is established, and the proof is complete. 
\end{proof}
To verify that the Nehari manifold is a natural constraint for $J_\zeta$, we establish the following result.
\begin{lemma}\label{lem8.333}
  If $u$ is a local minimizer of $J_\zeta$ on the Nehari manifold $\mathcal{N}_\zeta$, then $u$ is a critical point of  $J_\zeta$.
\end{lemma}
\begin{proof}
Let $u$ be a local minimizer of $J_\zeta$ on $\mathcal{N}_\zeta$. Now, we define 
\begin{align*}
I_\zeta(v):=\left\langle{J_\zeta^\prime(v),v}\right\rangle=\|v\|_\zeta^n-\int_{\mathbb{R}^n} g(v)v\,\mathrm{d}x\quad\text{for all}\quad v\in W^{1,n}(\mathbb{R}^n). 
\end{align*}
It follows from $u\in\mathcal{N}_\zeta$ that $I_\zeta(u)=0$. Due to the hypothesis \eqref{g6} and Theorem \ref{thm1.1}, we deduce that $I_\zeta$ is of class $C^1$ on  $W^{1,n}(\mathbb{R}^n)$. Moreover, by applying the Lagrange multiplier rule, there exists $\theta\in\mathbb{R}$ such that
\begin{align*}
     J_\zeta^\prime(u)=\theta I^\prime_\zeta(u).
\end{align*}
This shows that $\left\langle{J_\zeta^\prime(u),u}\right\rangle=\theta\left\langle{I^\prime_\zeta(u),u}\right\rangle$. Since $u\in \mathcal{N}_\zeta$, we have $\left\langle{J_\zeta^\prime(u),u}\right\rangle=0$ and thus,  $\theta\left\langle{I^\prime_\zeta(u),u}\right\rangle=0$.  By direct calculation, we obtain from $u\in \mathcal{N}_\zeta$ and \eqref{g5} that
\begin{align*}
 \left\langle{I^\prime_\zeta(u),u}\right\rangle &=n\|u\|_\zeta^n-\int_{\mathbb{R}^n}\big(g^\prime(u)u^2+ g(u)u\big)\,\mathrm{d}x\\
 &=(n-1)\int_{\mathbb{R}^n} g(u)u\,\mathrm{d}x-\int_{\mathbb{R}^n}g^\prime(u)u^2\,\mathrm{d}x\\
 &=-\int_{\{u>0\}} u^{n+1}\left(\frac{g(u)}{u^{n-1}}\right)^\prime\,\mathrm{d}x<0,  
\end{align*}
where we have used $g(s)=0$ and $g^\prime(s)=0$ for all $s\leq0$, which follows from \eqref{g1} and the fact that  
\begin{align*}
u\in \mathcal{N}_\zeta\implies\int_{\mathbb{R}^n} g(u)u\,\mathrm{d}x = \|u\|_\zeta^n >0. 
\end{align*}
This shows that $\{x\in\mathbb{R}^n: u>0\}$ has positive measure. It follows that $\left\langle{I^\prime_\zeta(u),u}\right\rangle<0$. In particular, we have $I^\prime_\zeta(u)\neq 0$. By the identity $\theta\left\langle{I^\prime_\zeta(u),u}\right\rangle=0$, we deduce that $\theta=0$ and thus, $J_\zeta^\prime(u)=0$. Hence, $u$ is a critical point of $J_\zeta$ and thereby completing the proof.
\end{proof}
Next, we derive an estimate for the least energy $m_\zeta$, which plays a key role in the study of existence and nonexistence of ground state solutions to \eqref{main problem}.
\begin{lemma}\label{lem8.3}
It holds that $$0<m_\zeta\leq\frac{1}{n}\left(\frac{\lambda_n}{\lambda_0}\right)^{n-1}\quad\text{for all}\quad \zeta>0.$$   
\end{lemma}
\begin{proof}
  It follows from \eqref{eq8.2} and \eqref{g3} that $m_\zeta \geq0$. We now claim that $m_\zeta>0$. Indeed, if not, let $m_\zeta=0$. Then there exists a sequence $\{u_k\}_k\subset\mathcal{N}_\zeta$ such that
  \begin{align}\label{eq8.3}
     \left\langle{J_\zeta^\prime(u_k),u_k}\right\rangle=0\qquad\text{and}\qquad J_\zeta(u_k)=o_k(1)\quad\text{as}\quad k\to\infty.
  \end{align}
  By using $\eqref{eq8.3}$, we obtain from $\eqref{eq8.2}$ and \eqref{g3} that
  \begin{align*}
      0&=\lim_{k\to\infty}\left(\frac{1}{n}\int_{\mathbb{R}^n}g(u_k)u_k\,\mathrm{d}x-\int_{\mathbb{R}^n}G(u_k)\,\mathrm{d}x\right)\\
      &\geq \left(\frac{1}{n}-\frac{1}{\mu}\right)\lim_{k\to\infty}\int_{\mathbb{R}^n}g(u_k)u_k\,\mathrm{d}x\\
      &=\left(\frac{1}{n}-\frac{1}{\mu}\right)\lim_{k\to\infty}\left(\int_{\mathbb{R}^n}F(\nabla u_k)^n\,\mathrm{d}x+\zeta\int_{\mathbb{R}^n}|u_k|^n\,\mathrm{d}x\right)\\
      &\geq 0.
  \end{align*}
  This yields that
\begin{align}\label{eq8.4}
      \int_{\mathbb{R}^n}F(\nabla u_k)^n\,\mathrm{d}x+\zeta\int_{\mathbb{R}^n}|u_k|^n\,\mathrm{d}x=o_k(1)\quad\text{as}\quad k\to\infty.
\end{align}
Let $\sigma,\,\sigma^\prime>1$ be such that $\frac{1}{\sigma}+\frac{1}{\sigma^\prime}=1$. In view of \eqref{eq8.4}, there exists $k_0\in\mathbb{N}$ large enough such that we can assume $\lambda \sigma^\prime \|u_k\|_\zeta^\frac{n}{n-1}<\lambda_n$ for all $k\geq k_0$. Now, by employing \eqref{eq1.1}, \eqref{eq8.3}, H\"{o}lder's inequality, Lemmas \ref{lem2.3}, \ref{lem2.4} and Corollary \ref{cor3.1}, we obtain for all $k\geq k_0$ that
\begin{equation}\label{eq8.5}
\begin{aligned}
 1&=\frac{1}{\|u_k\|_\zeta^n}\int_{\mathbb{R}^n}g(u_k)u_k\,\mathrm{d}x\\
 &\leq \frac{1}{\|u_k\|_\zeta^n}\left(\varepsilon \|u_k\|_n^n+D_\varepsilon\left(\int_{\mathbb{R}^n}|u_k|^{\mu\sigma}\,\mathrm{d}x\right)^{\frac{1}{\sigma}}\left(\int_{\mathbb{R}^n}\Phi\left(\lambda\sigma^\prime|u_k|^{\frac{n}{n-1}}\right)\,\mathrm{d}x\right)^{\frac{1}{\sigma^\prime}}\right)\\
 &\leq \varepsilon S^{-n}_n+ D_\varepsilon S^{-\mu}_{\mu\sigma}\|u_k\|_\zeta^{\mu-n} \left(\int_{\mathbb{R}^n}\Phi\left(\lambda\sigma^\prime \|u_k\|_\zeta^\frac{n}{n-1}\big|u_k/\|u_k\|_\zeta\big|^{\frac{n}{n-1}}\right)\,\mathrm{d}x\right)^{\frac{1}{\sigma^\prime}}\\
 &\leq \varepsilon S^{-n}_n+ D_\varepsilon S^{-\mu}_{\mu\sigma} C\|u_k\|_\zeta^{\mu-n},
 \end{aligned}
 \end{equation}
 where
\begin{align*}
     C:=\sup_{k\geq k_0}\left(\int_{\mathbb{R}^n}\Phi\left(\lambda\sigma^\prime \|u_k\|_\zeta^\frac{n}{n-1}\big|u_k/\|u_k\|_\zeta\big|^{\frac{n}{n-1}}\right)\,\mathrm{d}x\right)^{\frac{1}{\sigma^\prime}}<+\infty.
 \end{align*}
 Since $\|u_k\|_\zeta=o_k(1)$ as $k\to\infty$ and $\varepsilon>0$ is arbitrary, therefore we deduce that \eqref{eq8.5} can not hold. Hence, we arrive at a contradiction. It follows that we must have $m_\zeta>0$ and thus, the claim is verified. 
 
 To finish the proof, we only need to show that $m_\zeta\leq\frac{1}{n}\big(\frac{\lambda_n}{\lambda_0}\big)^{n-1}$. Let $u\in W^{1,n}(\mathbb{R}^n)\setminus\{0\}$ be such that $\|u\|_\zeta\leq 1$. It follows from Lemma \ref{eq8.2} that there exists a unique $t_u>0$ such that $t_u u\in \mathcal{N}_\zeta$. In addition, due to Lemma \ref{lem8.1}, one sees that $t_u u$ is nonnegative. Consequently, we obtain by using $u\neq 0$ that
 \begin{align*}
     m_\zeta&\leq J_\zeta(t_u u)=\frac{t^n_u}{n}\|u\|_\zeta^n-\int_{\mathbb{R}^n}G(t_u u)\,\mathrm{d}x<\frac{t^n_u}{n}\|u\|_\zeta^n\leq \frac{t^n_u}{n}.
 \end{align*}
 By taking $m_\zeta=\frac{t^n_\zeta}{n}$ and using \eqref{g5}, we infer that
 \begin{align*}
     \int_{\mathbb{R}^n}\frac{g(t_\zeta u)}{t_\zeta^{n-1}}\,u\,\mathrm{d}x\leq \int_{\mathbb{R}^n}\frac{g(t_u u)}{t_u^{n-1}}\,u\,\mathrm{d}x=\|u\|^n_\zeta\leq 1,
 \end{align*}
 where the fact $t_u u\in \mathcal{N}_\zeta$ has been utilized. This shows that
 $$\sup_{u\in W^{1,n}(\mathbb{R}^n),\,\|u\|_\zeta\leq 1} \int_{\mathbb{R}^n}\frac{g(t_\zeta u)}{t_\zeta^{n-1}}\,u\,\mathrm{d}x<+\infty.$$
 The above inequality, combined with \eqref{eq1.1}, H\"{o}lder's inequality, Lemmas \ref{lem2.3}, \ref{lem2.4} and Corollary \ref{cor3.1}, implies that
 $$m_\zeta=\frac{t^n_\zeta}{n}<\frac{1}{n}\left(\frac{\lambda_n}{\lambda}\right)^{n-1}\quad\text{for all}\quad \lambda>\lambda_0.$$
 Letting $\lambda\to\lambda_0^+$, we have
  $$m_\zeta\leq\frac{1}{n}\left(\frac{\lambda_n}{\lambda_0}\right)^{n-1}\quad\text{for all}\quad \zeta>0.$$
  Indeed, if $m_\zeta>\frac{1}{n}\left(\frac{\lambda_n}{\lambda_0}\right)^{n-1}$, then we have $t_\zeta^{\frac{n}{n-1}}\lambda_0>\lambda_n$. Thus, we can find some $\delta_0>0$ such that $t_\zeta^{\frac{n}{n-1}}\lambda_0>\lambda_n+\delta_0$, thanks to the density argument of real numbers. Now, by employing \eqref{g2}, one can notice that
  \begin{align*}
      \frac{g(t_\zeta s)}{\operatorname{\exp}\Big(\frac{\lambda_n+\delta_0}{t_\zeta^{\frac{n}{n-1}}}(t_\zeta s)^{\frac{n}{n-1}}\Big)}=\frac{g(t_\zeta s)}{\operatorname{\exp}\Big((\lambda_n+\delta_0) s^{\frac{n}{n-1}}\Big)}\to+\infty\quad\text{as}\quad s\to\infty.
  \end{align*}
  It follows that
  $$\lim_{s\to\infty}\frac{g(t_\zeta s)s}{\operatorname{\exp}\Big((\lambda_n+\delta_0) s^{\frac{n}{n-1}}\Big)}=+\infty. $$
Recall the Moser-type sequence of functions $\{v_k\}_k$ that introduced in \eqref{eq3.4}. Observe that $\|v_k\|_\zeta=1+o_k(1)$ as $k\to\infty$. Hence, by employing the above limit, we infer that
 \begin{align*}
  \sup_{u\in W^{1,n}(\mathbb{R}^n),\,\|u\|_\zeta\leq 1} &\int_{\mathbb{R}^n}\frac{g(t_\zeta u)}{t_\zeta^{n-1}}\,u\,\mathrm{d}x\\
  &\geq\int_{\{x\in \mathbb{R}^n\colon\, F^0(x)\leq sk^{-1}\}} \frac{g\big(t_\zeta \big(v_k/\|v_k\|_\zeta\big)\big)}{t_\zeta^{n-1}}\,\big(v_k/\|v_k\|_\zeta\big)\,\mathrm{d}x\\
  &\gtrsim \int_{\{x\in \mathbb{R}^n\colon\, F^0(x)\leq sk^{-1}\}} \operatorname{\exp}\big((\lambda_n+\delta_0) \big(v_k/\|v_k\|_\zeta\big)^{\frac{n}{n-1}}\big)\,\mathrm{d}x\\
  &\gtrsim \int_{0}^{sk^{-1}}\operatorname{\exp}\Bigg(\frac{(\lambda_n+\delta_0)\big((n\kappa_n)^{-\frac{1}{n}}(\ln k)^{\frac{n-1}{n}}\big)^{\frac{n}{n-1}}}{\|v_k\|_\zeta^{\frac{n}{n-1}}}\Bigg)\eta^{n-1}\,\mathrm{d}\eta\\
  &\gtrsim\operatorname{\exp}\left(n\ln k\Bigg(\frac{1+\frac{\delta_0}{n}(n\kappa_n)^{-\frac{1}{n-1}}}{\|v_k\|_\zeta^{\frac{n}{n-1}}}-1\Bigg)\right)\to+\infty
 \end{align*}
as $k\to\infty$, which is a contradiction. Therefore, we must have $m_\zeta\leq\frac{1}{n}\left(\frac{\lambda_n}{\lambda_0}\right)^{n-1}$ for all $\zeta>0$. Thus, the desired conclusion follows. 
\end{proof}
We now define the anisotropic Moser--Trudinger ratio by
$$C_{MT}^L(G):=\sup\left\{\frac{n}{\|u\|_n^n}\int_{\mathbb{R}^n}G(u)\,\mathrm{d}x\colon\, u\in W^{1,n}(\mathbb{R}^n)\setminus\{0\}\quad\text{and}\quad \|F(\nabla u)\|_n^n\leq L\right\}.$$
The associated anisotropic Moser--Trudinger threshold is given by
$$R(G):=\sup\left\{L>0\colon\, C_{MT}^L(G)<+\infty\right\}.$$
Denote
$$\zeta^\ast=C_{MT}^{R(G)}(G),$$ which will be referred to as the ratio at threshold. The critical exponential growth condition on $g$, combined with Theorem \ref{thm1.3}, yields $$R(G)=\left(\frac{\lambda_n}{\lambda_0}\right)^{n-1}.$$ 
The isotropic analogue of this ratio was first studied by Ibrahim--Masmoudi--Nakanishi \cite{Ibrahim-Masmoudi-Nakanishi-2015} in dimension two, and later by Masmoudi--Sani \cite{Masmoudi-Sani-2015} in arbitrary higher dimensions. In addition, the precise value of $L$ for isotropic case was rigorously derived by Chen--Qin--R\u{a}dulescu--Tang \cite[Lemma 2.6]{Chen-Qin-Radulescu-Tang-2025}.

The following result, which will be frequently used in our analysis, is stated below.
\begin{lemma}\label{lem8.4}
Let $\zeta\in(0,\zeta^\ast)$. Then there holds 
\begin{align}\label{eq8.6}
 0<m_\zeta<\frac{1}{n}\left(\frac{\lambda_n}{\lambda_0}\right)^{n-1}.   
\end{align} 
\end{lemma}
\begin{proof}
To begin the proof, we first define the Poho\v{z}aev manifold $\mathcal{P}_\zeta$ associated with \eqref{main problem} by
\begin{align*}
  \mathcal{P}_\zeta:=\left\{u\in W^{1,n}(\mathbb{R}^n)\setminus\{0\}\colon\,\zeta\int_{\mathbb{R}^n}|u|^n\,\mathrm{d}x=n\int_{\mathbb{R}^n}G(u)\,\mathrm{d}x\right\}.
\end{align*}
Denote $ M_\zeta:=\inf\left\{J_\zeta(u)\colon\, u\in \mathcal{P}_\zeta\right\}$. It follows that for each $u\in \mathcal{P}_\zeta$, there holds $J_\zeta(u)=\frac{1}{n}\|F(\nabla u)\|_n^n$. The remainder of the proof is now split into two separate cases as follows.

\textbf{Case-1:} In this case, we claim that $m_\zeta\leq M_\zeta$. Let $\{u_k\}_k\subset W^{1,n}(\mathbb{R}^n)\setminus\{0\}$ be a minimizing sequence for $M_\zeta$. It follows that $\{u_k\}_k\subset \mathcal{P}_\zeta$ and there holds
\begin{align*}
 J_\zeta(u_k)=M_\zeta+o_k(1)\quad\text{as}\quad k\to\infty.
\end{align*}
Set 
\begin{align*}
    \theta_k:=\left(\frac{\int_{\mathbb{R}^n}\left(g(u_k)u_k-nG(u_k)\right)\,\mathrm{d}x}{\|F(\nabla u)\|_n^n}\right)^{\frac{1}{n}}\quad\text{for all}\quad k\in\mathbb{N}.
\end{align*}
In view of \eqref{g3}, we also have $\theta_k>0$ for each $k\in\mathbb{N}$. The above relation can be simplified further to obtain, via a change of variables, that
$$\int_{\mathbb{R}^n}F(\nabla u_k(\theta_k x))^{n}\,\mathrm{d}x+\zeta\int_{\mathbb{R}^n}|u_k(\theta_k x)|^{n}\,\mathrm{d}x=\int_{\mathbb{R}^n}g(u_k(\theta_k x))u_k(\theta_k x)\,\mathrm{d}x.$$
It follows that $\{u_k(\theta_k x)\}_k\subset \mathcal{P}_\zeta\cap\mathcal{N}_\zeta$. By gathering these information, we have
\begin{align*}
    m_\zeta &\leq \lim_{k\to\infty}J_\zeta(u_k(\theta_k x))=\lim_{k\to\infty}\frac{1}{n}\|F(\nabla u_k(\theta_k x))\|_n^n\\
    &=\lim_{k\to\infty}\frac{1}{n}\|F(\nabla u_k)\|_n^n=\lim_{k\to\infty}J_\zeta(u_k)=M_\zeta.
\end{align*}
This shows that $m_\zeta\leq M_\zeta$, which verifies the claim.

\textbf{Case-2:} In this case, we show that if $\zeta<\zeta^\ast$, then $M_\zeta<\frac{1}{n}\big(\frac{\lambda_n}{\lambda_0}\big)^{n-1}$. The subsequent discussion is separated into two possibilities: $\zeta^\ast<+\infty$ and $\zeta^\ast=+\infty$. The case $\zeta^\ast<+\infty$ is considered first. By using $\zeta<\zeta^\ast$, there exists some $\delta_0$ such that $\zeta<\zeta^\ast-\delta_0$, thanks to the density argument of real numbers. Hence, we obtain from the definition of $\zeta^\ast$ that there exists some $w\in W^{1,n}(\mathbb{R}^n)\setminus\{0\}$ satisfying $\|F(\nabla w)\|_n^n\leq R(G)$ such that
\begin{align}\label{eq8.7}
 \zeta<\zeta^\ast-\delta_0  <\frac{n}{\|w\|_n^n}\int_{\mathbb{R}^n}G(w)\,\mathrm{d}x. 
\end{align}
Define the map $\Pi\colon(0,+\infty)\to\mathbb{R}$ by
$$ \Pi(t):=\zeta\int_{\mathbb{R}^n}|tw|^n\,\mathrm{d}x-n\int_{\mathbb{R}^n}G(tw)\,\mathrm{d}x\quad\text{for all}\quad t>0. $$
In view of \eqref{eq8.7}, one has $\Pi(1)<0$. In addition, by similar as in Lemma \ref{lem8.2} with minor modifications, we can easily deduce that $\Pi(t)>0$ for sufficiently small $t>0$. Thus, the intermediate value theorem guarantees the existence of $t_0\in(0,1)$ for which $\Pi(t_0)=0$. It follows that $t_0 w\in \mathcal{P}_\zeta$. Hence, we obtain
\begin{align}\label{eq8.8}
    M_\zeta\leq J_\zeta(t_0 w)=\frac{1}{n}\|F(\nabla t_0 w)\|_n^n=\frac{t^n}{n}\|F(\nabla w)\|_n^n<\frac{R(G)}{n}=\frac{1}{n}\left(\frac{\lambda_n}{\lambda_0}\right)^{n-1}.
\end{align}
If $\zeta^\ast=+\infty$, then for any $\zeta>0$, there exists some $w\in W^{1,n}(\mathbb{R}^n)\setminus\{0\}$ satisfying $\|F(\nabla w)\|_n^n\leq R(G)$ such that
\begin{align*}
 \zeta<\frac{n}{\|w\|_n^n}\int_{\mathbb{R}^n}G(w)\,\mathrm{d}x. 
\end{align*}
Now, the proof proceeds in a completely analogous fashion to that of the case $\zeta^\ast<+\infty$, we omit the repetitive details and simply note that \eqref{eq8.8} remains valid.

 By the preceding discussion, we deduce that the desired inequality \eqref{eq8.6} follows. This completes the proof.
\end{proof}
\begin{lemma}\label{lem8.5}
 Let $\{u_k\}_k\subset \mathcal{N}_\zeta$ be a minimizing sequence for $m_\zeta$. Then the sequence $\{u_k\}_k$ is bounded in $W^{1,n}(\mathbb{R}^n)$. Furthermore, there hold
 \begin{align}\label{eq8.9}
\sup_{k\in\mathbb{N}}\int_{\mathbb{R}^n}g(u_k)u_k\,\mathrm{d}x<+\infty\quad\text{and}\quad \sup_{k\in\mathbb{N}}\int_{\mathbb{R}^n}G(u_k)\,\mathrm{d}x<+\infty.
 \end{align}
\end{lemma}
\begin{proof}
It follows that from the given hypothesis that
$$ \left\langle{J_\zeta^\prime(u_k),u_k}\right\rangle=0\quad\text{and}\quad J_\zeta(u_k)=m_\zeta+o_k(1)\quad\text{as}\quad k\to\infty. $$
In virtue of \eqref{g3} and the above estimates, we obtain 
\begin{align*}
 m_\zeta+o_k(1)=J_\zeta(u_k)-\frac{1}{\mu} \left\langle{J_\zeta^\prime(u_k),u_k}\right\rangle\geq \left(\frac{1}{n}-\frac{1}{\mu}\right)\|u_k\|_\zeta^n
\end{align*}
as $k\to\infty$. Hence, the above inequality ensures that $\{u_k\}_k$ is bounded in $W^{1,n}(\mathbb{R}^n)$. Now, by using \eqref{g3} and the fact that $\{u_k\}_k \subset \mathcal{N}_\zeta$, we obtain
\begin{align*}
    \mu \int_{\mathbb{R}^n} G(u_k)\,\mathrm{d}x
    \leq \int_{\mathbb{R}^n} g(u_k)u_k\,\mathrm{d}x
    = \|u_k\|_\zeta^n.
\end{align*}
Together with the boundedness of $\{u_k\}_k$ in $W^{1,n}(\mathbb{R}^n)$, the above inequality yields \eqref{eq8.9}. The proof is thereby concluded.
\end{proof}
Note that by using Lemmas \ref{lem8.1} and \ref{lem8.5}, without loss of generality of generality, we can assume that the minimizing sequence $\{u_k\}_k\subset \mathcal{N}_\zeta$ for $m_\zeta$ is nonnegative and up to a subsequence still denoted by itself, there exist $u\in W^{1,n}(\mathbb{R}^n)$ and $p\geq 1$ such that 
\begin{align}\label{eq8.10}
u_k\rightharpoonup u~~\text{in}~~ W^{1,n}(\mathbb{R}^n),\,
u_k\to u ~~\text{in}~~ L^p _{\textit{loc}}(\mathbb{R}^n),~~\text{and}~~
u_k\to u~~\text{a.e. in}~~\mathbb{R}^n 
\end{align}
as $k\to\infty$. It follows that the weak limit $u$ is also nonnegative.
\begin{lemma}\label{lem8.6}
 Let $\{u_k\}_k\subset \mathcal{N}_\zeta$ be a minimizing sequence for $m_\zeta$ such that $u_k\rightharpoonup u$ in $W^{1,n}(\mathbb{R}^n)$ as $k\to\infty$. Then there holds
 \begin{align*}
\lim_{k\to\infty}\int_{\mathbb{R}^n}G(u_k)\,\mathrm{d}x=\int_{\mathbb{R}^n}G(u)\,\mathrm{d}x.
 \end{align*}
\end{lemma}
\begin{proof}
By Lemma \ref{lem8.5}, the sequence $\{u_k\}_k$ is bounded in $W^{1,n}(\mathbb{R}^n)$, and there holds
\begin{align*}
 \int_{\mathbb{R}^n} g(u_k)u_k\,\mathrm{d}x\leq C,   
\end{align*}
where $C>0$ is a constant independent of $k$. Taking \eqref{eq8.10} into account, it follows that $u_k\to u$ in $L^1 _{\textit{loc}}(\mathbb{R}^n)$ as $k\to\infty$. We also note that both $g(u)$ and $g(u_k)$ belong to $ L^1 _{\textit{loc}}(\mathbb{R}^n)$. Consequently, the result established by Figueiredo--Miyagaki--Ruf \cite[Lemma 2.1]{Figueiredo-Miyagaki-Ruf-1995} ensures that
\begin{align}\label{eq8.11}
    g(u_k)\to g(u)\quad\text{in}\quad L^1 _{\textit{loc}}(\mathbb{R}^n)\quad\text{as}\quad k\to\infty.
\end{align}
Next, our aim is to show that
\begin{align}\label{eq8.12}
    G(u_k)\to G(u)\quad\text{in}\quad L^1 (\mathbb{R}^n)\quad\text{as}\quad k\to\infty.
\end{align}
To proceed, we take $R>0$ and write
\begin{align*}
\int_{\mathbb{R}^n}G(u_k)\,\mathrm{d}x=\int_{\mathcal{W}_R}G(u_k)\,\mathrm{d}x+\int_{\mathbb{R}^n\setminus \mathcal{W}_R}G(u_k)\,\mathrm{d}x.
\end{align*}
Due to \eqref{g3} and \eqref{g4}, there exists a constant $K_0>0$ such that $G(u_k)\leq K_0 g(u_k)$. Thus, we deduce from \eqref{eq8.10}, \eqref{eq8.11} and the generalized Lebesgue dominated convergence theorem that
\begin{align}\label{eq8.13}
\lim_{k\to\infty}\int_{\mathcal{W}_R}G(u_k)\,\mathrm{d}x=\int_{\mathcal{W}_R}G(u)\,\mathrm{d}x\quad\text{for all} \quad R>0.  
\end{align}
Next, we claim that for any $0<\varepsilon<1$, there exists $R>1$ such that
\begin{align}\label{eq8.14}
    \left|\int_{\mathbb{R}^n\setminus \mathcal{W}_R}G(u_k)\,\mathrm{d}x-\int_{\mathbb{R}^n\setminus \mathcal{W}_R}G(u)\,\mathrm{d}x\right|\leq C\varepsilon\quad\text{for all}\quad k\in\mathbb{N},
\end{align}
where $C>0$ is a constant independent of $\varepsilon$ and $R$. For this purpose, we choose $0<\varepsilon<1$ and $R>1$. It then follows from \eqref{eq1.2} that
\begin{align*}
    \int_{\mathbb{R}^n\setminus \mathcal{W}_R}G(u_k)\,\mathrm{d}x\leq \varepsilon\|u_k\|_n^n+D_\varepsilon\int_{\mathbb{R}^n\setminus \mathcal{W}_R}  |u_k|^\mu\Phi\big(\lambda|u_k|^{\frac{n}{n-1}}\big)\,\mathrm{d}x\quad\text{for all}\quad k\in\mathbb{N}.
\end{align*}
For all $R>1$, $\mu>n$ and $j\geq n-1$, one can notice that
\begin{align*}
\int_{\mathbb{R}^n\setminus \mathcal{W}_R}& F^0(x)^{-\big(\mu+\frac{nj}{n-1}\big)}\,\mathrm{d}x
\\&=\int_{R}^{\infty} \eta^{-\big(\mu+\frac{nj}{n-1}\big)} \int_{\partial\mathcal{W}_\eta}\frac{1}{|\nabla F^0(x)|} \,\mathrm{d}\sigma\,\mathrm{d}\eta=n\kappa_n \int_{R}^{\infty} \eta^{n-\big(\mu+\frac{nj}{n-1}\big)-1}\,\mathrm{d}\eta\\
&=\bigg(\frac{n\kappa_n}{\mu+\frac{nj}{n-1}-n}\bigg)\frac{1}{R^{\mu+\frac{nj}{n-1}-n}}\leq n\kappa_n R^{-1}.
\end{align*}
Denote by $u_k^\ast$ the convex symmetrization of $u_k$. It follows from the above estimate, along with the monotone convergence theorem, Lemmas \ref{lem2.91}, and \ref{lem2.2} that
\begin{align*}
  \int_{\mathbb{R}^n\setminus \mathcal{W}_R} & |u_k|^\mu\Phi\big(\lambda|u_k|^{\frac{n}{n-1}}\big)\,\mathrm{d}x
  \\&=\sum_{j=n-1}^{\infty}\frac{\lambda^j}{j!} \int_{\mathbb{R}^n\setminus \mathcal{W}_R} |u_k|^{\mu+\frac{nj}{n-1}}\,\mathrm{d}x=\sum_{j=n-1}^{\infty}\frac{\lambda^j}{j!} \int_{\mathbb{R}^n\setminus \mathcal{W}_R} |u_k^\ast|^{\mu+\frac{nj}{n-1}}\,\mathrm{d}x\\
  &\leq\bigg(\frac{\|u_k^\ast\|^\mu}{\kappa_n^{\frac{\mu}{n}}}\bigg)\sum_{j=n-1}^{\infty}\frac{\lambda^j}{j!}\left(\frac{1}{\kappa_n}\right)^{\frac{j}{n-1}}\|u_k^\ast\|^{\frac{nj}{n-1}}\int_{\mathbb{R}^n\setminus \mathcal{W}_R} F^0(x)^{-\big(\mu+\frac{nj}{n-1}\big)}\,\mathrm{d}x\\
  &\leq \bigg(\frac{n\|u_k^\ast\|^\mu}{\kappa_n^{\frac{\mu}{n}-1}}\bigg)\left(\sum_{j=n-1}^{\infty}\frac{\lambda^j}{j!}\left(\frac{1}{\kappa_n}\right)^{\frac{j}{n-1}}\|u_k^\ast\|^{\frac{nj}{n-1}}\right)\frac{1}{R}\\
  &\leq \bigg(\frac{n\|u_k\|^\mu}{\kappa_n^{\frac{\mu}{n}-1}}\bigg)\operatorname{\exp}\left(\frac{\lambda\|u_k\|^{\frac{n}{n-1}}}{\kappa_n^{\frac{1}{n-1}}}\right)\frac{1}{R}.
\end{align*}
In light of the above discussion, we deduce from Lemmas \ref{lem2.3} and \ref{lem8.6} that
\begin{align*}
    \int_{\mathbb{R}^n\setminus \mathcal{W}_R}G(u_k)\,\mathrm{d}x\leq \tilde{C}\varepsilon+\frac{\tilde{D}(\varepsilon,\lambda,\mu,n)}{R} \quad\text{for all}\quad k\in\mathbb{N},
\end{align*}
where $\tilde{C}>0$ is a constant independent of $\varepsilon$ and $R$, and $\tilde{D}(\varepsilon,\lambda,\mu,n)>0$ is a constant depending only on $\varepsilon$, $\lambda$, $\mu$ and $n$. Therefore, without loss of generality, we may assume that $\tilde{D}(\varepsilon,\lambda,\mu,n)>1$. In addition, by setting 
$$ R:=\frac{\tilde{D}(\varepsilon,\lambda,\mu,n)}{\varepsilon} >1,$$
we obtain 
\begin{align*}
    \int_{\mathbb{R}^n\setminus \mathcal{W}_R}G(u_k)\,\mathrm{d}x\leq (\tilde{C}+1)\varepsilon \quad\text{for all}\quad k\in\mathbb{N}.
\end{align*}
Similarly, we can prove that
\begin{align*}
    \int_{\mathbb{R}^n\setminus \mathcal{W}_R}G(u)\,\mathrm{d}x\leq (\tilde{C}+1)\varepsilon.
\end{align*}
This immediately yields \eqref{eq8.14} and hence, the claim follows. It follows from \eqref{eq8.14} that
\begin{align*}
    \limsup_{k\to\infty}\left|\int_{\mathbb{R}^n\setminus \mathcal{W}_R}G(u_k)\,\mathrm{d}x-\int_{\mathbb{R}^n\setminus \mathcal{W}_R}G(u)\,\mathrm{d}x\right|\leq C\varepsilon.
\end{align*}
Letting $\varepsilon\to 0^+$, we infer from the above estimate that
\begin{align}\label{eq8.15}
  \lim_{k\to\infty}\int_{\mathbb{R}^n\setminus \mathcal{W}_R}G(u_k)\,\mathrm{d}x=\int_{\mathbb{R}^n\setminus \mathcal{W}_R}G(u)\,\mathrm{d}x.  
\end{align}
The validity of \eqref{eq8.12} follows directly from \eqref{eq8.13} and \eqref{eq8.15}. The proof is complete.
\end{proof}
\begin{lemma}\label{lem8.7}
   Let $\zeta\in(0,\zeta^\ast)$ and $\{u_k\}_k\subset \mathcal{N}_\zeta$ be a minimizing sequence for $m_\zeta$ such that $u_k\rightharpoonup u\neq 0$ in $W^{1,n}(\mathbb{R}^n)$ as $k\to\infty$ with $\|u\|^n_\zeta>\int_{\mathbb{R}^n}g(u)u\,\mathrm{d}x$. Then there holds
 \begin{align}\label{eq8.16}
\lim_{k\to\infty}\int_{\mathbb{R}^n}g(u_k)u_k\,\mathrm{d}x=\int_{\mathbb{R}^n}g(u)u\,\mathrm{d}x.
 \end{align}  
\end{lemma}
\begin{proof}
By the given hypothesis, we infer that \eqref{eq8.10} holds. We also deduce from Lemma \ref{lem8.4} that
\begin{align}\label{eq8.17}
  \left\langle{J_\zeta^\prime(u_k),u_k}\right\rangle=0\quad\text{and}\quad \lim_{k\to\infty}J_\zeta(u_k)=m_\zeta<\frac{1}{n}\left(\frac{\lambda_n}{\lambda_0}\right)^{n-1}.   
\end{align}
Notice that the weak lower semicontinuity of the norm in $W^{1,n}(\mathbb{R}^n)$ ensures, after passing to a subsequence still denoted by itself, we have
$ \|u\|_\zeta\leq \lim_{k\to\infty} \|u_k\|_\zeta.$
Therefore, the proof proceeds by considering the following two cases.

\textbf{Case-1:} In this case, we consider $\lim_{k\to\infty} \|u_k\|_\zeta=\|u\|_\zeta$. Since $u_k\rightharpoonup u$ in $W^{1,n}(\mathbb{R}^n)$ as $k\to\infty$, we conclude that $u_k\to u$ in $W^{1,n}(\mathbb{R}^n)$ as $k\to\infty$. By employing the result of do \'{O}--Medeiros--Severo \cite[Proposition 1]{do O-Medeiros-Severo-2009}, there exists $w\in W^{1,n}(\mathbb{R}^n)$ such that, up to a subsequence still denoted by itself, we have $|u_k|\leq w$ a.e. in $\mathbb{R}^n$. It follows from $\eqref{eq8.10}$ that $g(u_k)u_k\to g(u)u$ a.e. in $\mathbb{R}^n$ as $k\to\infty$. Now, by using H\"{o}lder's inequality, Lemmas \ref{lem2.3}, \ref{lem2.4} and Corollary \ref{cor3.2}, we obtain from \eqref{eq1.1} that
\begin{align*}
    |g(u_k)u_k|&\leq \varepsilon |u_k|^{n}+D_\varepsilon|u_k|^{\mu}\Phi(\lambda|u_k|^{\frac{n}{n-1}})\\
    &\leq \varepsilon |w|^{n}+D_\varepsilon|w|^{\mu}\Phi(\lambda|w|^{\frac{n}{n-1}})\in L^1(\mathbb{R}^n).
\end{align*}
Hence, the desired result \eqref{eq8.16} follows immediately from the Lebesgue dominated convergence theorem. 

\textbf{Case-2:} In this case, we consider $\lim_{k\to\infty} \|u_k\|_\zeta>\|u\|_\zeta$. Next, we define 
\begin{align*}
    w_k:=\frac{u_k}{\|u_k\|_\zeta}\qquad\text{and}\qquad  w:=\frac{u}{\lim_{k\to\infty}\|u_k\|_\zeta}.
\end{align*}
Notice that $\|w_k\|_\zeta=1$, $w\neq 0$, and $w_k\rightharpoonup w$ in $W^{1,n}(\mathbb{R}^n)$ as $k\to\infty$. Observe that if $\|w\|_\zeta=1$, then $u_k\to u$ in $W^{1,n}(\mathbb{R}^n)$ as $k\to\infty$. This contradicts $\lim_{k\to\infty} \|u_k\|_\zeta>\|u\|_\zeta$. Thus, we only need to treat the case where $\|w\|_\zeta<1$. By the hypothesis and \eqref{g3}, we obtain from \eqref{eq8.1} that 
\begin{equation}\label{eq8.18}
\begin{aligned}
    J_\zeta(u)&=\frac{1}{n}\|u\|^n_\zeta-\int_{\mathbb{R}^n}G(u)\,\mathrm{d}x\geq \frac{1}{n}\|u\|^n_\zeta-\frac{1}{\mu}\int_{\mathbb{R}^n}g(u)u\,\mathrm{d}x\\
    &\geq \frac{1}{n}\left(\|u\|^n_\zeta-\int_{\mathbb{R}^n}g(u)u\,\mathrm{d}x\right)>0.
\end{aligned}
\end{equation}
In virtue of \eqref{eq8.17}, \eqref{eq8.18} and Lemma \ref{lem8.6}, one has
\begin{align*}
    \limsup_{k\to\infty}\|u_k\|_\zeta^n\left(1-\|w\|_\zeta^n\right)&=\limsup_{k\to\infty}\|u_k\|_\zeta^n\left(1-\frac{\|u\|_\zeta^n}{\lim_{k\to\infty}\|u_k\|_\zeta^n}\right)\\&\leq \limsup_{k\to\infty}\|u_k\|_\zeta^n-\|u\|^n_\zeta\\
    &=\left(nm_\zeta+\int_{\mathbb{R}^n}G(u)\,\mathrm{d}x\right)-\left(nJ_\zeta(u)+\int_{\mathbb{R}^n}G(u)\,\mathrm{d}x\right)\\
    &<nm_\zeta<\left(\frac{\lambda_n}{\lambda_0}\right)^{n-1},
\end{align*}
where we have used $\lim_{k\to\infty}\|u_k\|_\zeta^n\leq\limsup_{k\to\infty}\|u_k\|_\zeta^n$. It follows at once that
$$ \limsup_{k\to\infty}\|u_k\|_\zeta^{\frac{n}{n-1}}<\left(\frac{\lambda_n}{\lambda_0}\right)\left(1-\|w\|_\zeta^n\right)^{-\frac{1}{n-1}}.$$
Hence, passing to a subsequence if necessary, still denoted by itself, we may assume that
$$ \sup_{k\in\mathbb{N}}\|u_k\|_\zeta^{\frac{n}{n-1}}<\left(\frac{\lambda_n}{\lambda_0}\right)\left(1-\|w\|_\zeta^n\right)^{-\frac{1}{n-1}}.$$
We now show that
\begin{align}\label{eq8.19}
\sup_{k\in\mathbb{N}}\int_{\mathbb{R}^n}\left(g(u_k)u_k\right)^q\,\mathrm{d}x<+\infty,   
\end{align}
where $q>1$ is assumed to be sufficiently close to $1$ so that 
$$ q\sup_{k\in\mathbb{N}}\|u_k\|_\zeta^{\frac{n}{n-1}}<\left(\frac{\lambda_n}{\lambda_0}\right)\left(1-\|w\|_\zeta^n\right)^{-\frac{1}{n-1}}.$$
Pick $\sigma,\,\sigma^\prime>1$ such that $\frac{1}{\sigma}+\frac{1}{\sigma^\prime}=1$. In view of the above inequality, one can choose
$$
m\in \Big(q\|u_k\|_\zeta^{\frac{n}{n-1}},
\Big(\frac{\lambda_n}{\lambda_0}\Big)
\big(1-\|w\|_\zeta^n\big)^{-\frac{1}{n-1}}\Big),
$$
$\lambda>\lambda_0$ sufficiently close to $\lambda_0$, and $\sigma>1$ sufficiently close to $1$ such that $\lambda\sigma m<
\lambda_n\big(1-\|w\|_\zeta^n\big)^{-\frac{1}{n-1}}$. Consequently, we obtain from \eqref{eq1.1}, H\"{o}lder's inequality, Lemmas \ref{lem2.3}, \ref{lem2.4} and Zhou--Zhou \cite[Theorem 2.4]{Zhou-Zhou-2020} that
\begin{align*}
 \int_{\mathbb{R}^n}\left(g(u_k)u_k\right)^q\,\mathrm{d}x &\lesssim  \|u_k\|^{qn}_{qn}+\int_{\mathbb{R}^n}|u_k|^{q\mu}\Phi\big(\lambda q|u_k|^{\frac{n}{n-1}}\big)\,\mathrm{d}x \\
 &\lesssim \|u_k\|^{qn}_{qn}+\|u_k\|^{q\mu}_{q\mu\sigma^\prime}\bigg(\int_{\mathbb{R}^n}\Phi\big(\lambda \sigma q|u_k|^{\frac{n}{n-1}}\big)\,\mathrm{d}x\bigg)^{\frac{1}{\sigma}}\\
 &\lesssim\|u_k\|_\zeta^{qn}+\|u_k\|_\zeta^{q\mu}\bigg(\int_{\mathbb{R}^n}\Phi\big(\lambda \sigma q \|u_k\|_\zeta^{\frac{n}{n-1}}|w_k|^{\frac{n}{n-1}}\big)\,\mathrm{d}x\bigg)^{\frac{1}{\sigma}}\\
 &\lesssim \|u_k\|_\zeta^{qn}+\|u_k\|_\zeta^{q\mu}\bigg(\int_{\mathbb{R}^n}\Phi\big(\lambda \sigma m|w_k|^{\frac{n}{n-1}}\big)\,\mathrm{d}x\bigg)^{\frac{1}{\sigma}}\\
 &\lesssim \|u_k\|_\zeta^{qn}+\|u_k\|_\zeta^{q\mu}\lesssim 1.
\end{align*}
This shows the validity of \eqref{eq8.19}. Next, we claim that for any $R>0$, there holds
\begin{align}\label{eq8.20}
\lim_{k\to\infty}\int_{\mathcal{W}_R}g(u_k)u_k\,\mathrm{d}x=\int_{\mathcal{W}_R}g(u)u\,\mathrm{d}x.
\end{align}
Indeed, one can obtain from \eqref{eq8.10} that $g(u_k)u_k\to g(u)u$ a.e. in $\mathcal{W}_R$ as $k\to\infty$. Let $q^\prime$ be the H\"{o}lder conjugate exponent of $q$. Then, by using  H\"{o}lder's inequality together with \eqref{eq8.19}, we obtain for every measurable subset ${D}\subset\mathcal{W}_R$ that
\begin{align*}
 \int_{{D}}g(u_k)u_k\mathrm{d}x\leq  \left(\int_{{D}}\left(g(u_k)u_k\right)^q\mathrm{d}x\right)^{\frac{1}{q}}| {D}|^{\frac{1}{q^\prime}}\leq \left(\int_{\mathbb{R}^n}\left(g(u_k)u_k\right)^q\mathrm{d}x\right)^{\frac{1}{q}}|{D}|^{\frac{1}{q^\prime}}\leq C|{D}|^{\frac{1}{q^\prime}},
\end{align*}
where $C>0$ is a constant independent of $k$. Now, let $\delta>0$ be such that $C\delta^{\frac{1}{q^\prime}}<\varepsilon$. It follows that for any measurable subset ${D}\subset\mathcal{W}_R$ with $|{D}|<\delta$, there holds
\begin{align*}
 \int_{{D}}g(u_k)u_k\,\mathrm{d}x<\varepsilon.
\end{align*}
This proves that the sequence $\{g(u_k)u_k\}_k$ is uniformly absolutely integrable on $\mathcal{W}_R$. Hence, by applying Vitali's convergence theorem, we infer that \eqref{eq8.20} holds, and the claim is  established. In addition, an argument analogous to that used in proving \eqref{eq8.15} yields
\begin{align}\label{eq8.21}
  \lim_{k\to\infty}\int_{\mathbb{R}^n\setminus \mathcal{W}_R}g(u_k)u_k\,\mathrm{d}x=\int_{\mathbb{R}^n\setminus \mathcal{W}_R}g(u)u\,\mathrm{d}x.  
\end{align}
Therefore, combining \eqref{eq8.20} and \eqref{eq8.21}, we obtain \eqref{eq8.16}. The proof is thus complete.
\end{proof}
\begin{lemma}\label{lem8.8}
    Let $\{u_k\}_k\subset \mathcal{N}_\zeta$ be a minimizing sequence for $m_\zeta$ for all $0<\zeta<\zeta^\ast$ such that $u_k\rightharpoonup u$ in $W^{1,n}(\mathbb{R}^n)$ as $k\to\infty$. Then, the situation $u= 0$ does not occur.
\end{lemma}
\begin{proof}
By contradiction, we assume that $u=0$. Then, by applying Lemma \ref{lem8.5}, it follows that $\{u_k\}_k$ remains bounded in $W^{1,n}(\mathbb{R}^n)$ and that \eqref{eq8.10} is valid for $u=0$. This yields the following dichotomy: either the sequence $\{u_k\}_k$ is vanishing in the sense that
\begin{align}\label{eq8.22}
    \lim_{k\to\infty}\sup_{y\in\mathbb{R}^n}
    \int_{\mathcal{W}_R(y)}|u_k|^n\,\mathrm{d}x=0
\end{align}
for every $R>0$, or there exist $R,\delta>0$ and a sequence $\{y_k\}\subset\mathbb{Z}^n$ satisfying
\begin{align}\label{eq8.23}
    \lim_{k\to\infty}
    \int_{\mathcal{W}_R(y_k)}|u_k|^n\,\mathrm{d}x\geq\delta,
\end{align}
which means that $\{u_k\}_k$ is nonvanishing. If $\{u_k\}_k$ is vanishing, then by Cheng--Zhang--Zhou \cite[Proposition 6.4]{Cheng-Zhang-Zhou-2023}, we have
$u_k\to 0$ in $L^p(\mathbb{R}^n)$ as $k\to\infty$ for all $p>n$. In virtue of Lemma \ref{lem8.6}, we also have 
\begin{align}\label{eq8.24}
\lim_{k\to\infty}\int_{\mathbb{R}^n}G(u_k)\,\mathrm{d}x=0.
 \end{align}
This together with Lemma \ref{lem8.4} yields
\begin{align*}
    \limsup_{k\to\infty} \|u_k\|_\zeta^n=\limsup_{k\to\infty}\left(n J_\zeta(u_k)+n \int_{\mathbb{R}^n}G(u_k)\,\mathrm{d}x\right)=nm_\zeta<\left(\frac{\lambda_n}{\lambda_0}\right)^{n-1}.
\end{align*}
It follows that $$\limsup_{k\to\infty} \|u_k\|_\zeta^\frac{n}{n-1}<\frac{\lambda_n}{\lambda_0}.$$
Hence, up to a subsequence not relabeled, we may assume
$$\sup_{k\in\mathbb{N}} \|u_k\|_\zeta^\frac{n}{n-1}<\frac{\lambda_n}{\lambda_0}.$$
Let $\sigma,\sigma^\prime>1$ be such that $\frac{1}{\sigma}+\frac{1}{\sigma^\prime}=1$. Next, we fix $m\in\big( \|u_k\|_\zeta^\frac{n}{n-1},\frac{\lambda_n}{\lambda_0}\big)$ and we also choose $\lambda>\lambda_0$ close to $\lambda_0$ and $\sigma>1$ close to $1$ in such a way that $\lambda\sigma m<\lambda_n$. Hence, it follows from \eqref{eq1.1}, H\"{o}lder's inequality, Lemmas \ref{lem2.3}, \ref{lem2.4} and Corollary \ref{cor3.1} that
\begin{align*}
 \int_{\mathbb{R}^n}g(u_k)u_k\,\mathrm{d}x
 &\leq \varepsilon \|u_k\|_n^n+D_\varepsilon\left(\int_{\mathbb{R}^n}|u_k|^{\mu\sigma^\prime}\,\mathrm{d}x\right)^{\frac{1}{\sigma^\prime}}\left(\int_{\mathbb{R}^n}\Phi\left(\lambda\sigma|u_k|^{\frac{n}{n-1}}\right)\,\mathrm{d}x\right)^{\frac{1}{\sigma}}\\
 &\leq \varepsilon S^{-n}_n\|u_k\|^n_\zeta+ D_\varepsilon \|u_k\|^\mu_{\mu \sigma^\prime} \left(\int_{\mathbb{R}^n}\Phi\left(\lambda\sigma \|u_k\|_\zeta^\frac{n}{n-1}\big|u_k/\|u_k\|_\zeta\big|^{\frac{n}{n-1}}\right)\,\mathrm{d}x\right)^{\frac{1}{\sigma}}\\
 &\leq \varepsilon S^{-n}_n\|u_k\|^n_\zeta+ CD_\varepsilon \|u_k\|^\mu_{\mu \sigma^\prime},
 \end{align*}
 where
\begin{align*}
     C:=\left(\int_{\mathbb{R}^n}\Phi\left(\lambda\sigma \|u_k\|_\zeta^\frac{n}{n-1}\big|u_k/\|u_k\|_\zeta\big|^{\frac{n}{n-1}}\right)\,\mathrm{d}x\right)^{\frac{1}{\sigma}}<+\infty.
 \end{align*}
By using the fact that $\|u_k\|_{\mu \sigma^\prime}=o_k(1)$ as $k\to\infty$, together with the boundedness of the sequence $\{u_k\}_k$ in $W^{1,n}(\mathbb{R}^n)$ and the arbitrariness of $\varepsilon>0$, we conclude that
\begin{align}\label{eq8.25}
 \lim_{k\to\infty}\int_{\mathbb{R}^n}g(u_k)u_k\,\mathrm{d}x=0.    
\end{align}
Consequently, we obtain from \eqref{eq8.17}, \eqref{eq8.24} and \eqref{eq8.25} that
$$ 0<m_\zeta=\frac{1}{n}\lim_{k\to\infty}\|u_k\|_\zeta^n=\frac{1}{n}\lim_{k\to\infty}\int_{\mathbb{R}^n}g(u_k)u_k\,\mathrm{d}x=0, $$
which contradicts to the fact that $m_\zeta>0$ for all $\zeta>0$. This shows that \eqref{eq8.22} is not satisfied and thus, \eqref{eq8.23} holds. For each $k\in\mathbb{N}$, we now define $w_k(x)=u_k(x+y_k)$ for all $x\in\mathbb{R}^n$. Hence, we deduce from \eqref{eq8.23} that 
 \begin{align}\label{eq8.26}
    \lim_{k\to\infty}
    \int_{\mathcal{W}_R}|w_k|^n\,\mathrm{d}x\geq\delta.
\end{align}
Notice that both $J_\zeta$ and $J_\zeta^\prime$ are invariant under $\mathbb{Z}^n$ translations. This infers that $\{w_k\}_k\subset \mathcal{N}_\zeta$ is also a minimizing sequence for $m_\zeta$ such that $w_k\rightharpoonup w$ in $W^{1,n}(\mathbb{R}^n)$ as $k\to\infty$. Moreover, it follows from \eqref{eq8.10} and \eqref{eq8.26} that
$$\int_{\mathcal{W}_R}|w|^n\,\mathrm{d}x\geq\delta. $$
This shows that $w\neq0$, and the proof is now complete.
\end{proof}
We now establish the attainability of the minimization level $m_\zeta$ for constant potential $\zeta<\zeta^\ast$, thereby obtaining ground state solutions to \eqref{main problem}.
\subsection{{Proof of the First Part of Theorem \ref{thm1.13}}}
Let $\zeta\in(0,\zeta^\ast)$ and consider a minimizing sequence $\{u_k\}_k\subset \mathcal{N}_\zeta$ for $m_\zeta$ such that $u_k\rightharpoonup u$ in $W^{1,n}(\mathbb{R}^n)$ as $k\to\infty$. In virtue of Lemma \ref{lem8.8}, one has $u\neq 0$. Next, our aim is to show
\begin{align}\label{eq8.27}
    \|u\|^n_\zeta\leq \int_{\mathbb{R}^n}g(u)u\,\mathrm{d}x.
\end{align}
Indeed, if not, let us assume
$$  \|u\|^n_\zeta> \int_{\mathbb{R}^n}g(u)u\,\mathrm{d}x.$$
It follows from Lemma \ref{lem8.7} that \eqref{eq8.16} and \eqref{eq8.17} hold. Hence, we obtain
\begin{align*}
 \|u\|^n_\zeta\leq \liminf_{k\to\infty} \|u_k\|^n_\zeta= \liminf_{k\to\infty}  \int_{\mathbb{R}^n}g(u_k)u_k\,\mathrm{d}x=\int_{\mathbb{R}^n}g(u)u\,\mathrm{d}x<\|u\|^n_\zeta,
\end{align*}
which is obviously a contradiction. This infers that \eqref{eq8.27} holds. Define the map $\Theta\colon (0,+\infty)\to\mathbb{R}$ by
$$ \Theta(s):=\|su\|_\zeta^n-\int_{\mathbb{R}^n}g(su) su\,\mathrm{d}x\quad\text{for all}\quad s>0. $$
Due to \eqref{eq8.27}, we deduce that $\Theta(1)\leq 0$. Now, by arguing similar to Lemma \ref{lem8.2}, one can prove that $\Theta(s)>0$ for sufficiently small $s>0$. It follows from the intermediate value theorem that $\Theta(s_0)=0$ for some $s_0\in(0,1]$, which in turn implies that $s_0u\in\mathcal{N}_\zeta$. In view of \eqref{g5}, the map $s\mapsto g(s)s-nG(s)$ is monotone increasing on $(0,+\infty)$. Consequently, we obtain from \eqref{eq8.10} and Fatou's lemma that
\begin{align*}
    m_\zeta\leq J_\zeta(s_0 u)&=\frac{1}{n}\int_{\mathbb{R}^n}\left(g(s_0 u)s_0u-nG(s_0u)\right)\,\mathrm{dx}\leq \frac{1}{n}\int_{\mathbb{R}^n}\left(g(u)u-nG(u)\right)\,\mathrm{dx}\\
    &\leq \frac{1}{n}\liminf_{k\to\infty}\int_{\mathbb{R}^n}\left(g(u_k)u_k-nG(u_k)\right)\,\mathrm{dx}=\liminf_{k\to\infty}J_\zeta(u_k)=m_\zeta.
\end{align*}
This proves $s_0=1$ and that $u\in\mathcal{N}_\zeta$ satisfying $J_\zeta(u)=m_\zeta$. Hence, we deduce that $u$ minimizes $J_\zeta$ on $\mathcal{N}_\zeta$. It follows from Lemma \ref{lem8.333} that $J_\zeta^\prime(u)=0$, so $u$ is a nontrivial critical point of $J_\zeta$. Equivalently, there holds
\begin{align*}
  Q_n u=\zeta|u|^{n-2}u-g(u)\quad\text{in}\quad \mathbb{R}^n. 
\end{align*}
By applying Lemma \ref{lem8.1},  we have $u\geq0$ in $\mathbb{R}^n$. Next, we claim that $u>0$ in $\mathbb{R}^n$. Indeed, if not, there exist $R>0$ and $x_1,\,x_2\in B_R(y)$ for some $y\in\mathbb{R}^n$ such that
\begin{align*}
    u(x)\geq 0\quad\text{for all}\quad x\in B_R(y),\quad u(x_1)= 0 \quad\text{and}\quad u(x_2)>0,
\end{align*}
where $B_R(y)$ denotes the notation for balls used in  Huynh--Le \cite[Lemma $8$]{Huynh-Le-2026}. By the local regularity theory for quasilinear elliptic equations developed by Serrin \cite{Serrin-1964} and Trudinger \cite{Trudinger-1967@}, together with its anisotropic version established by Huynh--Le \cite[Lemma $8\textnormal{(i)}$]{Huynh-Le-2026}, $u$ is locally bounded and H\"{o}lder continuous in $\mathbb{R}^n$. Define a function $H\colon\mathbb{R}\to \mathbb{R}$ by
\begin{align*}
     H(s):=\begin{cases}\cfrac{g(s)}{s^{n-1}}\quad\text{if}\quad s>0,\\
     0\quad\qquad\text{if}\quad s\leq 0.
       \end{cases}
\end{align*}
In view of \eqref{g1}, $H(s)\to 0$ as $s\to 0$ and thus, $H$ is continuous at $s=0$. Since $g\in C^1(\mathbb{R},\mathbb{R})$, it follows that $H\in C(\mathbb{R},\mathbb{R})$. Recalling $u$ is locally bounded in $\mathbb{R}^n$, we obtain
\begin{align*}
     d(x):=\zeta-H(u(x))\in L^\infty _{\textit{loc}}(\mathbb{R}^n).
\end{align*}
In particular, one can write $Q_n u=d(x)u^{n-1}$ in $\mathbb{R}^n$. Fix $R>0$. Restricting the equation $Q_n u=d(x)u^{n-1}$ to $B_{2R}(y)$, we deduce that
\begin{align*}
  Q_n u=d(x)u^{n-1}\quad\text{in}\quad B_{2R}(y). 
\end{align*}
Since $d\in L^\infty(B_{2R}(y))$, we have  $d\in L^s(B_{2R}(y))$ for every $s>1$. Indeed, there holds
\begin{align*}
 \|d\|_{L^s(B_{2R}(y))}
\leq
|B_{2R}(y)|^{1/s}
\|d\|_{L^\infty(B_{2R}(y))}
<+\infty.   
\end{align*}
Now, by using the local anisotropic Harnack inequality established by Huynh--Le \cite[Lemma $8\textnormal{(ii)}$]{Huynh-Le-2026}, which extends the classical result of Trudinger \cite{Trudinger-1967@}, and taking $p=n$, there exists a constant $C>0$, depending on $R^{n(1-\frac{1}{s})}\|d\|_{L^s(B_{2R}(y))}$, such that
\begin{align*}
    0\leq \sup_{B_R(y)}u(x)\leq
C\inf_{B_R(y)}u(x)=0.
\end{align*}
Thus, we infer that $u= 0$ in $B_R(y)$, which contradicts $u(x_2)>0$. This shows that $u>0$ in $\mathbb{R}^n$ and hence, $u$ is a positive ground state solution of \eqref{main problem} for any $\zeta\in(0,\zeta^\ast)$. This completes the proof.\qed

The next result plays a crucial role in completing the proof of Theorem \ref{thm1.13}.
\begin{lemma}\label{lem8.9}
 The inequality $0<m_\zeta<\frac{1}{n}\big(\frac{\lambda_n}{\lambda_0}\big)^{n-1}$
occurs if and only if $\zeta\in(0,\zeta^\ast)$.
\end{lemma}
\begin{proof}
 In view of Lemma \ref{lem8.4}, we only need to prove $m_\zeta<\frac{1}{n}(\frac{\lambda_n}{\lambda_0})^{n-1}$ implies $\zeta<\zeta^\ast$. In fact, if $\zeta^\ast=+\infty$, then  $\zeta<\zeta^\ast$, and the result follows. This allows us to consider $\zeta^\ast<+\infty$. The proof of the first part of Theorem \ref{thm1.13} further yields that, if $m_\zeta<\frac{1}{n}(\frac{\lambda_n}{\lambda_0})^{n-1}$, then there exists some $u\in \mathcal{N}_\zeta$ attaining the value $m_\zeta$, and such a function $u$ is a positive ground state solution of \eqref{main problem}. Define the scaling $u_t\colon\mathbb{R}^n\to\mathbb{R}$ by $u_t(x):=u(tx)$ for all $x\in \mathbb{R}^n$ and $t>0$. By the change of variables $y=tx$, we have
\begin{align*}
 \int_{\mathbb{R}^n}F(\nabla u_t)^n\,\mathrm{d}x=\int_{\mathbb{R}^n}F(\nabla u)^n\,\mathrm{d}y,\quad\int_{\mathbb{R}^n}|u_t|^n\,\mathrm{d}x=t^{-n}\int_{\mathbb{R}^n}|u|^n\,\mathrm{d}y 
\end{align*}
and 
\begin{align*}
 \int_{\mathbb{R}^n}G(u_t)\,\mathrm{d}x=t^{-n}\int_{\mathbb{R}^n}G(u)\,\mathrm{d}y.      
\end{align*}
It follows from \eqref{eq8.1} that
\begin{align*}
     J_\zeta(u_t)=\frac{1}{n}\left(\int_{\mathbb{R}^n}F(\nabla u)^n\,\mathrm{d}x+\zeta t^{-n}\int_{\mathbb{R}^n}|u|^n\,\mathrm{d}x\right)-t^{-n}\int_{\mathbb{R}^n}G(u)\,\mathrm{d}x.
\end{align*}
Differentiating with respect to $t$ yields
\begin{align*}
    \frac{\mathrm{d}}{\mathrm{d}t}J_\zeta(u_t)=-\zeta t^{-n-1}\int_{\mathbb{R}^n}|u|^n\,\mathrm{d}x+nt^{-n-1}\int_{\mathbb{R}^n}G(u)\,\mathrm{d}x.
\end{align*}
Since $u \in \mathcal{N}_\zeta$, it follows that $u$ is a nontrivial critical point of $J_\zeta$. This shows that $t=1$ is a critical point of the map $t\mapsto J_\zeta(u_t)$ and  we have
\begin{align*}
  \frac{\mathrm{d}}{\mathrm{d}t}J_\zeta(u_t)\bigg|_{t=1} =0. 
\end{align*}
In particular, one has
\begin{align*}
    \zeta\int_{\mathbb{R}^n}|u|^n\,\mathrm{d}x-n\int_{\mathbb{R}^n}G(u)\,\mathrm{d}x=0.
\end{align*}
Hence, we obtain that $u\in \mathcal{P}_\zeta$ and there holds
$$M_\zeta\leq J_\zeta(u)=m_\zeta.$$
The above inequality, combined with Case-$1$ of Lemma \ref{lem8.4}, leads to $m_\zeta = M_\zeta$, and thus $M_\zeta$ is attained by the same function $u$. It follows from the definitions of $M_\zeta$ and $\mathcal{P}_\zeta$ that
\begin{align*}
    \|F(\nabla u)\|_n^n<\left(\frac{\lambda_n}{\lambda_0}\right)^{n-1}\qquad\text{and}\qquad \zeta=\frac{n}{\|u\|_n^n}\int_{\mathbb{R}^n}G(u)\,\mathrm{d}x.
\end{align*}
Define the map $\Pi\colon (0,+\infty)\to\mathbb{R}$ by
\begin{align*}
    \Pi(t):=\frac{n}{t^n\|u\|_n^n}\int_{\mathbb{R}^n}G(tu)\,\mathrm{d}x\quad\text{for all} \quad t>0.
\end{align*}
From the above, it is clear that $\Pi(1)=\zeta$. In addition, we deduce from \eqref{g5} that the map $t\mapsto\Pi(t)$ is strictly increasing on $(0,+\infty)$. We now denote
\begin{align*}
    w:=\frac{R(G)^{\frac{1}{n}}}{\|F(\nabla u)\|_n}\, u.
\end{align*}
This yields that $\|F(\nabla w)\|^n_n=R(G)$. Moreover, the following estimate holds
\begin{align*}
    \zeta^\ast\geq \frac{n}{\|w\|_n^n}\int_{\mathbb{R}^n}G(w)\,\mathrm{d}x=\Pi\left(\frac{R(G)^{\frac{1}{n}}}{\|F(\nabla u)\|_n}\right)>\Pi(1)=\zeta.
\end{align*}
This completes the proof.
\end{proof}
Next, we establish that \eqref{main problem} fails to admit any ground state solution corresponding to the constant potential $\zeta>\zeta^\ast$.
\subsection{{Proof of the Second Part of Theorem \ref{thm1.13}}}
The proof is based on method of contradiction. Indeed, let there exists $\zeta_0>\zeta^\ast$ such that \eqref{main problem} has a ground state solution $u_0\in W^{1,n}(\mathbb{R}^n)\setminus\{0\}$. In virtue of Lemmas \ref{lem8.3} and \ref{lem8.9}, it follows immediately that $m_{\zeta_0}=\frac{1}{n}\big(\frac{\lambda_n}{\lambda_0}\big)^{n-1}$. Note that $m_{\zeta_0}$ is achieved by the function $u_0$. Moreover, we also infer that for all $\zeta\in(\zeta^\ast,\,\zeta_0)$, there holds $m_\zeta<m_{\zeta_0}$. This contradicts the assertion of Lemma \ref{lem8.9}. Hence, for every $\zeta>\zeta^\ast$, there exists no ground state solution of \eqref{main problem}, and the proof is complete.\qed

Finally, we conclude this section by proving sufficient asymptotic conditions for the finiteness or infiniteness of $\zeta^\ast$.
\subsection{{Proof of the Third Part of Theorem \ref{thm1.13}}} Let \eqref{eq1.18} be hold. Next, we claim that $\zeta^\ast=+\infty$. Indeed, by choosing the same Moser-type sequence of functions $\{v_k\}_k$ as defined in \eqref{eq3.4}, we have 
$$\|F(\nabla v_k)\|^n_n=1 \qquad \text{and}\qquad \|v_k\|_n^n=O\Big(\frac{1}{\ln k}\Big)\quad\text{as}\quad k\to\infty.$$
For each $k\in\mathbb{N}$, we define
\begin{align*}
    u_k(x):=\left(\frac{\lambda_n}{\lambda_0}\right)^{\frac{n-1}{n}}v_k(x)\quad\text{for all}\quad x\in\mathbb{R}^n.
\end{align*}
It follows from direct calculations that
\begin{align*}
   \|F(\nabla u_k)\|^n_n=\left(\frac{\lambda_n}{\lambda_0}\right)^{n-1}\qquad\text{and}\qquad \|u_k\|_n^n=O\Big(\frac{1}{\ln k}\Big)\quad\text{as}\quad k\to\infty.
\end{align*}
Hence, we infer that $\{u_k\}_k$ is an admissible sequence for $\zeta^\ast$. Denote the sequence $\{a_k\}_k$ by
\begin{align*}
    a_k:=\left(\frac{\lambda_n}{\lambda_0}\right)^{\frac{n-1}{n}}(n\kappa_n)^{-\frac{1}{n}}(\ln k)^{\frac{n-1}{n}}\quad\text{for all}\quad k\in\mathbb{N}.
\end{align*}
Set
\begin{align*}
    L(t):=\frac{t^{\frac{n}{n-1}}G(t)}{\operatorname{\exp}(\lambda_0 t^{\frac{n}{n-1}})}\quad\text{for all}\quad t\geq 0.
\end{align*}
In virtue of \eqref{eq1.18} and the fact that $a_k\to+\infty$ as $k\to\infty$, one has $L(a_k)\to+\infty$ as $k\to\infty$. Now, by using all the above information, we get
\begin{align*}
\int_{\mathbb{R}^n}G(u_k)\,\mathrm{d}x&\geq \int_{\{x\in\mathbb{R}^n\colon~F^0(x)\leq sk^{-1} \}} G(u_k)\,\mathrm{d}x=\int_{0}^{sk^{-1}} G(a_k)\int_{\partial\mathcal{W}_\eta}\frac{1}{|\nabla F^0(x)|}\,\mathrm{d}\sigma\,\mathrm{d}\eta\\
&\gtrsim  G(a_k)\int_{0}^{sk^{-1}}\eta^{n-1}\,\mathrm{d}\eta\gtrsim G(a_k)k^{-n}\gtrsim\frac{L(a_k)}{\ln k}.
\end{align*}
This shows that
\begin{align*}
    \zeta^\ast\geq \frac{n}{\|u_k\|_n^n}\int_{\mathbb{R}^n}G(u_k)\,\mathrm{d}x\gtrsim L(a_k)\to +\infty\quad\text{as}\quad k\to\infty.
\end{align*}
Thus, the claim is established. To complete the proof, we only need to verify that $\zeta^\ast<+\infty$ whenever \eqref{eq1.19} holds. To this end, we assume that \eqref{eq1.19} holds. It follows that there exist $s_0>0$ and $C>0$ such that
\begin{align*}
    G(s)\leq \frac{C\operatorname{\exp}(\lambda_0 s^{\frac{n}{n-1}})}{s^{\frac{n}{n-1}}}\quad\text{for all}\quad s\geq s_0.
\end{align*}
It is well-known that $$\frac{\Phi(s)}{\operatorname{\exp}(s)}\to 1\quad\text{as}\quad s\to\infty.$$ As a result, there exists $\tilde{s}_0>0$ such that $\operatorname{\exp}(s)\leq2 \Phi(s)$ for all $s\geq \tilde{s}_0$. Now, enlarging $s_0$, if necessary, so that $\lambda_0s_0^{\frac{n}{n-1}}\geq \tilde{s}_0$, we consequently obtain
\begin{align*}
\operatorname{\exp}(s)\leq2 \Phi(s) \quad\text{for all}\quad s\geq \lambda_0s_0^{\frac{n}{n-1}}.  
\end{align*}
Moreover, the condition $s\geq s_0$ yields
\begin{align*}
    \frac{1}{s^\frac{n}{n-1}}\leq \left(1+\frac{1}{s_0}\right)^\frac{n}{n-1}\frac{1}{(1+s)^\frac{n}{n-1}}.
\end{align*}
The preceding estimates together imply that
\begin{align*}
    G(s)\leq2C\left(1+\frac{1}{s_0}\right)^\frac{n}{n-1} \frac{\Phi(\lambda_0 s^{\frac{n}{n-1}})}{(1+s)^{\frac{n}{n-1}}}\quad\text{for all}\quad s\geq s_0. 
\end{align*}
Next, we consider an arbitrary admissible function $ u\in W^{1,n}(\mathbb{R}^n)\setminus\{0\}$ for $\zeta^\ast$. This implies that $ \|F(\nabla u)\|^n_n\leq(\frac{\lambda_n}{\lambda_0})^{n-1}$. Define the auxiliary function $w$ by
\begin{align*}
    w(x):=\left(\frac{\lambda_0}{\lambda_n}\right)^{\frac{n-1}{n}} u(x)\quad\text{for all}\quad x\in\mathbb{R}^n.
\end{align*}
By the construction, we have 
\begin{align*}
\|F(\nabla w)\|_n\leq 1,\quad \|w\|_n=\left(\frac{\lambda_0}{\lambda_n}\right)^{\frac{n-1}{n}} \|u\|_n\quad\text{and}\quad\frac{1}{(1+|u|)^{\frac{n}{n-1}}}\lesssim \frac{1}{(1+|w|)^{\frac{n}{n-1}}}.  
\end{align*}
Recall that $G(s)=0$ for all $s\leq 0$. In addition, by using the preceding estimates, Theorem \ref{thm1.3} yields that
\begin{equation}\label{eq8.28}
\begin{aligned}
    \int_{\{x\in\mathbb{R}^n:\, |u|>s_0\}}G (u)\,\mathrm{d}x&=\int_{\{x\in\mathbb{R}^n:\, u>s_0\}}G (u)\,\mathrm{d}x\\&\lesssim \int_{\{x\in\mathbb{R}^n:\, u>s_0\}}\frac{\Phi(\lambda_0 u^{\frac{n}{n-1}})}{(1+u)^{\frac{n}{n-1}}}\,\mathrm{d}x
    \lesssim\int_{\mathbb{R}^n}\frac{\Phi(\lambda_0 |u|^{\frac{n}{n-1}})}{(1+|u|)^{\frac{n}{n-1}}}\,\mathrm{d}x\\
    &\lesssim \int_{\mathbb{R}^n}\frac{\Phi(\lambda_n |w|^{\frac{n}{n-1}})}{(1+|w|)^{\frac{n}{n-1}}}\,\mathrm{d}x\lesssim \|w\|_n^n\lesssim \|u\|_n^n.
\end{aligned}
\end{equation}
In view of \eqref{g1}, we also have $G(s)=o(s^{n})$ as $s\to 0$. Thus, for every $\varepsilon>0$, there exists $\delta>0$ such that $G(s)\leq \varepsilon s^n$ for all $0\leq s\leq\delta$. Due to the continuity of $G$, one sees that the function $s\mapsto\frac{G(s)}{s^n}$ is bounded on $[\delta, s_0]$. This concludes that there exists a constant $C>0$ such that
\begin{align*}
    G(s)\leq C s^n\quad\text{for all}\quad 0\leq s\leq s_0.
\end{align*}
Hence, we obtain
\begin{align*}
    G(s)\leq C |s|^n\quad\text{whenever}\quad |s|\leq s_0,
\end{align*}
since $G(s)=0$ for all $s\leq 0$. Consequently, it follows that
\begin{equation}\label{eq8.29}
\begin{aligned}
    \int_{\{x\in\mathbb{R}^n:\, |u|\leq s_0\}}G (u)\,\mathrm{d}x\lesssim\int_{\{x\in\mathbb{R}^n:\, |u|\leq s_0\}}|u|^n\,\mathrm{d}x\lesssim \|u\|_n^n.
\end{aligned}
\end{equation}
Now, by combining \eqref{eq8.28} and \eqref{eq8.29}, there holds
\begin{align*}
    \zeta^\ast=\sup_{\substack{u\in W^{1,n}(\mathbb{R}^n)\setminus\{0\},\\ \|F(\nabla u)\|^n_n\leq(\frac{\lambda_n}{\lambda_0})^{n-1}}}\frac{n}{\|u\|_n^n}\int_{\mathbb{R}^n}G(u)\,\mathrm{d}x<+\infty.
\end{align*}
This finishes the proof.
\qed

\section*{Acknowledgements}
Giovany M. Figueiredo has been partially supported by FAPESP 2023/10287-7, CNPq, and FAPDF. Deepak Kumar Mahanta gratefully acknowledges the financial support provided during his tenure as a Project Scientist at the Department of Mathematics, Indian Institute of Technology Kanpur, through the ANRF PMECRG Grant (Project No. ANRF/ECRG/2024/004743/PMS) of Dr. Indranil Chowdhury, funded by the Government of India. 

\end{document}